\documentclass[11pt]{article}
\usepackage{amsfonts}
\usepackage{amsmath, amssymb, amsthm}
\usepackage{enumerate}
\usepackage{epsfig}
\usepackage{verbatim}

\newtheorem{thm}{Theorem}[section]
\newtheorem{lem}[thm]{Lemma}

\newtheorem{cor}[thm]{Corollary}
\newtheorem{defn}[thm]{Definition}
\newtheorem{remark}{Remark}
\theoremstyle{remark}

\providecommand{\C}{\ensuremath{\mathbb{C}}}
\providecommand{\R}{\ensuremath{\mathbb{R}}}
\providecommand{\N}{\ensuremath{\mathbb{N}}}

\providecommand{\decay}[1]{\operatorname*{Exp}(#1)}
\providecommand{\ua}{\underline{a}}
\providecommand{\oa}{\overline{a}}

\providecommand{\BLinf}{\ensuremath{{\mathcal B\mathcal L}^\infty}}
\providecommand{\BLtwo}{\ensuremath{{\mathcal B\mathcal L}^2}}
\DeclareMathOperator*{\dist}{dist}

\newenvironment{numberedproof}[1]{{\bf Proof of #1:}}{{}\hfill{\hbox{$\Box$}}\par\bigskip}
\numberwithin{equation}{section}

\newif\iftechreport
\techreporttrue    
\iftechreport
\else
\fi

\begin{document}

\iftechreport
\title{
Analytic regularity for a singularly perturbed system of reaction-diffusion
equations with multiple scales: proofs}
\else 
\title{
Analytic regularity for a singularly perturbed system of reaction-diffusion
equations with multiple scales}
\fi 
\author{J. M. Melenk \\
Institut f\"{u}r Analysis und Scientific Computing\\
Vienna University of Technology\\
Wiedner Hauptstrasse 8-10, A-1040 Wien \\
AUSTRIA \vspace{0.125cm} \\
C. Xenophontos\\
Department of Mathematics and Statistics\\
University of Cyprus \\
P.O. BOX 20537\\
Nicosia 1678\\
CYPRUS\\
and\\
L. Oberbroeckling\\
Department of Mathematics and Statistics\\
Loyola University Maryland\\
4501 N. Charles Street\\
Baltimore, MD 21210\\
USA}
\maketitle

\begin{abstract}
We consider a coupled system of two singularly perturbed reaction-diffusion
equations, with two small parameters $0< \varepsilon \le \mu \le 1$, each
multiplying the highest derivative in the equations. The presence of these
parameters causes the solution(s) to have \emph{boundary layers} which
overlap and interact, based on the relative size of $\varepsilon$ and $%
\mu$. We construct full asymptotic expansions together with error bounds 
that cover the complete range $0 < \varepsilon \leq \mu \leq 1$. 
For the present case of analytic input data, we derive 
derivative growth estimates for the terms of the asymptotic expansion that
are explicit in the perturbation parameters and the expansion order. 
\end{abstract}

\section{Introduction}
\label{intro} 

Singularly perturbed (SP) boundary value problems (BVPs),  and their numerical approximation, have received a lot of attention in the last few decades (see, e.g., the classical texts \cite{omalley91,wasow65} on asymptotic analysis and 
the books 
\cite{mos}, \cite{morton}, \cite{rst}, whose focus is more on numerical methods for this problem class). One common feature that these problems share is the presence of \emph{boundary layers} in the solution. 
In order for a numerical method, designed for the approximation of the solution to SP BVPs, to be considered \emph{robust} it must be able to perform well independently of the singular perturbation parameter(s).  To achieve this, information about the regularity of the exact solution is utilized, and in particular, bounds on the derivatives. Such information is available in the literature for scalar SP BPVs of reaction- and convection-diffusion type in one- and two-dimensions 
(see, e.g., \cite{m,melenk-schwab99b} for scalar versions of the problem studied in the present article). 
For \emph{systems} of SP BVPs, the bibliography is scarce, even in one-dimension; see the relatively recent review article \cite{linss-stynes} and the references therein, for such results available to date. 
It is, therefore, the purpose of this article to begin filling this void; in particular, we provide the regularity theory for a system of two coupled SP linear reaction-diffusion equations, with two singular perturbation parameters. Our analysis
is complete for the problem under consideration in that we derive full asymptotic expansions for all relevant cases 
of singular perturbation parameters and give explicit control of all derivatives of all terms appearing in the expansions. 
Even though this is a linear, one-dimensional problem, the methodology presented here can be the starting point for 
treating more difficult problems. 

The regularity results obtained here are used in the companion communication \cite{mxo} 
to prove, for the first time, exponential convergence of the $hp$-FEM for problems with multiple singular
perturbation parameters. This exponential convergence result for the $hp$-FEM relies on mesh design principles 
firmly established for problems with a single singular perturbation parameter as 
discussed in \cite{ss,ssx,melenk-schwab99b,m,mB}; the mathematical analysis of \cite{mxo} shows that these
mesh design principles extend to problems with multiple singular perturbation parameters and confirms the 
numerical results of \cite{xo} for the problem class under consideration here. 

The rest of the paper is organized as follows: In Section~\ref{model} we
present the model problem and discuss the typical phenomena. In Sections~\ref{case_1}--\ref{case_3} we address the regularity of the solution, as it depends on the relationship between the singular perturbation parameters. 
\iftechreport
The proofs of most of the results presented in these sections are technical and are relegated to Appendices~\ref{appendix-misc}--\ref{appendix:case_4}. 
\else
The proofs of most of the results presented in these sections are technical and are relegated to 
\cite[Appendices~\ref{appendix-misc}--\ref{appendix:case_4}]{mxo-1}.
\else
\fi

In what follows, the space
of square integrable functions on an interval $I\subset \mathbb{R}$ will be
denoted by $L^{2}\left( I\right) $, with associated inner product
\begin{equation*}
\langle u,v\rangle_{I}:=\int_{I}u(x)v(x)dx.
\end{equation*}
We will also utilize the usual Sobolev space notation $H^{k}\left( I\right)$
to denote the space of functions on $I$ with $0,1,2,...,k$ generalized
derivatives in $L^{2}\left( I\right) $, equipped with norm and seminorm 
$\left\Vert \cdot \right\Vert _{k,I}$ and $\left\vert \cdot \right\vert
_{k,I}\,$, respectively. For vector functions 
$\mathbf{U}:=(u_{1}(x),u_{2}(x))^{T}$, we will write
\begin{equation*}
\left\Vert \mathbf{U}\right\Vert _{k,I}^{2}=\left\Vert u_{1}\right\Vert
_{k,I}^{2}+\left\Vert u_{2}\right\Vert _{k,I}^{2}.
\end{equation*}
We will also use the space
\begin{equation*}
H_{0}^{1}\left( I\right) =\left\{ u\in H^{1}\left( I\right) :\left.
u\right\vert _{\partial I}=0\right\} ,
\end{equation*}
where $\partial I$ denotes the boundary of $I$. For $z \in \C$, we use 
$B_r(z)$
to denote the ball of radius $r$ centered at $z$.
Finally, the letter $C$ will
be used to denote a generic positive constant, independent of any
discretization or singular perturbation parameters and possibly having
different values in each occurrence.


\section{The Model Problem and Main Results}

\label{model}

We consider the following model problem: Find a pair of functions $(u,v)^T$ such that
\begin{subequations}
\label{eq:model-problem}
\begin{equation}
\left\{ 
\begin{array}{c}
-\varepsilon ^{2}u^{\prime \prime }(x)+a_{11}(x)u(x)+a_{12}(x)v(x)=f(x)\text{
in }I=(0,1), \\ 
-\mu ^{2}v^{\prime \prime }(x)+a_{21}(x)u(x)+a_{22}(x)v(x)=g(x)\text{ in }%
I=(0,1),%
\end{array}
\right.  \label{1}
\end{equation}
along with the boundary conditions
\begin{equation}
u(0)=u(1)=0\;,\;v(0)=v(1)=0.  \label{2}
\end{equation}
\end{subequations}
With the abbreviations
$$
{\mathbf U} = \left(\begin{array}{c} u \\ v \end{array}\right), 
\qquad 
{\mathbf E}^{\varepsilon,\mu} := \left(\begin{array}{cc} \varepsilon^2 & 0 \\ 0 & \mu^2\end{array}\right), 
\qquad 
{\mathbf A}(x) := \left(\begin{array}{cc} a_{11}(x) & a_{12}(x) \\ a_{21}(x) & a_{22}(x)\end{array}\right), 
\qquad 
{\mathbf F} = \left(\begin{array}{c} f \\ g \end{array}\right), 
$$
equations (\ref{1})--(\ref{2}) may also be written in the following, more compact form: 
\begin{equation} 
\label{eq:model-problem-short}
L_{\varepsilon,\mu} {\mathbf U}:= 
-{\mathbf E}^{\varepsilon,\mu} {\mathbf U}^{\prime\prime}(x) + {\mathbf A}(x) {\mathbf U} = {\mathbf F}, 
\qquad {\mathbf U}(0) = {\mathbf U}(1) =0. 
\end{equation} 
The parameters $0<\varepsilon \leq \mu \leq 1$ are given, as are the
functions $f$, $g$, and $a_{ij}$, $i,j\in \{1,2\},$ which are assumed to
be analytic on $\overline{I}=[0,1]$. \ Moreover we assume that there exist
constants $C_{f},\gamma _{f},C_{g},\gamma _{g},C_{a},\gamma _{a}>0$ such
that 
\begin{equation}
\left\{ 
\begin{array}{c}
\left\Vert f^{(n)}\right\Vert _{L^\infty(I)}\leq C_{f}\gamma _{f}^{n}n!\quad
\forall \;n\in \mathbb{N}_{0}, \\ 
\left\Vert g^{(n)}\right\Vert _{L^\infty(I)}\leq C_{g}\gamma _{g}^{n}n!\quad
\forall \;n\in \mathbb{N}_{0}, \\ 
\left\Vert a_{ij}^{(n)}\right\Vert _{L^\infty(I)}\leq C_{a}\gamma
_{a}^{n}n!\quad \forall \;n\in \mathbb{N}_{0},i,j\in \{1,2\}%
\end{array}%
\right. .  \label{3}
\end{equation}
The variational formulation of (\ref{1})--(\ref{2}) reads: Find 
$\mathbf{U}:=(u,v)^T\in \left[ H_{0}^{1}\left( I \right) \right] ^{2}$ 
such that 
\begin{equation}
B\left( U,V\right) =F\left( V\right) \;\;\forall \;\mathbf{V}:=(\overline{u},%
\overline{v})\in \left[ H_{0}^{1}\left( I \right) \right] ^{2},
\label{4}
\end{equation}
where, with $\left\langle \cdot ,\cdot \right\rangle _{I}$ the usual 
$L^{2}(I)$ inner product, 
\begin{eqnarray}
B\left( \mathbf{U},\mathbf{V}\right) &=&\varepsilon ^{2}\left\langle
u^{\prime },\overline{u}^{\prime }\right\rangle _{I}+\mu ^{2}\left\langle
v^{\prime },\overline{v}^{\prime }\right\rangle _{I}+\left\langle
a_{11}u+a_{12}v,\overline{u}\right\rangle _{I}+\left\langle a_{21}u+a_{22}v,%
\overline{v}\right\rangle _{I},  \label{5} \\
F\left( \mathbf{V}\right) &=&\left\langle f,\overline{u}\right\rangle
_{I}+\left\langle g,\overline{v}\right\rangle _{I}.  \label{6}
\end{eqnarray}
%
%
The matrix-valued function ${\mathbf A}$ is assumed to be pointwise positive definite, i.e., 
for some fixed $\alpha > 0$ 
\begin{equation}
\label{eq:A-positive-definite}
\overrightarrow{\xi }^{T}\mathbf{A}\overrightarrow{\xi }\geq \alpha ^{2}%
\overrightarrow{\xi }^{T}\overrightarrow{\xi }\;\;\;\;\forall \;%
\overrightarrow{\xi }\in \mathbb{R}^{2} 
\qquad \forall x \in \overline{I}. 
\end{equation}
It follows that the bilinear form $B\left( \cdot ,\cdot \right) $ given
by (\ref{5}) is coercive with respect to the \emph{energy norm} 
\begin{equation}
\left\Vert \mathbf{U}\right\Vert _{E,I}^{2}\equiv \left\Vert
(u,v)\right\Vert _{E,I}^{2}:=\varepsilon ^{2}\left\vert u\right\vert
_{1,I}^{2}+\mu ^{2}\left\vert v\right\vert _{1,I}^{2}+\alpha ^{2}\left(
\left\Vert u\right\Vert _{0,I}^{2}+\left\Vert v\right\Vert _{0,I}^{2}\right)
,  \label{8}
\end{equation}
i.e., 
\begin{equation*}
B\left( \mathbf{U},\mathbf{U}\right) \geq \left\Vert \mathbf{U}\right\Vert
_{E,I }^{2}\;\;\forall \;\mathbf{U}\in \left[ H_{0}^{1}\left( I
\right) \right] ^{2}.
\end{equation*}
This, along with the continuity of $B\left( \cdot ,\cdot \right) $ and 
$F\left( \cdot \right) \,,$ imply the unique solvability of (\ref{4}). We
also have, by the Lax-Milgram lemma, the following \emph{a priori} estimate 
\begin{equation}
\label{eq:a-priori}
\left\Vert \mathbf{U}\right\Vert _{E,I }\leq 
\alpha^{-1} 
\sqrt{ \|f\|^2_{0,I} + \|g\|^2_{0,I}}.
\end{equation}
For the development of certain asymptotic expansions, it will be convenient to 
observe that our assumption (\ref{eq:A-positive-definite}) implies 
(see 
\iftechreport 
Lemma~\ref{lemma:elementary-matrix-properties} for the proof
\else
\cite[Appendix~\ref{appendix-misc}]{mxo-1}
\fi
)
\begin{eqnarray}
\label{eq:diagonal-positive}
a_{kk}(x)& \ge &\alpha^2 \text{ }\forall x\in \overline{I},\quad k=1,2,\\
\label{eq:det-positive}
\operatorname*{det} {\mathbf A}(x)  = a_{11}(x) a_{22}(x) - a_{12}(x) a_{21}(x) 
&\ge & \alpha^2 \max\{a_{11}(x),a_{22}(x)\} \ge \alpha^4
\qquad \forall x \in \overline{I}.  
\end{eqnarray}
We note that the special case when the parameters are equal, i.e. 
$\varepsilon =\mu $, was analyzed in \cite{xo2}. In the general case
considered here, there are three scales ($1\geq \mu \geq \varepsilon$) and
the regularity depends on how the scales are separated. Correspondingly,
there are 4 cases:

\begin{enumerate}[(I)]
\item 
\label{case:I}
The \textquotedblleft no scale separation case\textquotedblright\
which occurs when \emph{neither} $\mu /1$ \emph{nor} $\varepsilon /\mu $ is
small.

\item 
\label{case:IV}
The \textquotedblleft 3-scale case\textquotedblright\ in which all
scales are separated and occurs when $\mu /1$ is small \emph{and} $%
\varepsilon /\mu $ is small.

\item 
\label{case:II}
The first \textquotedblleft 2-scale case\textquotedblright\ which occurs
when $\mu /1$ is $\operatorname*{not}$ small \emph{and} $\varepsilon /\mu $ is small.

\item 
\label{case:III}
The second \textquotedblleft 2-scale case\textquotedblright\ which
occurs when $\mu /1$ is small \emph{and} $\varepsilon /\mu $ is \emph{not}
small.

\end{enumerate}

The concept of \textquotedblleft small \textquotedblright\ (or 
\textquotedblleft not small\textquotedblright ) mentioned above, 
is tied in two ways to our performing regularity theory in terms of 
asymptotic expansions. First, on the level 
of constructing asymptotic expansions, the decision which parameters
are deemed small determines the ansatz to be made and thus the form 
of the expansion. Second, on the level of 
applying asymptotic expansions, the decision which parameters are 
deemed small depends on whether the remainder resulting
from the asymptotic expansion can be regarded as small.


We need to introduce some notation: 
\begin{defn}
\label{def:function-spaces}
\begin{enumerate}
\item 
We say that a function $w$ is analytic with length scale $\nu > 0$ 
(and analyticity
parameters $C_w$, $\gamma_w$), abbreviated $w \in {\mathcal A}(\nu,C_w,\gamma_w)$, if 
$$
\|w^{(n)}\|_{L^\infty(I)} \leq C_w \gamma_w^n \max\{n,\nu^{-1}\}^n 
\qquad \forall n \in \N_0. 
$$
\item 
We say that that an entire function $w$ is of $L^\infty$-boundary layer type 
with length scale $\nu > 0$ (and analyticity parameters $C_w$, $\gamma_w$), 
abbreviated
$w \in \BLinf(\nu,C_w,\gamma_w)$, if for all $x \in I$ 
$$
|w^{(n)}(x)| \leq C_w \gamma_w^n \nu^{-n}
e^{-\dist(x,\partial I)/\nu}\qquad 
\forall n \in \N_0.
$$
\item 
We say that that an entire function $w$ is of $L^2$-boundary layer type with 
length scale $\nu >  0$ 
(and analyticity parameters $C_w$, $\gamma_w$), abbreviated
$w \in \BLtwo(\nu, C_w,\gamma_w)$, if 
$$
\|e^{\dist(x,\partial I)/\nu} w^{(n)}\|_{L^2(I)} \leq C_w \nu^{1/2} \gamma_w^n 
\nu^{-n} 
\qquad \forall n \in \N_0.
$$
\end{enumerate}
All three definitions extend naturally to vector-valued functions by requiring
the above bounds componentwise. 
\end{defn}

In view of the length of the article, we collect the main result at this 
point; the four scale separation cases (\ref{case:I})--(\ref{case:III}) 
listed above correspond to the four cases listed in the following theorem. 
We emphasize, however, that the four cases of the ensuing theorem are 
not mutually exclusive---in fact, the decompositions are valid regardless 
of the relation between the parameters $\varepsilon$ and $\mu$.

\begin{thm}
\label{thm:main} 
There exist constants $C$, $b$, $\delta$, $q$, $\gamma > 0$ independent of 
$0 < \varepsilon \leq \mu \leq 1$ such that the following assertions are true: 
\begin{enumerate}[(I)]
\item
${\mathbf U} \in {\mathcal A}(\varepsilon,C \varepsilon^{-1/2} ,\gamma)$. 
\item
${\mathbf U}$ can be written as 
${\mathbf U} = {\mathbf W} + \widetilde {\mathbf U}_{BL} + \widehat {\mathbf U}_{BL}
+{\mathbf R}$, where ${\mathbf W} \in {\mathcal A}(1,C,\gamma)$, 
$\widetilde{\mathbf U}_{BL} \in \BLinf(\delta \mu,C,\gamma)$
$\widehat{\mathbf U}_{BL} \in \BLinf(\delta \varepsilon,C,\gamma)$, 
and $\|{\mathbf R}\|_{L^\infty(\partial I)} + \|{\mathbf R}\|_{E,I} 
\leq C \left[ e^{-b/\mu} + e^{-b \mu/\varepsilon}\right]$. Additionally, 
the second component $\widehat v$ of $\widehat{\mathbf U}$ satisfies the 
sharper estimate $\widehat v \in \BLinf(\delta\varepsilon,C(\varepsilon/\mu)^2,\gamma)$. 
\item
If $\varepsilon/\mu \leq q$ then 
${\mathbf U}$ can be written as 
${\mathbf U} = {\mathbf W} +  \widehat {\mathbf U}_{BL}
+{\mathbf R}$, where ${\mathbf W} \in {\mathcal A}(\mu,C,\gamma)$, 
$\widehat{\mathbf U}_{BL} \in \BLinf(\delta \varepsilon,C,\gamma)$, 
and $\|{\mathbf R}\|_{L^\infty(\partial I)} + \|{\mathbf R}\|_{E,I} 
\leq C e^{-b /\varepsilon}$. Additionally, the second component
$\widehat v$ of $\widehat{\mathbf U}$ satisfies the sharper 
estimate 
$\widehat v \in \BLinf(\delta\varepsilon,C(\varepsilon/\mu)^2,\gamma)$. 
\item
${\mathbf U}$ can be written as 
${\mathbf U} = {\mathbf W} +  \widetilde {\mathbf U}_{BL}
+{\mathbf R}$, where ${\mathbf W} \in {\mathcal A}(1,C,\gamma)$, 
$\widetilde{\mathbf U}_{BL} \in \BLinf(\delta \mu ,C\sqrt{\mu/\varepsilon},\gamma \mu/\varepsilon)$, 
and $\|{\mathbf R}\|_{L^\infty(\partial I)} + \|{\mathbf R}\|_{E,I} 
\leq C \left(\mu/\varepsilon\right)^2 e^{-b/\mu}$. 
\end{enumerate}
\end{thm}
\begin{proof}
This result is obtained by combining 
Theorems~\ref{thm:case-I}, \ref{thm:case-IV}, \ref{thm:case-II}, and
 \ref{thm:case-III} to be found 
in Sections~\ref{case_1}--\ref{case_3}. We emphasize that some
results in these theorems are slightly sharper since they analyze 
all terms of the asymptotic expansions whereas 
Theorem~\ref{thm:main} is obtained from the asymptotic expansions by suitably
selecting the expansion order. 
\end{proof}

\section{The no scale separation case: Case~\ref{case:I}}
\label{case_1}
In this case neither 
$\mu/1 $ nor $\varepsilon /\mu $ is small, which means that the boundary layer
effects are not very pronounced. By the analyticity of $a_{ij},f$ and $g$,
we have that $u$ and $v$ are analytic. Moreover, we have the following
theorem.

\begin{thm}
\label{thm:case-I} Let ${\mathbf U} = (u,v)^T$ be the solution to \emph{(\ref{1})--(\ref{2})} with 
$0<\varepsilon \le \mu \leq 1$. Then there exist constants $C$ and $K>0$,
independent of $\varepsilon $ and $\mu $, such that
${\mathbf U} \in {\mathcal A}(\varepsilon,C \varepsilon^{-1/2},K)$ satisfying 
the sharper estimate 
\begin{equation}
\label{eq:thm:case-I-1}
\left\Vert u^{(n)}\right\Vert _{0,I}+\left\Vert v^{(n)}\right\Vert
_{0,I}\leq CK^{n}\max \{n,\varepsilon ^{-1}\}^{n}\quad \forall \;n\in 
\mathbb{N}_{0}. 
\end{equation}
\end{thm}

\begin{proof} 
The $L^2$-based estimate (\ref{eq:thm:case-I-1}) was shown 
in \cite{xo2} for the special case 
$\varepsilon = \mu$. The extension to the current situation $\varepsilon \leq \mu$ 
is straight forward. We note that the Sobolev embedding theorem in the form 
$\|v\|^2_{L^\infty(I)} \leq C \|v\|_{L^2(I)} \|v\|_{H^1(I)}$ allows us 
to infer from (\ref{eq:thm:case-I-1}) 
the assertion 
${\mathbf U} \in {\mathcal A}(\varepsilon, C \varepsilon^{-1/2},\gamma)$ 
for suitable $C$, $\gamma > 0$ independent of $\varepsilon$ and $\mu$. 
\end{proof}

\section{The three scale case\label{case_4}: Case~\ref{case:IV}}

In this case all scales are separated and it occurs when \emph{both} $\mu /1$
and $\varepsilon /\mu $ are deemed small. This is arguably the most interesting
(and challenging from the approximation point of view) case, since boundary
layers of multiple scales appear. Additionally, this case shows most clearly
the general procedure for obtaining asymptotic expansions and error bounds
for problems with multiple scales. Before developing the asymptotic
expansion, we formulate the main result: 
\begin{thm}
\label{thm:case-IV} 
The solution ${\mathbf U}$ of (\ref{eq:model-problem}) can be 
written as ${\mathbf U} = {\mathbf W} + \widetilde{\mathbf U}_{BL} 
+ \widehat{\mathbf U}_{BL}  + {\mathbf R}$, where 
${\mathbf W} \in {\mathcal A}(1, C_W,\gamma_W)$, 
$\widetilde {\mathbf U} \in \BLinf(\delta \mu, C_{BL},\gamma_{BL})$, 
$\widehat {\mathbf U} \in \BLinf(\delta \varepsilon, C_{BL},\gamma_{BL})$
for suitable constants $C_W$, $C_{BL}$, $\gamma_W$, $\gamma_{BL}$, $\delta > 0$
independent of $\mu$ and $\varepsilon$. Furthermore, ${\mathbf R}$ satisfies 
\begin{eqnarray*}
\|{\mathbf R}\|_{L^\infty(\partial I)} + 
\|L_{\varepsilon,\mu} {\mathbf R}\|_{L^\infty(I)} &\leq& C 
\left[ e^{-b/\mu} + e^{-b\mu/\varepsilon}\right],
\end{eqnarray*}
for some constants $C$, $b > 0$ independent of $\mu$ and $\varepsilon$. In particular, 
$\|{\mathbf R}\|_{E,I} \leq 
C \left[ e^{-b/\mu} + e^{-b\mu/\varepsilon}\right]$. 

Additionally, the second component $\widehat v$ of $\widehat {\mathbf U}_{BL}$  satisfies the sharper
regularity assertion 
$$
\widehat v \in \BLinf(\delta \varepsilon, C_{BL}(\varepsilon/\mu)^2, 
\gamma_{BL}).
$$
\end{thm}
\begin{proof}
See Section~\ref{sec:proof-of-thm:case-IV}.
\end{proof}

Anticipating that boundary layers of length scales $O(\mu)$ and $O(\varepsilon)$ will 
appear at the endpoints $x = 0$ and $x = 1$,  we introduce the stretched
variables $\widetilde{x}=x/\mu $ , $\widehat{x}=x/\varepsilon $ for the expected
layers at the left endpoint $x = 0$ and 
variables $\widetilde{x}^R=(1-x)/\mu $ , $\widehat{x}^R=(1-x)/\varepsilon $ for the expected
behavior at the right endpoint $x = 1$. We make the following formal ansatz for the 
solution ${\mathbf U}$:
\begin{equation}
\label{eq:28} 
{\mathbf U} \sim  \sum_{i=0}^\infty \sum_{j=0}^\infty 
\left(\frac{\mu}{1}\right)^i 
\left(\frac{\varepsilon}{\mu}\right)^j 
\left[
{\mathbf U}_{ij}(x) + \widetilde{\mathbf U}^L_{ij}(\widetilde x) + \widehat{\mathbf U}^L_{ij}(\widehat x) 
+ \widetilde{\mathbf U}^R_{ij}(\widetilde x^R)
+ \widehat{\mathbf U}^R_{ij}(\widehat x^R)
\right],
\end{equation}
where the functions ${\mathbf U}_{ij}$, $\widetilde{\mathbf U}^L_{ij}$,  $\widehat{\mathbf U}^L_{ij}$
$\widetilde{\mathbf U}^R_{ij}$,  $\widehat{\mathbf U}^R_{ij}$ will be determined shortly.
The decomposition of Theorem~\ref{thm:case-IV} is obtained by truncating the asymptotic 
expansion (\ref{eq:28}) after a finite number of terms: 
\begin{equation}
\mathbf{U}^M(x):=\left( 
\begin{array}{c}
u(x) \\ 
v(x)%
\end{array}%
\right) =\mathbf{W}_{M}(x)+\widehat{\mathbf{U}}_{BL}^{M}(\widehat{x})+%
\widehat{\mathbf{V}}_{BL}^{M}(\widehat{x}^R)+\widetilde{\mathbf{U}}_{BL}^{M}(%
\widetilde{x})+\widetilde{\mathbf{V}}_{BL}^{M}(\widetilde{x}^R)+\mathbf{R}%
_{M}(x),  \label{U_ior4}
\end{equation}
where
\begin{equation}
\mathbf{W}_{M}(x)=\sum_{i=0}^{M_{1}}\sum_{j=0}^{M_{2}}\mu ^{i}\left( \frac{%
\varepsilon }{\mu }\right) ^{j}\left( 
\begin{array}{c}
u_{ij}(x) \\ 
v_{ij}(x)%
\end{array}%
\right) ,  \label{wMc}
\end{equation}
denotes the outer (smooth) expansion,
\begin{equation}
\widehat{\mathbf{U}}_{BL}^{M}(\widehat{x})=\sum_{i=0}^{M_{1}}%
\sum_{j=0}^{M_{2}}\mu ^{i}\left( \frac{\varepsilon }{\mu }\right) ^{j}\left( 
\begin{array}{c}
\widehat{u}_{ij}^{L}(\widehat{x}) \\ 
\widehat{v}_{ij}^{L}(\widehat{x})%
\end{array}%
\right) \;,\;\widehat{\mathbf{V}}_{BL}^{M}(\widehat{x}^R)=\sum_{i=0}^{M_{1}}%
\sum_{j=0}^{M_{2}}\mu ^{i}\left( \frac{\varepsilon }{\mu }\right) ^{j}\left( 
\begin{array}{c}
\widehat{u}_{ij}^{R}(\widehat{x}^R) \\ 
\widehat{v}_{ij}^{R}(\widehat{x}^R)%
\end{array}%
\right) ,  \label{BL_Mc1}
\end{equation}
denote the left and right inner (boundary layer) expansions associated with 
the variables $\widehat{x}$, $\widehat{x}^R$, respectively, 
\begin{equation}
\widetilde{\mathbf{U}}_{BL}^{M}(\widetilde{x})=\sum_{i=0}^{M_{1}}%
\sum_{j=0}^{M_{2}}\mu ^{i}\left( \frac{\varepsilon }{\mu }\right) ^{j}\left( 
\begin{array}{c}
\widetilde{u}_{ij}^{L}(\widetilde{x}) \\ 
\widetilde{v}_{ij}^{L}(\widetilde{x})%
\end{array}%
\right) \;,\;\widetilde{\mathbf{V}}_{BL}^{M}(\widetilde{x}%
)=\sum_{i=0}^{M_{1}}\sum_{j=0}^{M_{2}}\mu ^{i}\left( \frac{\varepsilon }{\mu 
}\right) ^{j}\left( 
\begin{array}{c}
\widetilde{u}_{ij}^{R}(\widetilde{x}^R) \\ 
\widetilde{v}_{ij}^{R}(\widetilde{x}^R)%
\end{array}%
\right) ,  \label{BL_Mc2}
\end{equation}
denote the left and right inner (boundary layer) expansions associated with 
the variables $\widetilde{x}$, $\widetilde{x}^R$
respectively, and
\begin{equation}
\mathbf{R}_{M}(x):=\left( 
\begin{array}{c}
r_{u}(x) \\ 
r_{v}(x)%
\end{array}%
\right) =\mathbf{U}(x)-\left( \mathbf{W}_{M}(x)+\widehat{\mathbf{U}}%
_{BL}^{M}(\widehat{x})+\widehat{\mathbf{V}}_{BL}^{M}(\widehat{x}^R)+\widetilde{%
\mathbf{U}}_{BL}^{M}(\widetilde{x})+\widetilde{\mathbf{V}}_{BL}^{M}(%
\widetilde{x}^R)\right)   \label{rMc}
\end{equation}
denotes the remainder. Theorem~\ref{thm:case-IV} will be established by 
selecting $M_1  = O(1/\mu)$, $M_2 = O(\mu/\varepsilon)$. 

\subsection{Derivation of the asymptotic expansion}
In order to derive equations for the functions 
${\mathbf U}_{ij}$, 
$\widetilde{\mathbf U}^L_{ij}$, 
$\widehat{\mathbf U}^L_{ij}$, 
$\widetilde{\mathbf U}^R_{ij}$, 
$\widehat{\mathbf U}^R_{ij}
$, 
the procedure is as follows: first, the ansatz
(\ref{eq:28}) is inserted in the differential equation (\ref{1}), then the scales are separated
and finally recursions are obtained by equating like powers of $\mu$ and $\varepsilon/\mu$. 

In order to perform the scale separation, 
we need to write the differential operator $L_{\varepsilon,\mu}$ in different ways 
on the various scales. In particular, for the $\widetilde x$ and the $\widehat x$-scales, 
the coefficient ${\mathbf A}$ is written, by Taylor expansion, as  
\begin{eqnarray}
{\mathbf A}(x) &=& \sum_{k=0}^\infty \mu^k {\mathbf A}_k \widetilde x^k, 
\qquad {\mathbf A}_k := {\mathbf A}^{(k)}(0)=  
\left(\begin{array}{cc} \frac{a_{11}^{(k)}(0)}{k!} & \frac{a_{12}^{(k)}(0)}{k!} \\
                        \frac{a_{21}^{(k)}(0)}{k!} & \frac{a_{22}^{(k)}(0)}{k!} 
       \end{array}
\right), \\
{\mathbf A}(x) &=& \sum_{k=0}^\infty \mu^k \left(\frac{\varepsilon}{\mu}\right)^k 
{\mathbf A}_k \widehat x^k. 
\end{eqnarray}
Corresponding representations are obtained for the variables 
$\widetilde x^R$ and $\widehat x^R$ by expanding around the right endpoint $x = 1$. 
Hence, the differential operator $L_{\varepsilon,\mu}$ applied to a function depending
on $\widetilde x$ or $\widehat x$ takes the following form: 
\begin{eqnarray} 
\label{eq:L-on-tilde-scale}
\mbox{on the $\widetilde x$-scale:}
&\qquad & -\mu^{-2} {\mathbf E}^{\varepsilon,\mu} \partial_{\widetilde x}^2 {\mathbf U}(\widetilde x) + 
\sum_{k=0}^\infty \mu^k {\mathbf A}_k \widetilde x^k 
{\mathbf U}(\widetilde x), \\
\label{eq:L-on-hat-scale}
\mbox{on the $\widehat x$-scale:} 
&\qquad & -\varepsilon^{-2}  {\mathbf E}^{\varepsilon,\mu} \partial_{\widehat x}^2 
{\mathbf U}(\widehat x) + 
\sum_{k=0}^\infty \mu^k \left(\frac{\varepsilon}{\mu}\right)^k {\mathbf A}_k \widehat x^k {\mathbf U}(\widehat x).
\end{eqnarray}
Clearly, analogous forms 
exist for the operator on the $\widetilde x^R$ and $\widehat x^R$
scales. 
We now insert the ansatz (\ref{eq:28}) in the differential equation equation (\ref{1}), where
the differential operator $L_{\varepsilon,\mu}$ takes the form given above on the fast scales 
$\widetilde x$, $\widehat x$, $\widetilde x^R$, $\widehat x^R$, and we separate the scales, i.e., we view
the variables $x$, $\widetilde x$, $\widehat x$, $\widetilde x^R$, $\widehat x^R$ as independent
variables. Then, we obtain
\begin{eqnarray}
\label{eq:3-scale-10}
\sum_{i=0}^\infty\sum_{j=0}^\infty \mu^i \left(\frac{\varepsilon}{\mu}\right)^j 
\left[ - {\mathbf E}^{\varepsilon,\mu} {\mathbf U}_{ij}^{\prime\prime} + {\mathbf A}(x) {\mathbf U}_{ij}
\right]
&=& {\mathbf F}, \\
\label{eq:3-scale-20}
\sum_{i=0}^\infty\sum_{j=0}^\infty \mu^i \left(\frac{\varepsilon}{\mu}\right)^j 
\left[ - \mu^{-2} {\mathbf E}^{\varepsilon,\mu} (\widetilde {\mathbf U}^L_{ij})^{\prime\prime} + 
\sum_{k=0}^\infty \mu^k {\mathbf A}_k \widetilde x^k \widetilde{\mathbf U}^L_{ij} 
\right] 
&=& 0, \\
\label{eq:3-scale-30}
\sum_{i=0}^\infty\sum_{j=0}^\infty \mu^i \left(\frac{\varepsilon}{\mu}\right)^j 
\left[ - \varepsilon^{-2} {\mathbf E}^{\varepsilon,\mu} (\widehat {\mathbf U}^L_{ij})^{\prime\prime} + 
\sum_{k=0}^\infty \varepsilon^k {\mathbf A}_k \widehat x^k \widehat{\mathbf U}^L_{ij} 
\right] 
&=& 0, 
\end{eqnarray}
and two additional equations for $\widetilde {\mathbf U}^R$, $\widehat {\mathbf U}^R$ 
corresponding to the scales $\widetilde x^R$, $\widehat x^R$ that 
are completely analogous to (\ref{eq:3-scale-20}), (\ref{eq:3-scale-30}). We write 
\begin{equation}
\label{eq:3-scale-40}
{\mathbf U}_{ij} = \left(\begin{array}{c} u_{ij} \\ 
                                          v_{ij}
                         \end{array}\right), 
\qquad 
\widetilde {\mathbf U}^L_{ij} = \left(\begin{array}{c} \widetilde u^L_{ij} \\ 
                                                       \widetilde v^L_{ij}
                                      \end{array}
                                \right), 
\qquad 
\widehat {\mathbf U}^L_{ij} = \left(\begin{array}{c} \widehat u^L_{ij} \\ 
                                                       \widehat v^L_{ij}
                                      \end{array}
                                \right), 
\end{equation}
and equate like powers of $\mu$ and $\varepsilon/\mu$ 
in (\ref{eq:3-scale-10}), (\ref{eq:3-scale-20}), (\ref{eq:3-scale-30}) 
to get the following recursions: 
\begin{subequations}
\label{eq:case-IV-recurrence}
\begin{eqnarray}
\label{eq:case-IV-recurrence-outer}
- \left(\begin{array}{c} u_{i-2,j-2}^{\prime\prime} \\
                         v_{i-2,j}^{\prime\prime}
        \end{array}
  \right)
 + {\mathbf A}(x) {\mathbf U}_{ij} &=& {\mathbf F}_{ij}, \\
\label{eq:case-IV-recurrence-tilde}
- \left(\begin{array}{c} (\widetilde u^L_{i,j-2})^{\prime\prime} \\
                         (\widetilde v^L_{i,j})^{\prime\prime}
        \end{array}
  \right)
 + \sum_{k=0}^i{\mathbf A}_k \widetilde x^k  \widetilde {\mathbf U}^L_{i-k,j} &=&  0, \\
\label{eq:case-IV-recurrence-hat}
- \left(\begin{array}{c} (\widehat u^L_{i,j})^{\prime\prime} \\
                         (\widehat v^L_{i,j+2})^{\prime\prime}
        \end{array}
  \right)
 + \sum_{k=0}^{\min\{i,j\}}{\mathbf A}_k \widehat x^k  \widehat {\mathbf U}^L_{i-k,j-k} &=&  0, 
\end{eqnarray}
where we adopt the convention that if a function appears with a negative subscript, then 
it is assumed to be zero. Furthermore, we set
$$
{\mathbf F}_{00} = \left(\begin{array}{c} f \\ g \end{array}\right), 
\qquad {\mathbf F}_{ij} = 0 \quad \mbox{ if $(i,j) \ne (0,0)$}.
$$
The procedure so far leads to a formal solution ${\mathbf U}$ of the differential equation
(\ref{1}); further boundary conditions are imposed in order to conform to the 
boundary conditions (\ref{2}), namely,
\begin{eqnarray}
\label{eq:case-IV-recurrence-bc}
{\mathbf U}_{ij}(0) + \widetilde{\mathbf U}^L_{ij}(0) + \widehat{\mathbf U}^L_{ij}(0) = 0, 
\qquad \mbox{ plus decay conditions for $\widetilde {\mathbf U}^L_{ij}$, $\widehat{\mathbf U}^L_{ij}$ at $+\infty$}
\end{eqnarray}
\end{subequations}
with analogous conditions at the right endpoint $x = 1$, 
which couple ${\mathbf U}_{ij}$, $\widetilde{\mathbf U}_{ij}^R$, 
and $\widehat{\mathbf U}^R_{ij}$.  
%
%
\subsection{Analysis of the functions ${\mathbf U}_{ij}$}
\label{sec:analysis-of-Uij}
Since the matrix ${\mathbf A}(x)$ is invertible for every $x \in I$, 
equation (\ref{eq:case-IV-recurrence-outer}) may be solved for any 
$i$, $j$ yielding
%
%
%
\begin{equation}
\left( 
\begin{array}{c}
u_{0,0} \\ 
v_{0,0}%
\end{array}%
\right) ={\mathbf A}^{-1}\left( 
\begin{array}{c}
f \\ 
g%
\end{array}%
\right) ,  \label{38}
\end{equation}
and for $(i,j) \ne (0,0)$
\begin{equation}
\left( 
\begin{array}{c}
u_{ij} \\ 
v_{ij}%
\end{array}%
\right) ={\mathbf A}^{-1}\left( 
\begin{array}{c}
u_{i-2,j-2}^{\prime \prime } \\ 
v_{i-2,j}^{\prime \prime }%
\end{array}%
\right) ,  \label{39}
\end{equation}
with, as mentioned above,
\begin{equation*}
u_{ij}=0\;,\;v_{ij}=0\text{ if }i<0\text{ or }j<0.
\end{equation*}
Note that (\ref{39}) gives all the cases $(i,0)$ and $(0,j)$ because the
right-hand side in (\ref{39}) is known. \ Moreover, for each $j$, (\ref{39})
allows us to compute $u_{ij},v_{ij}$ $\forall $ $i$, thus (\ref{38})--(\ref{39}) 
uniquely determine $u_{ij},v_{ij}$ $\forall $ $i,j.$

We have the following lemma concerning the regularity 
of the functions ${\mathbf U}_{ij}$: 

\begin{lem}
\label{lem10}
Let $f$, $g$, and ${\mathbf A}$ satisfy (\ref{3}) and 
(\ref{eq:A-positive-definite}). 
Let $u_{ij}$, $v_{ij}$ be the solutions of 
\emph{(\ref{38}), (\ref{39})}. 
Then there exist positive constants $C_S$ and $K$ and a complex
neighborhood $G$ of the closed interval $\overline{I}$ 
independent of $i$ and $j$ 
such that
\begin{equation}
u_{ij}=v_{ij}=0\;\forall \;j>i,  \label{39b}
\end{equation}
\begin{equation}
u_{ij}=v_{ij}=0\; \quad \mbox{ if $i$ or $j$ is odd,}
\label{39c}
\end{equation}
\begin{equation}
\left\vert u_{ij}(z)\right\vert +\left\vert v_{ij}(z)\right\vert \leq
C_S\delta ^{-i}K^{i}i^{i}\;\forall \;z\in G_{\delta }:=
\{z \in G \colon \dist(z,\partial G) > \delta\}.  \label{39d}
\end{equation}
\end{lem}

\begin{proof}
The proof is by induction on $i$, where the estimate (\ref{39d}) 
follows by arguments of the type worked out in the proof of 
\cite[Lemma~{2}]{m}. 
\iftechreport
For details, see Appendix~\ref{appendix:case_4_lem10}. 
\else
For details, see \cite[Appendix~\ref{appendix:case_4_lem10}]{mxo-1}
\fi
\end{proof}

\subsection{Analysis of $\widetilde{\mathbf U}_{ij}^L$, $\widehat{\mathbf U}_{ij}^L$}
\subsubsection{Properties of some solution operators}
\begin{lem}
\label{lem6}Let $\alpha$ ,$a$, $b\in \mathbb{R}^{+}.$ Then
\begin{eqnarray*}
\int_x^\infty e^{-\alpha t}\left( a+bt\right)
^{i}dt &\leq & 
\frac{1}{\alpha }e^{-\alpha x}\overset{i}{\underset{\nu =0}{\sum }%
}\left( a+bx\right) ^{i-\nu }\left( \frac{ib}{\alpha }\right) ^{\nu },\\
\int_x^\infty \int_t^\infty 
e^{-\alpha \tau }\left( a+b\tau \right) ^{i}d\tau dt &\leq & 
\frac{1}{\alpha ^{2}%
}e^{-\alpha x}\overset{i}{\underset{\nu =0}{\sum }}\overset{i-\nu }{\underset%
{\ell =0}{\sum }}\left( \frac{ib}{\alpha }\right) ^{\nu }\left( (i-\nu )%
\frac{b}{\alpha }\right) ^{\ell }\left( a+bx\right) ^{i-\nu -\ell }.
\end{eqnarray*}
\end{lem}

\begin{proof}
We have, after successive integrations by parts,
\begin{eqnarray*}
\int_x^\infty e^{-\alpha t}\left( a+bt\right) ^{i}dt& =&
\frac{1}{\alpha }e^{-\alpha x}\overset{i}{\underset{\nu =0}{\sum }}\left(
a+bx\right) ^{i-\nu }i(i-1)\cdots (i-\nu +1)\left( \frac{b}{\alpha }\right)
^{\nu }\\
& \leq & 
\frac{1}{\alpha }e^{-\alpha x}\overset{i}{\underset{\nu =0}{\sum }}%
\left( a+bx\right) ^{i-\nu }\left( \frac{ib}{\alpha }\right) ^{\nu }.
\end{eqnarray*}
The statement about the double integral follows from this result.
\end{proof}

The above lemma can be formulated in the complex plane as follows.

\begin{lem}
\label{lem7}Let $\alpha$ ,$a$, $b\in \mathbb{R}^{+}$ and $v$ be holomorphic
in the half-plane $\operatorname*{Re} z > z_0$ 
and assume 
$\left\vert
v(z)\right\vert \leq e^{-\alpha \operatorname*{Re}(z)}\left( a+b\left\vert
z\right\vert \right) ^{j}$. Then
for $z$ with $\operatorname*{Re} z > z_0$:
\begin{eqnarray*}
\left\vert \int_z^\infty v(t)dt\right\vert & \leq &
\frac{1}{\alpha }e^{-\alpha \operatorname*{Re}(z)}\overset{j}{\underset{\nu =0}{\sum }%
}\left( a+b\left\vert z\right\vert \right) ^{j-\nu }\left( \frac{jb}{\alpha }%
\right) ^{\nu }, \\
\left\vert \int_z^\infty \int_t^\infty 
v(\tau )d\tau dt\right\vert & \leq & \frac{1}{\alpha ^{2}}%
e^{-\alpha \operatorname*{Re}(z)}\overset{j}{\underset{\nu =0}{\sum }}\overset{j-\nu }%
{\underset{\ell =0}{\sum }}\left( \frac{jb}{\alpha }\right) ^{\nu }\left(
(j-\nu )\frac{b}{\alpha }\right) ^{\ell }\left( a+b\left\vert z\right\vert
\right) ^{j-\nu -\ell }.
\end{eqnarray*}
\end{lem}
\begin{lem}
\label{lemma:scalar-bvp-constant-coefficients}
Let $a > 0$, $g \in \R$, and $f$ be entire satisfying 
$$
|f(z)| \leq C_f (q+|z|)^j 
\begin{cases} 
e^{-\ua \operatorname*{Re}(z)}  & \mbox{ for }\operatorname*{Re}(z) >  0 \\ 
e^{-\oa \operatorname*{Re}(z)}  & \mbox{ for }\operatorname*{Re}(z) <  0 
\end{cases}
$$
for some $C_f$, $q$, $\ua$, $\oa$, 
with $(a+\ua) q \ge 2j+1$ and $0 < \ua\leq a 
\leq \oa$. 

Then the solution $u$ of 
$$
-u^{\prime\prime}(z) + a^2 u(z) = f(z), 
\qquad u(0) = g, \qquad \lim_{z \rightarrow \infty} u(z) = 0 
$$
satisfies the bound 
$$
|u(z)| \leq C_f  \left(\frac{1}{a} \left( q + |z|\right)^{j+1}\frac{1}{j+1} + |g|\right)
\begin{cases}
e^{-\ua \operatorname*{Re}(z)} & \mbox{ for $\operatorname*{Re}(z) \ge 0$}\\
e^{-\oa \operatorname*{Re}(z)} & \mbox{ for $\operatorname*{Re}(z) < 0$}.
\end{cases}
$$
\end{lem}
\begin{proof}
The proof follows from the appropriate modifications of \cite[Lemma~{7.3.6}]{mB}.
\iftechreport
For details, see 
Appendix~\ref{appendix:case_4_lemma:scalar-bvp-constant-coefficients}
\else
For details, see 
\cite[Appendix~\ref{appendix:case_4_lemma:scalar-bvp-constant-coefficients}]{mxo-1}. 
\fi

\end{proof}
\begin{lem}
\label{lemma:two-anti-derivatives}
Let $0 < \ua \leq \oa$. Let $v$ be entire and satisfy 
for some $a >  0$, $C_v$, $b \ge 0$, $j \in \N_0$, the bound 
$$
|v(z)| \leq C_v (a+ b|z|)^j 
\begin{cases}
e^{-\ua \operatorname*{Re}(z)} & \mbox{ for $\operatorname*{Re}(z) \ge 0$} \\
e^{-\oa \operatorname*{Re}(z)} & \mbox{ for $\operatorname*{Re}(z) < 0$}.  
\end{cases}
$$
Assuming $\frac{j b}{a \ua} < 1$, there holds
$$
\left| \int_{z}^\infty \int_t^\infty v(\tau)\,d\tau\,dt\right|
\leq C_v \frac{1}{\ua^2} \left(\frac{1}{1-\frac{jb}{\ua a}}\right)^2 (a+ b|z|)^j 
\begin{cases}
e^{-\ua \operatorname*{Re}(z)} & \mbox{ for $\operatorname*{Re}(z) \ge 0$}\\
e^{-\oa \operatorname*{Re}(z)} & \mbox{ for $\operatorname*{Re}(z) < 0$.}
\end{cases}
$$
\end{lem}
\begin{proof}
Follows essentially from Lemma~\ref{lem7}. We sketch the argument for the case
$\operatorname*{Re}(z) < 0$. By linearity, we may assume $C_v = 1$. 
We start with the single integral 
$
\int_z^\infty v(t)\,dt,
$
by selecting as the path of integration the line $z + \tau$, $\tau \in \R_+$, we 
get with the aid of Lemma~\ref{lem7}, 
\begin{eqnarray*}
\left| \int_z^\infty v(t)\,dt\right|  & \leq &
e^{-\oa \operatorname*{Re}(z)}  \int_0^{-\operatorname*{Re}z} e^{-\oa \tau}(a+b|z| + b \tau)^j\,d\tau + 
e^{-\ua \operatorname*{Re}(z)}  \int_{-\operatorname*{Re} z}^{\infty} e^{-\ua \tau} (a+b|z| + b \tau)^j\,d\tau\\
&\leq& e^{-\oa \operatorname*{Re}(z)} \int_0^\infty e^{-\ua \tau} ( a + b|z| + b\tau)^j\,d\tau
\leq e^{-\oa \operatorname*{Re}(z)} \frac{1}{\ua} 
\sum_{\nu=0}^j (a + b |z|)^{j-\nu} \left(\frac{j b}{\ua}\right)^{\nu} \\
&\leq & e^{-\oa \operatorname*{Re}(z)} \frac{1}{\ua} (a + b|z|)^j 
\sum_{\nu=0}^j \left(\frac{j b}{\ua (a + b|z|)}\right)^{\nu}
\leq  e^{-\oa \operatorname*{Re}(z)} \frac{1}{\ua} (a + b|z|)^j 
\frac{1}{1 - \frac{b j}{a \ua}}. 
\end{eqnarray*}
Inspection of the above derivation shows that for $\operatorname*{Re}(z) \ge 0$, the same estimate
holds with $e^{-\oa \operatorname*{Re}(z)}$ replaced by $e^{-\ua \operatorname*{Re}(z)}$. We may
therefore repeat the argument once more for the function $z \mapsto \int_z^\infty v(t)\,dt$ to get
the claimed estimate. 
\end{proof}
\begin{lem}
\label{lemma:derivatives-of-entire-fcts}
Let the entire function $v$ satisfy the hypotheses stated in 
Lemma~\ref{lemma:two-anti-derivatives}. Then 
$$
|v^{\prime\prime}(z)| \leq 2 e^{\oa} C_v (a + b + b|z|)^j 
\begin{cases}
e^{-\ua \operatorname*{Re}(z)} & \mbox{ for $\operatorname*{Re}(z) \ge 0$}\\
e^{-\oa \operatorname*{Re}(z)} & \mbox{ for $\operatorname*{Re}(z) < 0$}
\end{cases}
 \qquad \forall z \in \C.
$$
\end{lem}
\begin{proof}
Follows from Cauchy's formula for derivatives by taking $\partial B_1(z)$ as the contour. 
\end{proof}
\begin{lem}
\label{lemma:elementary-properties-of-factorial}
For $C_1$, $\beta>0$ and $x \ge 0$ the following estimates are valid 
with $\gamma = 2 \max\{1,C_1^2\}$: 
\begin{eqnarray}
\label{eq:lemma:estimate-remainder-case-II-1}
(C_{1}\ell +\widehat{x})^{2\ell}& \leq & 
2^{\ell}(C_{1}\ell )^{2\ell}
+2^{\ell}\widehat{x}^{2 \ell }
\leq \gamma ^{\ell }\left( \ell^{2\ell }+\widehat{x}%
^{2\ell }\right), \\
\label{eq:lemma:estimate-remainder-case-II-2}
\sup_{x > 0} x^n e^{-\beta/4 x} &\leq & \left(\frac{4 n}{ e \beta }\right)^n.
\end{eqnarray}
\end{lem}
\begin{proof}
The result follow from elementary considerations. 
\end{proof}
\subsubsection{Regularity of the functions $\widetilde {\mathbf U}^L_{ij}$ and $\widehat{\mathbf U}^L_{ij}$}
We turn our attention to equations 
(\ref{eq:case-IV-recurrence-tilde}) and 
(\ref{eq:case-IV-recurrence-hat}), which, after introducing appropriate boundary conditions,  
determine $\widetilde{u}_{ij}^{L}$, $\widetilde{v}_{ij}^{L}$ and 
$\widehat{u}_{ij}^{L}$, $\widehat{v}_{ij}^{L}$, respectively. These equations turn out to be 
systems of differential-algebraic equations (DAEs); however, their structure is such that the algebraic
side constraint of the DAE can be eliminated explicitly and, additionally, 
we will be able to solve for 4 scalar functions sequentially instead of having to consider the coupled system. We recall that the functions 
${\mathbf U}_{ij} = (u_{ij},v_{ij})^T$ have been defined and studied
in Section~\ref{sec:analysis-of-Uij}. 

These equations may be solved
by induction on $j$ and $i$. For $j=0$, we solve (\ref{eq:case-IV-recurrence-tilde}) for any $(i,0)$ by first
solving for $\widetilde{u}_{i,0}^{L}$ and inserting it into the equation
for $\widetilde{v}_{i,0}^{L}$. We have from (\ref{eq:case-IV-recurrence-tilde}, 1$^\text{st}$ eqn)
\begin{equation}
\widetilde{u}_{i,0}^{L}=-\frac{a_{12}(0)}{a_{11}(0)}\widetilde{v}_{i,0}^{L}-%
\frac{1}{a_{11}(0)}\underset{k=1}{\overset{i}{\sum }}\frac{\widetilde{x}^{k}%
}{k!}\left[ a_{11}^{(k)}(0)\widetilde{u}_{i-k,0}^{L}+a_{12}^{(k)}(0)%
\widetilde{v}_{i-k,0}^{L}\right] ,  \label{40}
\end{equation}
which, upon inserted into (\ref{eq:case-IV-recurrence-tilde}, 2$\text{nd}$ eqn) gives
\begin{subequations}
\label{eq:41-42}
\begin{eqnarray}
&&-\left( \widetilde{v}_{i,0}^{L}\right) ^{\prime \prime }+\frac{%
a_{11}(0)a_{22}(0)-a_{21}(0)a_{12}(0)}{a_{11}(0)}\widetilde{v}_{i,0}^{L}
\label{41} \\
&=&\underset{k=1}{\overset{i}{\sum }}\frac{\widetilde{x}^{k}}{k!}\left[
\left( \frac{a_{21}(0)}{a_{11}(0)}a_{11}^{(k)}(0)-a_{21}^{(k)}(0)\right) 
\widetilde{u}_{i-k,0}^{L}+\left( \frac{a_{21}(0)}{a_{11}(0)}%
a_{12}^{(k)}(0)-a_{22}^{(k)}(0)\right) \widetilde{v}_{i-k,0}^{L}\right] . 
\notag
\end{eqnarray}
The above second order differential equation is now posed as an 
equation in $(0,\infty)$ 
and supplemented with the two ``boundary'' conditions 
\begin{equation}
\label{42}
\widetilde{v}_{i,0}^{L}(0) = -v_{i,0}(0), 
\qquad 
\widetilde{v}_{i,0}^L(\widetilde x) \rightarrow 0 \qquad \mbox{ as $\widetilde x\rightarrow \infty$}.
\end{equation}
\end{subequations}
So, solving (\ref{eq:41-42}) gives us $\widetilde{v}_{i,0}^{L}$
and then from (\ref{40}) we get $\widetilde{u}_{i,0}^{L}$.  
Inductively, we obtain $\widetilde{v}_{i,0}^L$ and $\widetilde{u}_{i,0}^L$
for all $i \ge 0$. 

Next, we set 
\begin{equation}
\label{eq:42a}
\widehat v_{i,0}^L = \widehat v_{i,1}^L = 0, 
\end{equation}
and we solve with $j=0$ (\ref{eq:case-IV-recurrence-hat}, 1$^\text{st}$ eqn) 
for $\widehat{u}_{i,0}^{L}$ 
(using $\widehat{v}_{i,0}^{L}=0$) with boundary conditions from $u_{i,0}$:
\begin{equation}
\left\{ 
\begin{array}{c}
-\left( \widehat{u}_{i,0}^{L}\right) ^{\prime \prime }+a_{11}(0)\widehat{u}%
_{i,0}^{L}=0 \\ 
\widehat{u}_{i,0}^{L}(0) = -u_{i,0}, 
\qquad \widehat{u}_{i,0}^{L}(\widehat x) \rightarrow 0 \quad \mbox{ for } \widehat x \rightarrow \infty.
\end{array}%
\right. \label{43}
\end{equation}
Then, we solve 
(\ref{eq:case-IV-recurrence-hat}, $2^\text{nd}$ eqn) for $\widehat{v}_{i,2}^{L}$:
\begin{equation}
\widehat{v}_{i,2}^{L}(z)=\int_z^\infty
\int_t^\infty a_{21}(0)\widehat{u}_{i,0}^{L}(\tau )d\tau dt.
\label{44}
\end{equation}
In general, assume we have performed the previous steps and we have determined 
$\widetilde{u}_{i,j}^{L}$, $\widetilde{v}_{i,j}^{L}$, 
$\widehat{u}_{i,j}^{L}$, $\widehat{v}_{i,j+2}^{L}$ 
for all $i\geq 0$ and second index up to $j$. To obtain the corresponding functions 
(with $j$ replaced by $j+1$) we proceed analogously.
We first solve (\ref{eq:case-IV-recurrence-tilde}, 1$^{\text{st}}$ eqn) 
for $\widetilde{u}_{i,j+1}^{L},$
\begin{equation}
\widetilde{u}_{i,j+1}^{L}=-\frac{a_{12}(0)}{a_{11}(0)}\widetilde{v}%
_{i,j+1}^{L}+\frac{\left( \widetilde{u}_{i,j-1}^{L}\right) ^{\prime \prime }%
}{a_{11}(0)}-\frac{1}{a_{11}(0)}\underset{k=1}{\overset{i}{\sum }}\frac{%
\widetilde{x}^{k}}{k!}\left[ a_{11}^{(k)}(0)\widetilde{u}%
_{i-k,j+1}^{L}+a_{12}^{(k)}(0)\widetilde{v}_{i-k,j+1}^{L}\right] ,
\label{51}
\end{equation}
and plug it into (\ref{eq:case-IV-recurrence-tilde}, 2$^{\text{nd}}$ eqn):
\begin{subequations}
\label{eq:52-53}
\begin{eqnarray}
\nonumber 
\lefteqn{
-\left( \widetilde{v}_{i,j+1}^{L}\right) ^{\prime \prime }+\frac{%
a_{11}(0)a_{22}(0)-a_{21}(0)a_{12}(0)}{a_{11}(0)}\widetilde{v}_{i,j+1}^{L}=-%
\frac{a_{21}(0)}{a_{11}(0)}\left( \widetilde{u}_{i,j-1}^{L}\right) ^{\prime
\prime }+ } \\
&&
+\underset{k=1}{\overset{i}{\sum }}\frac{\widetilde{x}^{k}}{k!}\left[ \left( 
\frac{a_{21}(0)}{a_{11}(0)}a_{11}^{(k)}(0)-a_{21}^{(k)}(0)\right) \widetilde{%
u}_{i-k,j+1}^{L}+\left( \frac{a_{21}(0)}{a_{11}(0)}%
a_{12}^{(k)}(0)-a_{22}^{(k)}(0)\right) \widetilde{v}_{i-k,j+1}^{L}\right] .
\qquad 
\label{52}
\end{eqnarray}
The second order ODE, equation (\ref{52}), is supplemented with the boundary conditions
\begin{equation}
\widetilde{v}_{i,j+1}^{L}(0) =- \left(
v_{i,j+1}(0)+\widehat{v}_{i,j+1}^{L}(0)\right), 
\qquad \widetilde{v}_{i,j+1}^L (\widetilde x) \rightarrow 0 \quad \mbox{ for } 
\widetilde x \rightarrow \infty.
\label{53}
\end{equation}
\end{subequations}
%
This
gives us $\widetilde{v}_{i,j+1}^{L}$ and in turn 
$\widetilde{u}_{i,j+1}^{L}$ from 
(\ref{51}).

Next, we solve (\ref{eq:case-IV-recurrence-hat}, 1$^{\text{st}}$ eqn) for $\widehat{u}_{i,j+1}^{L}$
with boundary conditions from $u_{i,j+1}$ and $\widetilde{u}_{i,j+1}^{L}$:
\begin{subequations}
\label{eq:54-55}
\begin{eqnarray}
-\left( \widehat{u}_{i,j+1}^{L}\right) ^{\prime \prime }+a_{11}(0)\widehat{u}%
_{i,j+1}^{L} &=& a_{12}(0) \widehat v_{i,j+1}^L - \nonumber \\
&-& \underset{k=1}{\overset{\min \{i,j+1\}}{\sum }}\frac{%
\widehat{x}^{k}}{k!}\left( a_{11}^{(k)}(0)\widehat{u}%
_{i-k,j+1-k}^{L}-a_{12}^{(k)}(0)\widehat{v}_{i-k,j+1-k}^{L}\right)
\label{54} \\
\label{55a}
\widehat{u}_{i,j+1}^{L}(0) &=&- \left(
u_{i,j+1}(0)+\widetilde{u}_{i,j+1}^{L}(0)\right) , \\
\widehat{u}_{i,j+1}^{L}(\widehat x) &\rightarrow & 0 \quad \mbox{ for } \widehat x \rightarrow \infty.
\label{55b}
\end{eqnarray}
\end{subequations}
Finally, we solve (\ref{eq:case-IV-recurrence-hat}, 2$^{\text{nd}}$ eqn) for $\widehat{v}_{i,j+3}^{L}$:
\begin{equation}
\widehat{v}_{i,j+3}^{L}(z)=\underset{k=0}{\overset{\min \{i,j+1\}}{\sum }}%
\frac{1}{k!}\int_z^\infty \int_t^\infty 
\tau^{k}\left\{ a_{21}^{(k)}(0)\widehat{u}%
_{i-k,j+1-k}^{L}(\tau )+a_{22}^{(k)}(0)\widehat{v}_{i-k,j+1-k}^{L}(\tau
)\right\} d\tau dt.  \label{56}
\end{equation}

The following theorem establishes the regularity of the functions 
$\widetilde{u}_{i,j}^{L}$, $\widetilde{v}_{i,j}^{L}$, 
$\widehat{u}_{i,j}^{L}$, $\widehat{v}_{i,j}^{L}.$

\begin{thm}
\label{thm11}
Assume that $f$, $g$, and ${\mathbf A}$ satisfy (\ref{3}) 
and (\ref{eq:A-positive-definite}). 
Let $\widetilde{u}_{i,j}^{L}$, $\widetilde{v}_{i,j}^{L}$,
$\widehat{u}_{i,j}^{L}$, $\widehat{v}_{i,j}^{L}$ 
be defined recursively as above, i.e., they solve 
(\ref{40}), (\ref{eq:41-42}), (\ref{43}), (\ref{eq:42a}), (\ref{44})
for the case $j= 0$ 
and, for $j \ge 1$ 
(\ref{51}), (\ref{eq:52-53}), (\ref{eq:54-55}), (\ref{56}). 
Set 
\begin{eqnarray*}
\overline{a}&:=&
\max\left\{ a_{11}(0), \ \frac{a_{11}(0)a_{22}(0)-a_{21}(0)a_{12}(0)}{a_{11}(0)}\right\} > 0,\\
\underline{a}&:=&
\min\left\{ a_{11}(0), \ \frac{a_{11}(0)a_{22}(0)-a_{21}(0)a_{12}(0)}{a_{11}(0)}\right\} > 0, \\
\decay{z} &:=& 
\begin{cases}
e^{-\underline{a} \operatorname*{Re}(z)} & \mbox{ for $\operatorname*{Re}(z) \ge 0$} \\
e^{-\overline{a} \operatorname*{Re}(z)} & \mbox{ for $\operatorname*{Re}(z) < 0$.} 
\end{cases}
\end{eqnarray*}
Then the functions $\widetilde u_{ij}^L$, $\widetilde v_{ij}^L$, 
$\widehat u_{ij}^L$, $\widehat v_{ij}^L$ are entire functions and 
there exist positive constants 
$C_{\widetilde{u}}$, $C_{\widetilde{v}}$,
$C_{\widehat{u}}$, $C_{\widehat{v}}$, 
$C_{i}$, $K_{i}$, $\overline{K}_{i}$, $i=1,...,4$,
independent of $\varepsilon $ and $\mu $ such that
\begin{eqnarray}
\left\vert \widetilde{u}_{i,j}^{L}(z)\right\vert &\leq & C_{1}{K}%
_{1}^{i}\overline{K}_{1}^{j}\left( C_{\widetilde{u}}(i+j)+\left\vert z\right\vert
\right) ^{2(i+j)}\frac{1}{(i+j)!} \decay{z},
\label{57}\\
\left\vert \widetilde{v}_{i,j}^{L}(z)\right\vert &\leq & C_{2}{K}%
_{2}^{i}\overline{K}_{2}^{j}\left( C_{\widetilde{v}}(i+j)+\left\vert z\right\vert
\right) ^{2(i+j)}\frac{1}{(i+j)!} \decay{z},
\label{58}
\\
\left\vert \widehat{u}_{i,j}^{L}(z)\right\vert &\leq & C_{3}{K}%
_{3}^{i}\overline{K}_{3}^{j}\left( C_{\widehat{u}}(i+j)+\left\vert z\right\vert \right)
^{2(i+j)}\frac{1}{(i+j)!} \decay{z},
\label{59}
\\
\left\vert \widehat{v}_{i,j+2}^{L}(z)\right\vert &\leq & C_{4}{K}%
_{4}^{i}\overline{K}_{4}^{j}\left( C_{\widehat{v}}(i+j)+\left\vert z\right\vert \right)
^{2(i+j)}\frac{1}{(i+j)!} \decay{z};
\label{60}
\end{eqnarray}
furthermore, $\widehat v^L_{i,0} = \widehat v^L_{i,1} \equiv  0$.  
\end{thm}

\begin{proof}
The proof is by induction on $j$ and $i$. After establishing the claims for the
base cases $(i,j) = (0,0)$, $(i,j) \in \{0\} \times \N$, $(i,j) \in \N \times \{0\}$ 
one shows it by induction on $j$ with induction arguments on $i$ as parts 
of the induction argument in $j$. 
The structure of the equations defining 
$\widetilde u_{i,j}^L$, 
$\widetilde v_{i,j}^L$, 
$\widehat u_{i,j}^L$, 
$\widehat v_{i,j+2}^L$, 
is such that one can proceed successively in the induction argument
on $j$ by providing estimates for $\widehat v_{i,j+2}$, $\widetilde v_{i,j}$, 
$\widetilde u_{i,j}^L$, $\widehat u_{i,j}^L$ in turn. In these estimates 
estimates, one relies on Lemma~\ref{lemma:scalar-bvp-constant-coefficients}
for the estimates for $\widetilde v_{i,j}^L$, $\widehat u_{i,j}^L$, 
on Lemma~\ref{lemma:two-anti-derivatives} for $\widehat v_{i,j+2}^L$, 
and on Lemma~\ref{lemma:derivatives-of-entire-fcts} for 
$\widetilde u_{i,j}^L$. 

\iftechreport
For details, see Appendix~\ref{appendix:case_4_thm11}.
\else
For details, see \cite[Appendix~\ref{appendix:case_4_thm11}]{mxo-1}
\fi
%
%
\end{proof}
We conclude this section by showing that the boundary layer functions
are in fact entire: 
\begin{cor}
\label{lemma:bdy-layer-fct-entire}
The functions $\widetilde{\mathbf U}^L_{i,j}$ and $\widehat {\mathbf U}^L_{i,j}$
are entire functions, and there exist constants $C$, $\gamma_1$, $\gamma_2$, $\beta > 0$
independent of $i$, $j$, $n$, such that for all $x \ge 0$
\begin{eqnarray*}
|\widetilde u_{i,j}^{(n)}(x) | + 
|\widetilde v_{i,j}^{(n)}(x)| + 
|\widehat u_{i,j}^{(n)}(x)|+
|\widehat v_{i,j+2}^{(n)}(x)| 
\leq C e^{-\beta x}\gamma_1^{i+j} (i+j)^{i+j} \gamma_2^n \qquad \forall n \in \N_0. 
\end{eqnarray*}
\end{cor}
\begin{proof}
Theorem~\ref{thm11} already asserted that the boundary layer functions 
are entire. For the stated bound, let $n \in \N_0$, $x \in (0,\infty)$ 
and use Cauchy's integral theorem 
for derivatives with contour $\partial B_{n+1}(x)$. We illustrate the procedure
for $\widehat u_{i,j}$, the other cases being similar. Theorem~\ref{thm11} 
then yields  suitable constants $C$, $\gamma > 0$ such that
$$
|\widehat u_{i,j}^{(n)}(x)| 
\leq C \frac{n!}{(n+1)^n} \gamma^{i+j} \frac{1}{(i+j)!}(C_{\widehat u} (i+j) + x + n+1)^{2(i+j)} \decay{x} e^{\oa (n+1)}.
$$
With the aid of Lemma~\ref{lemma:elementary-properties-of-factorial}, we obtain by
suitably adjusting $C$ and $\gamma$, 
$$
|\widehat u_{i,j}^{(n)}(x)| 
\leq C \frac{n!}{(n+1)^n} \gamma^{i+j} \frac{1}{(i+j)!}(C_{\widehat u} (i+j)+ n+1)^{2(i+j)} 
e^{\oa (n+1)}e^{-\ua x}.
$$
Using the observation $((i+j) +n+1)^{2(i+j)} \leq (i+j)^{2(i+j)} (1 + (n+1)/(i+j))^{2(i+j)}
\leq (i+j)^{2(i+j)} e^{2(n+1)}$, allows us to conclude the proof. 
\end{proof}
Corollary~\ref{lemma:bdy-layer-fct-entire} shows that the terms defining 
the boundary layer contributions 
$\widetilde {\mathbf U}^M_{BL}$ and 
$\widehat {\mathbf U}^M_{BL}$ are indeed of boundary layer type. A summation
argument then shows that also 
$\widetilde {\mathbf U}^M_{BL}$ and 
$\widehat {\mathbf U}^M_{BL}$ have this property provided $M_1$ and $M_2$
are not ``too large'', viz., $M_1 = O(\mu^{-1})$ 
and $M_2 = O((\mu/\varepsilon)^{-1})$:

\begin{thm}
\label{thm:case-IV-bdy-layer-fct-entire}
There exist constants $C$, $\delta$, $\gamma$, $K> 0$, independent of $\varepsilon$ and $\mu$, such that 
under the assumptions
$\mu (M_1+1) K \leq 1$ and $\varepsilon/\mu (M_2+1) K\leq 1$, there holds 
for the boundary layer functions $\widetilde{\mathbf U}^M_{BL}$ and 
$\widehat{\mathbf U}^M_{BL}$ of (\ref{BL_Mc1}) and (\ref{BL_Mc2}), that, 
upon viewing $\widetilde{\mathbf U}^M_{BL}$ and 
$\widehat{\mathbf U}^M_{BL}$ as functions of $x$ (via the 
changes of variables $\widetilde x = x/\mu$, $\widehat x = x/\varepsilon$ 
etc.), we have 
$\widetilde {\mathbf U}^M_{BL} \in \BLinf(\delta \mu,C,\gamma)$ and 
$\widehat {\mathbf U}^M_{BL} \in \BLinf(\delta \varepsilon,C,\gamma)$. 
Furthermore, the second
component $\widehat v$ of $\widehat {\mathbf U}^M_{BL}$ satisfies the stronger assertion 
$\widehat v \in \BLinf(\delta \varepsilon,C (\varepsilon/\mu)^2,\gamma)$. 
\end{thm}
\begin{proof}
We do not work out the details here since structurally 
similar arguments can be found, for example, 
in the proofs of \cite[Thm.~3]{m} or \cite[Thms.~{7.2.2}, {7.3.3}]{mB}. 
Essentially, by inserting the bounds of 
Corollary~\ref{lemma:bdy-layer-fct-entire}
in the sums defining $\widetilde{\mathbf U}^M_{BL}$, 
$\widehat{\mathbf U}^M_{BL}$ and using the conditions 
$\mu(M_1+1) K \leq 1$ and 
$\varepsilon/\mu (M_2+1) K \leq 1$ for $K$ sufficiently large, one obtains 
upper estimates in the form of (convergent) geometric series. We point 
out that the sharper estimates for the second component $\widehat v$ of 
$\widehat{\mathbf U}^M_{BL}$ stems from the fact that 
$\widehat v_{i,0} = \widehat v_{i,1} = 0$. 
\end{proof}
\subsection{Remainder estimates}
In this section, we analyze ${\mathbf R}_M$. This is done by estimating 
the residual $L_{\varepsilon,\mu} {\mathbf R}_M $ and then 
appealing to the stability estimate
(\ref{eq:a-priori}). We will estimate $L_{\varepsilon,\mu} {\mathbf W}_M - {\mathbf F}$, 
$L_{\varepsilon,\mu}\widetilde{\mathbf U}^M_{BL}$,  and $L_{\varepsilon,\mu}\widehat{\mathbf U}^M_{BL}$.  
\subsubsection{Remainder resulting from the outer expansion: control
of $L_{\varepsilon,\mu} {\mathbf W}_M - {\mathbf F}$}
\begin{thm}
\label{thm_case4_smooth_part}
Let $\mathbf{U}$ 
 be the solution to the problem \emph{(\ref{eq:model-problem})}. 
Then there exist $\gamma$, $C > 0$ depending only on $f$, $g$, and 
${\mathbf A}$ such that the following is true: 
If $M_1$, $M_2 \in \N$ are such that 
$\mu M_{1} \gamma <1$,  then with 
$\mathbf{W}_{M}$ given by \emph{(\ref{wMc})} we have 
\begin{equation*}
\left\Vert L_{\varepsilon,\mu }\left( \mathbf{U}-\mathbf{W}_{M}\right)
\right\Vert_{L^\infty(I)}\leq C \mu^2 \left[ \frac{1}{\mu-\varepsilon}\left( \mu M_{1} \gamma
\right) ^{M_{1}} + \left(\frac{\varepsilon}{\mu}\right)^{M_2+2}
\right]
\end{equation*}
\end{thm}
\begin{proof} We have
\begin{eqnarray*}
L_{\varepsilon,\mu }\left( \mathbf{U}-\mathbf{W}_{M}\right) &=&\left( 
\begin{array}{c}
f \\ 
g%
\end{array}%
\right) -\sum_{i=0}^{M_{1}}\sum_{j=0}^{M_{2}}\mu ^{i}\left( \frac{%
\varepsilon }{\mu }\right) ^{j}L_{\varepsilon,\mu }\left( 
\begin{array}{c}
u_{ij} \\ 
v_{ij}%
\end{array}%
\right) \\
&=&\left( 
\begin{array}{c}
f \\ 
g%
\end{array}%
\right) -\sum_{i=0}^{M_{1}}\sum_{j=0}^{M_{2}}\mu ^{i}\left( \frac{%
\varepsilon }{\mu }\right) ^{j}\left( 
\begin{array}{c}
-\varepsilon ^{2}u_{ij}^{\prime \prime }+a_{11}u_{ij}+a_{12}v_{ij} \\ 
-\mu ^{2}v_{ij}^{\prime \prime } + a_{21}u_{ij}+a_{22}v_{ij}%
\end{array}%
\right) .
\end{eqnarray*}
Defining the sets 
\begin{eqnarray}
\label{eq:Iu}
I_u &:= & \{(i,j) \colon i \leq M_1, \quad j \leq M_2, \quad 
i \ge M_1-1 \vee j \ge M_2 -1\}, \\
\label{eq:Iv}
I_v &:=&  \{(i,j) \colon i \leq M_1, \quad j \leq M_2, \quad 
M_1-1  \leq i \leq M_1\},
\end{eqnarray}
we see, after some calculations, that 
(\ref{eq:case-IV-recurrence-outer}) and Lemma \ref{lem10} imply 
\begin{eqnarray*}
L_{\varepsilon,\mu }\left( \mathbf{U}-\mathbf{W}_{M}\right) &=&\left( 
\begin{array}{c}
f \\ 
g%
\end{array}%
\right) -\sum_{i=0}^{M_{1}}\sum_{j=0}^{M_{2}}\mu ^{i}\left( \frac{%
\varepsilon }{\mu }\right) ^{j}\left\{ \left( 
\begin{array}{c}
f_{ij} \\ 
g_{ij}%
\end{array}%
\right) +\left( 
\begin{array}{c}
u_{i-2,j-2}^{\prime \prime }-\varepsilon ^{2}u_{ij}^{\prime \prime } \\ 
v_{i-2,j}^{\prime \prime }-\mu ^{2}v_{ij}^{\prime \prime }%
\end{array}%
\right) \right\} \\
&=&\left( 
\begin{array}{c}
\sum_{(i,j) \in I_u} \mu^{i+2} (\varepsilon/\mu)^{j+2} u_{i,j}^{\prime\prime}  \\
\sum_{(i,j) \in I_v} \mu^{i+2} (\varepsilon/\mu)^j v_{i,j}^{\prime\prime}  
\end{array}%
\right) .
\end{eqnarray*}
Hence, with the aid of Lemma~\ref{lem10} and Cauchy's Integral Theorem
for derivatives, we get for a fixed $\delta > 0$ in the statement of Lemma~\ref{lem10},
\begin{eqnarray*}
\|L_{\varepsilon,\mu} ({\mathbf U} - {\mathbf W}_M)\|_{L^\infty(I)} &\leq& 
C C_S \Bigl[
\mu^2 \sum_{i=0}^{M_1-2} (\mu i\delta^{-1} K)^{i} \sum_{j=M_2-1}^{M_2} (\varepsilon/\mu)^{j+2} + 
\mu^2 \sum_{i=M_1-1}^{M_1} (\mu i\delta^{-1} K)^{i} \sum_{j=0}^{M_2} (\varepsilon/\mu)^{j+2} \\
&&\mbox{} + 
\mu^2 \sum_{i=M_1-1}^{M_1} (\mu i\delta^{-1} K)^{i} \sum_{j=0}^{M_2} (\varepsilon/\mu)^j 
\Bigr].
\end{eqnarray*}
Hence, by selecting $\gamma = \delta^{-1} K /2$ we get 
\begin{eqnarray*}
\|L_{\varepsilon,\mu} ({\mathbf U} - {\mathbf W}_M)\|_{L^\infty(I)} 
&\leq& 
C C_S \Bigl[\mu^2 \left(\frac{\varepsilon}{\mu}\right)^{M_2+1} + 
(\mu M_1 \gamma)^{M_1-1}  \frac{1}{1-\varepsilon/\mu}
\Bigr] \\
&\leq& 
C C_S \mu^2 \Bigl[ \left(\frac{\varepsilon}{\mu}\right)^{M_2+1} + 
(\mu M_1 \gamma)^{M_1}  \frac{1}{\mu-\varepsilon}
\Bigr]. 
\end{eqnarray*}
\end{proof}
\subsubsection{Remainder resulting from the inner expansion 
on the $\varepsilon$-scale: $L_{\varepsilon,\mu} \widehat{\mathbf U}^M_{BL}$ }

We next consider the inner expansions. 
%
We will only consider the contribution $\widehat {\mathbf U}^M_{BL}$ from the left endpoint
as the contribution $\widehat {\mathbf V}^M_{BL}$ from the right endpoint is treated 
completely analogously. To simplify the notation, we drop the superscript $L$ in 
$\widehat u_{i,j}^L$, $\widehat v_{i,j}^L$. 

%
%
In order to simplify the ensuing calculations, 
we employ the convention that 
\begin{equation}
\label{eq:convention3}
\widehat{u}_{ij}=\widehat{v}_{ij}=0\mbox{  for $i > M_1$ or $j > M_2$ and  }{%
\mathbf{A}}_{k}=0\quad \forall k<0, 
\end{equation}
and let the summation on $i$ and $j$ in the definition 
of $\widehat{\mathbf U}^M_{BL}$ run from $0$ to $\infty $. 
We recall that the differential operator $L_{\varepsilon,\mu}$ takes
the form (\ref{eq:L-on-hat-scale}) when applied to functions depending solely on $\widehat x$,
and compute $L_{\varepsilon,\mu} \widehat {\mathbf U}^M_{BL}$ 
(cf. (\ref{BL_Mc1}) for the definition of $\widehat{\mathbf U}^M_{BL}$): 
\begin{eqnarray*}
L_{\varepsilon,\mu }\widehat{{\mathbf{U}}}_{BL}^{M} &=&\sum_{i=0}^{\infty
}\sum_{j=0}^{\infty }\mu ^{i}\left(\frac{\varepsilon}{\mu} \right)^{j}\left( 
\begin{array}{c}
-\widehat{u}_{ij}^{\prime \prime } \\ 
-\frac{\mu ^{2}}{\varepsilon ^{2}}\widehat{v}_{ij}^{\prime \prime }%
\end{array}%
\right) +\sum_{i\geq 0}\sum_{j\geq 0}\sum_{k\geq 0}\mu ^{i}
\left(\frac{\varepsilon}{\mu}\right)^{j}
\left(\frac{\varepsilon}{\mu}\right )^{k}\mu ^{k}\widehat{x}^{k}{\mathbf A}_{k}\left( 
\begin{array}{c}
\widehat{u}_{ij} \\ 
\widehat{v}_{ij}%
\end{array}%
\right)  \\
&=&\sum_{i\geq 0}\sum_{j\geq 0}\mu ^{i}
\left(\frac{\varepsilon}{\mu}\right)^{j}\left\{ \left( 
\begin{array}{c}
-\widehat{u}_{ij}^{\prime \prime } \\ 
-\widehat{v}_{i,j+2}^{\prime \prime }%
\end{array}%
\right) +\sum_{k\geq 0}\widehat{x}^{k}{\mathbf A}_{k}\left( 
\begin{array}{c}
\widehat{u}_{i-k,j-k} \\ 
\widehat{v}_{i-k,j-k}%
\end{array}%
\right) \right\} ,
\end{eqnarray*}
where the fact that $\widehat{v}_{k,0}=\widehat{v}_{k,1}=0$ was used. We see
that \textquotedblleft equating like powers of $\mu $ and $\varepsilon /\mu $
\textquotedblright\ yields equation (\ref{eq:case-IV-recurrence-hat}), hence 
there will be no
contribution to the sums for when both $i=0,\ldots ,M_{1}$ and $j=0,\ldots ,M_{2}-2$.
Moreover, the convention (\ref{eq:convention3}) implies the following
restrictions on the sums: 
\begin{align}
\label{eq:LhatU^M-conditions-on-i-j-k}
k \leq \min \{i,j\}, 
\qquad 
i-k \leq M_{1}, 
\qquad 
j-k \leq M_{2}.
\end{align}
Thus, we have 
\begin{eqnarray}
\label{eq:LhatU^M}
\lefteqn{
L_{\varepsilon,\mu }\widehat{{\mathbf{U}}}_{BL}^{M}=}\\
\nonumber 
&& \sum_{\substack{ i,j:i\geq
M_{1}+1 \\ \text{or }j\geq M_{2}-1}}\mu ^{i}\left(\frac{\varepsilon}{\mu}\right)^j
\!\!\sum_{k=\max \{i-M_{1},j-M_{2}\}}^{\min \{i,j\}}\widehat{x}%
^{k}{\mathbf A}_{k}\left( 
\begin{array}{c}
\widehat{u}_{i-k,j-k} \\ 
\widehat{v}_{i-k,j-k}%
\end{array}%
\right) 
 - \sum_{i = 0}^{M_1} \sum_{j=M_2-1}^{M_2} \mu^i \left(\frac{\varepsilon}{\mu}\right)^j
          \left(\begin{array}{c} \widehat u_{i,j}^{\prime\prime} \\
 0 \end{array}\right).
\end{eqnarray}
Using the estimates (\ref{59}), (\ref{60}) for $\widehat u_{i,j}$, $\widehat v_{i,j}$ 
of Theorem~\ref{thm11}, one can estimate this triple sum to obtain 
the following result for the remainder on the positive real line: 
\begin{thm}
\label{thm_BL_3scales}
There exist $C$, $\gamma $, $\widetilde{K}>0$ depending only on ${\mathbf A}$, $f$, $g$ such
that under the assumptions 
\begin{equation*}
0 < \gamma \widehat{x}\varepsilon \leq 1 
\qquad \mbox{ and } \qquad 
\mu (M_1 +1) \gamma  \leq 1
\qquad \mbox{ and } \qquad 
\frac{\varepsilon}{\mu}  (M_2 +1) \gamma  \leq 1,
\end{equation*}
we have 
\begin{equation*}
\left|L_{\varepsilon,\mu }\widehat{{\mathbf{U}}}_{BL}^{M}(\widehat{x}%
)\right| \leq Ce^{-\ua \widehat{x}/2}\left( (\widetilde{K}%
(M_{1}+1)\mu )^{M_{1}}+(\widetilde{K}(M_{2}+1)\varepsilon /\mu
)^{M_{2}-1}\right) .
\end{equation*}
For $\gamma \widehat{x}\varepsilon >1$ we have, under the same conditions 
on $(M_{1}+1)\mu $ and $(M_{2}+1)\varepsilon /\mu $, that 
\begin{equation*}
\left| L_{\varepsilon,\mu }\widehat{{\mathbf{U}}}_{BL}^{M}(\widehat{x}%
)\right| \leq Ce^{-\ua \widehat{x}/2}.
\end{equation*}
\end{thm}

\begin{proof}
The proof relies on using the estimates (\ref{59}), (\ref{60}) that are 
available for $\widehat u_{ij}$ and $\widehat v_{ij}$. The triple 
sum in (\ref{eq:LhatU^M})
is estimated by using convexity of the function 
$k \mapsto \gamma ^{k}(i-k)^{i-k}(j-k)^{j-k}$ and 
$k \mapsto k^{k}\gamma ^{k}(i-k)^{i-k}(j-k)^{j-k}$
and by considering the two following two cases separately: 
\begin{eqnarray*}
(i &\geq &M_{1}+1\ \vee \ j\geq M_{2}+1)\ \wedge \ (i-M_{1}\leq j-M_{2})\
\wedge \ (j-M_{2}\leq k\leq i)\ \wedge \ (i\leq j),  \label{eq:case1} \\
(i &\geq &M_{1}+1\ \vee \ j\geq M_{2}+1)\ \wedge \ (i-M_{1}\leq j-M_{2})\
\wedge \ (j-M_{2}\leq k\leq j)\ \wedge \ (j\leq i).  \label{eq:case2}
\end{eqnarray*}
\iftechreport
For details, see Appendix~\ref{appendix:thm_BL_3scales}
\else
For details, see \cite[Appendix~\ref{appendix:thm_BL_3scales}]{mxo-1}
\fi
\end{proof}

\subsubsection{Remainder resulting from the inner expansion 
on the $\mu$-scale: $L_{\varepsilon,\mu} \widetilde{\mathbf U}^M_{BL}$ }
In a similar fashion we may treat the left inner (boundary layer) expansion
associated with $\widetilde{x}=x/\mu $ (cf. (\ref{BL_Mc1})),
\begin{equation}
\widetilde{\mathbf{U}}_{BL}^{M}(\widetilde{x})=\sum_{i=0}^{M_{1}}%
\sum_{j=0}^{M_{2}}\mu ^{i}\left( \frac{\varepsilon }{\mu }\right) ^{j}\left( 
\begin{array}{c}
\widetilde{u}_{ij}(\widetilde{x}) \\ 
\widetilde{v}_{ij}(\widetilde{x})%
\end{array}%
\right) ,  \label{eq:ansatz4}
\end{equation}
where we have dropped the superscript $L$ for notational convenience. We recall
from (\ref{eq:L-on-tilde-scale}) that for functions $\widetilde {\mathbf U}$
depending only on the variable $\widetilde x$, 
the differential operator $L_{\varepsilon,\mu}$
takes the form $L_{\varepsilon,\mu} \widetilde {\mathbf U} = 
- \mu^{-2} {\mathbf E}^{\varepsilon,\mu} \widetilde {\mathbf U}^{\prime\prime}  
+ \sum_{k=0}^\infty \mu^k \widetilde x^k {\mathbf A}_k \widetilde {\mathbf U}$. 
In order to simplify the ensuing calculations, we employ the convention that 
\begin{equation}
\widetilde{u}_{ij}=\widetilde{v}_{ij}=0\mbox{ for $i > M_1$ or $j > M_2$ and
}{\mathbf{A}}_{k}=0\quad \forall k<0,  \label{eq:convention4}
\end{equation}
and let the summation in (\ref{eq:ansatz4}) run from $0$ to $\infty$ for both 
$i$ and $j$. We calculate 
\begin{eqnarray*}
L_{\varepsilon, \mu }\widetilde{\mathbf{U}}_{BL}^{M} &=&\sum_{i=0}^{\infty
}\sum_{j=0}^{\infty }\mu ^{i}(\varepsilon /\mu )^{j}\left( 
\begin{array}{c}
-(\varepsilon/\mu) ^{2}\widetilde{u}_{ij}^{\prime \prime } \\ 
-\widetilde{v}_{ij}^{\prime \prime }%
\end{array}%
\right) +\sum_{i\geq 0}\sum_{j\geq 0}\sum_{k\geq 0}\mu ^{i}(\varepsilon /\mu
)^{j}\mu ^{k}\widetilde{x}^{k}{\mathbf A}_{k}\left( 
\begin{array}{c}
\widetilde{u}_{ij} \\ 
\widetilde{v}_{ij}%
\end{array}%
\right)  \\
&=&\sum_{i\geq 0}\sum_{j\geq 0}\mu ^{i}(\varepsilon /\mu )^{j}\left\{ \left( 
\begin{array}{c}
-\widetilde{u}_{i,j-2}^{\prime \prime } \\ 
-\widetilde{v}_{ij}^{\prime \prime }%
\end{array}%
\right) +\sum_{k\geq 0}\widetilde{x}^{k}{\mathbf A}_{k}\left( 
\begin{array}{c}
\widetilde{u}_{i-k,j} \\ 
\widetilde{v}_{i-k,j}%
\end{array}%
\right) \right\} ,
\end{eqnarray*}
where the convention $\widetilde{u}_{k,-2}=\widetilde{u}_{k,-1}=0$ was used.
As expected from the derivation of 
(\ref{eq:case-IV-recurrence-tilde}), 
\textquotedblleft equating like powers of $\mu $ and 
$\varepsilon /\mu $\textquotedblright\ yields eqn.(\ref{eq:case-IV-recurrence-tilde}), 
hence there will be no
contribution to the sums for when both $i=0,\ldots ,M_{1}$ and $j=0,\ldots ,M_{2}$. 
Moreover, the convention (\ref{eq:convention4}) implies the following
restrictions on the sums: 
\begin{eqnarray}
\label{eq:restrictions-on-i-j-case-tilde}
i \leq M_1 && 
\quad \mbox{ for the terms involving 
$\widetilde u_{i,j-2}^{\prime\prime}$, $\widetilde v_{i,j}^{\prime\prime}$},\\
0 \leq i-k\leq M_{1} && \quad \mbox{ for the sum on $k$},\\
j \leq M_2 && 
\quad \mbox{ for the terms involving 
$\widetilde u_{i,j}$, $\widetilde v_{i,j}$, $\widetilde v_{i,j}^{\prime\prime}$, $\widetilde u_{i-k,j}$, $\widetilde v_{i-k,j}$},\\
j \leq M_2 -2 && 
\quad \mbox{ for the terms involving 
$\widetilde u_{i,j-2}^{\prime\prime}$}.
\end{eqnarray}
Hence, we arrive at
\begin{eqnarray}
\label{eq:LwidetildeU}
\lefteqn{
L_{\varepsilon,\mu }\widetilde{\mathbf{U}}_{BL}^{M}=}\\
&& \sum_{i=M_1+1}^\infty
\sum_{j=0}^{M_2} 
\mu ^{i}\left(\frac{\varepsilon}{\mu}\right)^j
\sum_{k=i-M_{1}}^{i}\widetilde{x}^{k}{\mathbf A}_{k}\left( 
\begin{array}{c}
\widetilde{u}_{i-k,j} \\ 
\nonumber 
\widetilde{v}_{i-k,j}%
\end{array}%
\right) 
+ \sum_{i = 0}^{M_1} \sum_{j=M_2+1}^{M_2+2} \mu^i 
\left(\frac{\varepsilon}{\mu}\right)^j 
\left(\begin{array}{c} - \widetilde u_{i,j-2}^{\prime\prime} \\ 0 \end{array}\right).\qquad 
\end{eqnarray}
Using the bounds of Theorem~\ref{thm2}, these sums will be estimated, 
when $\widetilde x > 0$, in the following:  

\begin{thm}
\label{thm_BL_3scales(b)}There exists $C$, $\gamma $, $\widetilde{K}>0$ such
that under the assumption 
\begin{equation*}
0 < \gamma \widetilde x \mu \leq 1, \qquad \mbox{ and } \qquad 
\gamma \mu (M_1+1)\leq 1 \qquad \mbox{ and } \qquad \frac{\varepsilon}{\mu} \gamma (M_2+1) \leq 1,
\end{equation*}
we have 
\begin{equation*}
\left|  L_{\varepsilon,\mu }\widetilde{{\mathbf{U}}}_{BL}^{M}(\widetilde{x}%
)\right|  \leq Ce^{-\ua \widetilde{x}/2}\left( (\widetilde{K}%
(M_{1}+1)\mu )^{M_{1}+1}+(\widetilde{K}(M_{2}+1)\varepsilon /\mu
)^{M_{2}+1}\right) .
\end{equation*}
For $\widetilde \mu \widetilde x\gamma > 1$ and the same assumptions on $\mu (M_1+1)$ and 
$\varepsilon/\mu (M_2+1)$, we have 
$$
\left|  L_{\varepsilon,\mu }\widetilde{{\mathbf{U}}}_{BL}^{M}(\widetilde{x}) \right|
\leq C e^{-\ua \widetilde x/2}. 
$$
\end{thm}
\begin{proof}
The proof is split into two cases: for $\widetilde x \mu \gamma > 1$, 
one exploits the fact that $\widetilde{\mathbf U}^M_{BL}$ (and its 
derivatives) is exponentially small. For the converse case 
$\widetilde x \mu \gamma \leq 1$, one bounds the sums
(\ref{eq:LwidetildeU}). 
\iftechreport
For details, see Appendix~\ref{appendix:case_4_thm_BL_3scales(b)}
\else
For details, see \cite[\ref{appendix:case_4_thm_BL_3scales(b)}]{mxo-1}
\fi

\end{proof}
\subsection{Boundary mismatch of the expansion}
\begin{thm}
\label{thm:3scale-bdy-mismatch}
There exist constants $C$, $b$, $\gamma > 0$ such that under the assumptions
\begin{equation*}
\mu (M_1+1)\gamma \leq 1 \qquad \mbox{ and } \qquad \frac{\varepsilon}{\mu} (M_2+1)\gamma \leq 1,
\end{equation*}
one has 
$$
\|{\mathbf W}^M + \widetilde {\mathbf U}^M_{BL} + \widehat{\mathbf U}^M_{BL}
 + \widetilde {\mathbf V}^M_{BL} + \widehat{\mathbf V}^M_{BL}\|_{L^\infty(\partial I)} 
\leq  C \left[ e^{-b/\varepsilon} + e^{-b/\mu}\right]. 
$$
\end{thm}
\begin{proof}
For ${\mathbf R}_M(0)$ at the left endpoint, we note that ${\mathbf U}(0) = 0$ gives
$$
{\mathbf R}_M(0) = {\mathbf U}(0) - \left[ 
{\mathbf W}^M(0) + \widetilde{\mathbf U}^M_{BL}(0) + \widehat{\mathbf U}^M_{BL}(0) 
+  \widetilde{\mathbf V}^M_{BL}(1/\mu) + \widehat{\mathbf V}^M_{BL}(1/\varepsilon) 
\right].
$$
The condition (\ref{eq:case-IV-recurrence-bc}) for the boundary conditions of the 
left endpoint for the terms $\widehat {\mathbf U}_{i,j}$ and $\widetilde {\mathbf U}_{i,j}$,  
produces 
${\mathbf W}^M(0) + \widetilde{\mathbf U}^M_{BL}(0) + \widehat{\mathbf U}^M_{BL}(0)  = 0$. 
Hence, it remains to estimate 
$|\widetilde{\mathbf V}^M_{BL}(1/\mu)| + |\widehat{\mathbf V}^M_{BL}(1/\varepsilon) |$ 
which can be done based on Theorem~\ref{thm11}.  
\end{proof}
\subsection{Proof of Theorem~\ref{thm:case-IV}} 
\label{sec:proof-of-thm:case-IV}
From the estimates for the residual $L_{\varepsilon,\mu} {\mathbf R}_M$ of 
Theorems~\ref{thm_case4_smooth_part}, \ref{thm_BL_3scales}, \ref{thm_BL_3scales(b)}
we infer the existence of $q > 0$ such that under the assumption 
$$
\mu \leq q \qquad \mbox{ and } \qquad \frac{\varepsilon}{\mu} \leq q,
$$
the choice $M_1 \sim 1/\mu$ and $M_2 \sim \mu/\varepsilon$ yields 
$$
\|L_{\varepsilon,\mu} {\mathbf R}_M\|_{L^\infty(I)} \leq C 
\left[ e^{-b/\mu} + e^{-b\mu/\varepsilon}\right],
$$
where $C$, $b > 0$ are independent of $\mu$ and $\varepsilon$. Hence, by stability  
and Theorem~\ref{thm:3scale-bdy-mismatch}, we get 
that $\|{\mathbf R}_M\|_{E,I} \leq C \left[ e^{-b/\mu} + e^{-b \mu/\varepsilon}\right]$. 
The sharper result for the second component $\widehat v$ of $\widehat{\mathbf U}$ follows
from the fact that $\widehat v_{i,0} = \widehat v_{i,1} = 0$.

It remains to formulate a decomposition of ${\mathbf U}$ for the case 
that $\mu \ge q$ or $\varepsilon/\mu \ge q$. 
In this case, we have $e^{-b/\mu} + e^{-b\mu/\varepsilon}$ is $O(1)$. Given
that $\|{\mathbf U}\|_{E,I} = O(1)$, the trivial decomposition 
${\mathbf U} = 0 + 0 + 0 + {\mathbf R}_M$ provides the desired splitting. 

\section{The first two scale case: Case~\ref{case:II}}
\label{case_2}
In this case it is assumed that $\mu /1$ is $\emph{not}$ deemed small 
\emph{but} $\varepsilon /\mu $ is deemed small. 
The main result of this section is the following regularity assertion: 
\begin{thm}
\label{thm:case-II} 
There exist constants $C_W$, $\gamma_W$, $C_{BL}$, $\gamma_{BL}$, $\delta$, $b$, $q > 0$ 
independent of $\varepsilon$ and $\mu$ such that for $\varepsilon/\mu \leq q$ 
the solution ${\mathbf U}$ of (\ref{eq:model-problem}) can be 
written as ${\mathbf U} = {\mathbf W} + \widehat{\mathbf U}_{BL}  + {\mathbf R}$, where 
${\mathbf W} \in {\mathcal A}(\mu, C_W,\gamma_W)$, 
$\widehat {\mathbf U}_{BL} \in \BLinf(\delta \varepsilon, C_{BL},\gamma_{BL})$. 
Furthermore, ${\mathbf R}$ satisfies 
\begin{eqnarray*}
\|{\mathbf R}\|_{L^\infty(\partial I)} + 
\|L_{\varepsilon,\mu} {\mathbf R}\|_{L^\infty(I)} &\leq& C 
e^{-b/\varepsilon}. 
\end{eqnarray*}
In particular, $\|{\mathbf R}\|_{E,I} \leq C e^{-b/\varepsilon}$. 

Additionally, the second component $\widehat v$ of $\widehat {\mathbf U}_{BL}$  satisfies the sharper
regularity assertion 
$$
\widehat v \in \BLinf(\delta \varepsilon, C_{BL}(\varepsilon/\mu)^2, 
\gamma_{BL}).
$$
\end{thm}

We employ again the notation of the outset of Section~\ref{case_4} concerning
the stretched variables $\widehat x = x/\varepsilon$ and $\widehat x^R = (1-x)/\varepsilon$
and make the formal ansatz
\begin{equation}
\label{eq:ansatz-case-II} 
{\mathbf U}(x) \sim \sum_{i=0}^\infty \left(\frac{\varepsilon}{\mu}\right)^i 
\left[ {\mathbf U}_i(x) + \widehat {\mathbf U}^L_i(\widehat x) + 
\widehat {\mathbf U}^R_i(\widehat x^R)\right].
\end{equation}
We proceed as in Section~\ref{case_4} by inserting the ansatz (\ref{eq:ansatz-case-II})
in the differential equation (\ref{1}), separating the slow ($x$) and the 
fast ($\widehat x$ and $\widehat x^L$) variables and equating like powers of $\varepsilon/\mu$. 
We also recall that the differential operator $L_{\varepsilon,\mu}$ takes 
the form (\ref{eq:L-on-hat-scale}) on the $\widehat x$-scale. 
The separation of the slow and the fast variables leads to the following 
equations: 
\begin{subequations}
\label{eq:case_2-separation-of-scales}
\begin{eqnarray}
\sum_{i=0}^\infty \left(\frac{\varepsilon}{\mu}\right)^i 
\left( - {\mathbf E}^{\varepsilon,\mu} {\mathbf U}_i^{\prime\prime} + {\mathbf A}(x) {\mathbf U}_i
\right) &=& {\mathbf F} = \left(\begin{array}{c} f \\ g \end{array}\right), 
\label{eq:case_2-separation-of-scales-slow-variable}
\\
\sum_{i=0}^\infty \left(\frac{\varepsilon}{\mu}\right)^i 
\left( - \varepsilon^{-2} {\mathbf E}^{\varepsilon,\mu} 
(\widehat {\mathbf U}^L_i)^{\prime\prime} + 
\sum_{k=0}^\infty \mu^k \left( \frac{\varepsilon}{\mu}\right)^k \widehat x^k{\mathbf A}_k \widehat {\mathbf U}_i^L
\right) &=& 0,
\label{eq:case_2-separation-of-scales-fast-variable}
\end{eqnarray}
\end{subequations}
and an analogous equation for $\widehat {\mathbf U}^R_i$. Writing again 
$$
\widehat {\mathbf U}_i = \left(\begin{array}{c} u_i \\ v_i \end{array}\right), 
\qquad 
\widehat {\mathbf U}^L_i = \left(\begin{array}{c} \widehat u^L_i \\ \widehat v^L_i \end{array}\right), 
$$
we obtain from (\ref{eq:case_2-separation-of-scales}) by equating like powers 
of $\varepsilon/\mu$: 
\begin{subequations}
\label{eq:case_2-recursions}
\begin{eqnarray}
\label{eq:case_2-recursions-i}
- \mu^2 u_{i-2}^{\prime\prime} + a_{11} u_{i} + a_{12} v_{i} &=& f_i, \\
\label{eq:case_2-recursions-ii}
- \mu^2 v_{i}^{\prime\prime} + a_{21} u_{i} + a_{22} v_{i} &=& g_i, \\
\label{eq:case_2-recursions-iii}
- (\widehat u^L_i)^{\prime\prime} + 
\sum_{k=0}^i \mu^k \widehat x^k 
\left( \frac{a_{11}^{(k)}(0)}{k!} \widehat u^L_{i-k}  + \frac{a_{12}^{(k)}(0)}{k!} \widehat v^L_{i-k}\right) 
&=& 0, \\
\label{eq:case_2-recursions-iv}
- (\widehat v^L_{i+2})^{\prime\prime} + 
\sum_{k=0}^i \mu^k \widehat x^k 
\left( \frac{a_{21}^{(k)}(0)}{k!} \widehat u^L_{i-k}  + \frac{a_{22}^{(k)}(0)}{k!} \widehat v^L_{i-k}\right) 
&=& 0, 
\end{eqnarray}
\end{subequations}
where we employed the definition of ${\mathbf A}_k$, the notation $f_0 = f$, $g_0 = g$ as well as
$f_i = g_i = 0$ for $i > 0$, and the convention that function with negative subscripts are zero. 
Corresponding equations are satisfied the functions $\widehat u_i^R$ and $\widehat v^R_i$. 
The boundary condition (\ref{2}) is accounted for by stipulating 
${\mathbf U}_i(0) + \widehat {\mathbf U}_i^L(0) = 0$ and 
${\mathbf U}_i(1) + \widehat {\mathbf U}_i^R(0) = 0$ for all $i \ge 0$ and suitable decay conditions
for $\widehat {\mathbf U}_i^L$ 
as $\widehat x \rightarrow \infty $ and, correspondingly,  
for $\widehat {\mathbf U}_i^R$ 
as $\widehat x^R \rightarrow \infty $. Rearranging the above equations and incorporating these
boundary conditions, we get a a recursion of systems of DAEs in which the algebraic
constraints can be accounted for explicitly. We obtain for $i=0$, $1$, $2,\ldots$ :
\begin{equation}
\left\{ 
\begin{array}{c}
-\mu ^{2}v_{i}^{\prime \prime }+\frac{(a_{22}a_{11}-a_{12}a_{21})}{a_{11}}%
v_{i}=g_{i}-\frac{a_{21}}{a_{11}}\left( f_{i}+\mu^2 u_{i-2}^{\prime \prime }\right)
\\ 
v_{i}(0)=-\widehat{v}_{i}^{L}(0)\;,\;v_{i}(1)=-\widehat{v}_{i}^{R}(0)%
\end{array}%
\right. ,  \label{14}
\end{equation}
\begin{equation}
u_{i}=\frac{1}{a_{11}}\left( f_{i}+\mu^2 u_{i-2}^{\prime \prime
}-a_{12}v_{i}\right) ,  \label{15}
\end{equation}
\begin{equation}
\left\{ 
\begin{array}{c}
-\left( \widehat{u}_{i}^{L}\right) ^{\prime \prime }+a_{11}(0)\widehat{u}%
_{i}^{L}
=-\underset{k=1}{\overset{i}{\sum }}%
\left( \frac{a_{11}^{(k)}(0)}{k!}\mu^k \widehat{x}^{k}\widehat{u}_{i-k}^{L}+\frac{%
a_{12}^{(k)}(0)}{k!}\widehat{x}^{k}\mu^k \widehat{v}_{i-k}^{L}\right) 
-a_{12}(0)\widehat{v}_{i}^{L}
\\ 
\widehat{u}_{i}^{L}(0)=-u_{i}(0)\;,\;\widehat{u}_{i}^{L}\rightarrow 0\text{
as }\widehat{x}\rightarrow \infty%
\end{array}%
\right. ,  \label{16}
\end{equation}
\begin{equation}
- \left( \widehat{v}_{i+2}^{L}\right) ^{\prime \prime }=\underset{k=0}%
{\overset{i}{\sum }}\left( \frac{a_{21}^{(k)}(0)}{k!}\mu^k \widehat{x}^{k}\widehat{%
u}_{i-k}^{L}+\frac{a_{22}^{(k)}(0)}{k!}\mu^k \widehat{x}^{k}\widehat{v}%
_{i-k}^{L}\right) ,  \label{17}
\end{equation}
with
\begin{equation*}
\widehat{v}_{0}^R = \widehat{v}_{1}^{R} = 
\widehat{v}_{0}^{L}=\widehat{v}_{1}^{L}=0\;,\;u_{-i}=0\;\forall \;i>0,
\end{equation*}
\begin{equation*}
f_{i}=\left\{ 
\begin{array}{c}
f\text{ \ if }i=0, \\ 
0\text{ otherwise}%
\end{array}%
\right. \;,\;g_{i}=\left\{ 
\begin{array}{c}
g\text{ \ if }i=0, \\ 
0\text{ otherwise}%
\end{array}%
\right. .
\end{equation*}
(The functions $\widehat{u}_{i}^{R},\widehat{v}_{i}^{R}$ satisfy similar
problems as (\ref{16}) and (\ref{17}), respectively.) 
\subsection{Analysis of the functions ${\mathbf U}_i$, $\widehat {\mathbf U}^L_i$, 
$\widehat {\mathbf U}_i^R$}
\subsubsection{Properties of some solution operators}

\begin{lem}
\label{lem4}
Assume that the function $c$ is analytic on $\overline{I}$ and $c\geq 
\underline{c}>0$ $\forall \;z\in \overline{I}$. Let $\mu \in (0,1]$. 
Then there exists $\gamma_0 > 0$ (independent of $\mu)$ such that 
for all $\gamma \ge \gamma_0$ and all $C_g > 0$ the following 
is true: If $g$ satisfies 
\begin{equation}
\label{eq:lem4-1}
\|g^{(n)}\|_{L^\infty(I)} \leq C_g \gamma^n \max\{n,\mu^{-1}\}^n 
\qquad \forall n \in \N_0,
\end{equation}
then the solution $u$ of the boundary value problem 
\begin{equation*}
\left\{ 
\begin{array}{c}
-\mu ^{2}u^{\prime \prime }+cu=g\text{ in }I \\ 
u(0)=g_{-}\in \mathbb{R}\;,\;u(1)=g_{+}\in \mathbb{R}%
\end{array}%
\right.
\end{equation*}
satisfies
\begin{equation*}
\|u^{(n)}\|_{L^\infty(I)} 
\leq \widetilde C \gamma^n \max\{n,\mu^{-1}\}^n 
\left[ C_g + |g_{+}| + |g_{-}|\right] 
\qquad \forall n \in \N_0,
\end{equation*}
for a constant $\widetilde C$ that depends solely on the function $c$. 
\end{lem}

\begin{proof} 
This follows by inspection of the proof of 
\cite[Thm.~1]{m} if one replaces the energy type arguments
by an appeal to the comparison principle to get $L^\infty$-estimates. 

\iftechreport
For details, see Appendix~\ref{appendix:case_2_lem4}.
\else
For details, see \cite[Appendix~\ref{appendix:case_2_lem4}]{mxo-1}
\fi
\end{proof}
\begin{lem}
\label{lemma:leibniz-rule}
Let $m \in \N_0$ and $g$ be a function analytic on $\overline{I}$.  Then there exist
$C^\prime$, $\gamma_0 > 0$ depending only on $g$ such that the following 
is true: If $v$ satisfies 
$$
\|v^{(n)}\|_{L^\infty(I)} \leq C_v \gamma_v^n \max\{n,\mu^{-1}\}^n 
\qquad \forall n \in \N_0,
$$
for some $C_v$, $\mu > 0$ and $\gamma_v \ge \gamma_0$, then 
the function $V:= g v^{(m)}$ satisfies 
$$
\|V^{(n)}\|_{L^\infty(I)} \leq C^\prime C_v \gamma_v^{n +m}
\max\{n+m,\mu^{-1}\}^{n+m} \qquad 
\forall n \in \N_0. 
$$
\end{lem}
\begin{proof}
Since $g$ is analytic, there exist $C_g$, $\gamma_g > 0$ such that 
$$
\|g^{(n)}\|_{L^\infty(I)} \leq C_g \gamma_g^n n! \qquad \forall n \in \N. 
$$
By Leibniz' rule, we get for $\gamma_v \ge \gamma_0 > \gamma_g e$ 
in view of $\nu^\nu \leq \nu! e^\nu$: 
\begin{eqnarray*}
\|V^{(n)}\|_{L^\infty(I)} &\leq& C_v C_g \sum_{\nu=0}^n 
\binom{n}{\nu} \gamma_g^\nu \nu^\nu \gamma_v^{n+m-\nu} 
\max\{n+m-\nu,\mu^{-1}\}^{n+m-\nu}\\
&\leq& C_v C_g \gamma_v^n \sum_{\nu=0}^n \frac{n!}{(n-\nu)!} 
\left(\frac{ e \gamma_g}{\gamma_v}\right)^{\nu} \max\{n+m,\mu^{-1}\}^{n+m-\nu}\\
&\leq&  \frac{C_g}{1-e \gamma_g/\gamma_v} C_v \gamma_v^n 
\max\{n+m,\mu^{-1}\}^{n+m}. 
\end{eqnarray*}
\end{proof}
\begin{lem}
\label{lem5} 
For $0 < \widehat \delta \leq 1/(2e)$ and $i \in \N$ there holds 
\begin{equation*}
\overset{i}{\underset{k=0}{\sum }}\widehat{\delta }^{k}\left( \frac{i+2}{%
i+1-k}\right) ^{i-k} \leq 2e. 
\end{equation*}
\end{lem}

\begin{proof}
%
%
\iftechreport
See Appendix~\ref{appendix:proof_of_lem5}.
\else
See \cite[Appendix~\ref{appendix:proof_of_lem5}]{mxo-1}.
\fi
\end{proof}
\subsubsection{Regularity of the functions ${\mathbf U}_i$, $\widehat {\mathbf U}^L_i$, 
$\widehat {\mathbf U}_i^R$}
\begin{thm}
\label{thm2} Let $v_{i},u_{i},\widehat{u}_{i}^{L},\widehat{v}_{i}^{L}$, 
$\widehat{u}_{i}^{R}$, $\widehat{v}_{i}^{R}$ 
be
the solutions of \emph{(\ref{14})--(\ref{17})}, respectively. Let $G\subset 
\mathbb{C}$ be a complex neighborhood of $\overline{I}$ and set $\beta =%
\sqrt{a_{11}(0)}\in \mathbb{R}.$ Then the functions 
$\widehat u_i^L$ and $\widehat v_i^L$, 
$\widehat u_i^R$ and $\widehat v_i^R$ 
are entire, and the 
functions $u_i$, $v_i$ are analytic in a (fixed) neighborhood of $I$. 
Furthermore,  there exist positive constants 
$\gamma$, $C_{u},C_{v},C_{\ell },K_{\ell }$, $\ell =1,...,4$ independent of 
$\varepsilon $ and $\mu $, such that
\begin{equation}
\widehat{v}_{0}^{L}=\widehat{v}_{1}^{L}=0,
\end{equation}
\begin{equation}
\left\vert \widehat{v}_{i+2}^{L}(z)\right\vert \leq C_{1}K_{1}^{i}
\left( \mu+\frac{1}{i+1}\right) ^{i}\frac{1}{i!}\left(
iC_{v}+\left\vert z\right\vert \right) ^{2i}e^{-\beta \operatorname*{Re}(z)},\quad
\label{18}
\end{equation}
\begin{equation}
\left\vert \widehat{u}_{i}^{L}(z)\right\vert \leq C_{2}K_{2}^{i}\left( \mu+%
\frac{1}{i+1}\right) ^{i}\frac{1}{i!}\left( iC_{u}+\left\vert z\right\vert
\right) ^{2i}e^{-\beta \operatorname*{Re}(z)}, 
\label{19}
\end{equation}
\begin{equation}
\|u_{i}^{(n)}\|_{L^\infty(I)} \leq C_{3}K_{3}^{i}
\left(  (i+n)\mu+%
1\right) ^{i} \frac{i^i}{i!}\gamma^n \max\{n,\mu^{-1}\}^n \qquad \forall n \in \N_0, 
\quad \forall i \in \N_0, 
\label{20}
\end{equation}
\begin{equation}
\|v_{i}^{(n)}\|_{L^\infty(I)} \leq C_{4}K_{4}^{i}\left( (i+n)\mu+%
1\right) ^{i} \frac{i^i}{i!}\gamma^n \max\{n,\mu^{-1}\}^n \qquad \forall n \in \N_0, 
\quad \forall i \in \N_0. 
\label{21}
\end{equation}
Furthermore, analogous estimates hold for $\widehat u_i^R$ and 
$\widehat v_i^R$ with $\beta$ being now $\sqrt{a_{11}(1)}$. 
\end{thm}

\begin{proof} 
The proof is by induction on $i$. The general structure is 
to estimate $v_i$, $u_i$, $\widehat u_i$, and $\widehat v_{i+2}$ 
in turn using Lemmas~\ref{lemma:scalar-bvp-constant-coefficients}, 
\ref{lemma:two-anti-derivatives}, \ref{lem4}, \ref{lemma:leibniz-rule}. 
\iftechreport
For details, see Appendix~\ref{appendix:case_2_thm2}.
\else
For details, see \cite[Appendix~\ref{appendix:case_2_thm2}]{mxo-1}. 
\fi
\end{proof}
We conclude this section by showing that the boundary layer functions
are in fact entire:
\begin{lem}
\label{lemma:bdy-layer-fct-entire-case-II}
Let $\beta > 0$ be as in Theorem~\ref{thm2} (i.e., $\beta = \sqrt{a_{11}(0)}$).
The functions $\widehat{\mathbf U}^L_{i}$ 
are entire functions and there exist constants $C$, $\gamma_1$, $\gamma_2 >0$
independent of $i$, $j$, $n$ such that for all $x \ge 0$
\begin{eqnarray*}
|\widehat u_{i}^{(n)}(\widehat x)|+
|\widehat v_{i+2}^{(n)}(\widehat x)| 
\leq C e^{-\beta/2 \widehat x}\gamma_1^{i} 
(i\mu + 1)^i \gamma_2^n \qquad \forall n \in \N_0. 
\end{eqnarray*}
\end{lem}
\begin{proof}
Similar to the proof of Corollary~\ref{lemma:bdy-layer-fct-entire}. 
\end{proof}
By a simple summation argument we get from Lemma~\ref{lemma:bdy-layer-fct-entire-case-II}
\begin{thm}
\label{thm:bdy-layer-fct-entire-case-II}
Let $\beta$ be as in Theorem~\ref{thm2} (i.e., $\beta = \sqrt{a_{11}(0)}$).
There exist constants $C$, $\gamma$, $K >0$ such that for 
$\gamma\{ (M+1)\varepsilon + \frac{\varepsilon}{\mu}\} \leq 1$ the 
function 
$$
\widehat{\mathbf U}^M_{BL}(\widehat x):= \sum_{i=0}^M 
\left(\frac{\varepsilon}{\mu}\right)^i \widehat {\mathbf U}_i(\widehat x),
$$
satisfies for $\widehat x > 0$ 
$$
|\left(
\widehat{\mathbf U}^M_{BL}\right)^{(n)} (\widehat x)
| \leq C e^{-\beta/2 \widehat x} K^n 
\qquad \forall n \in \N_0. 
$$
An analogous result holds for $\widehat{\mathbf V}^M_{BL}:= \sum_{i=0}^M
(\varepsilon/\mu)^i \widehat {\mathbf U}_i^{R}$ with $\beta$ given by
$\beta = \sqrt{a_{11}(1)}$.
\end{thm}
\begin{proof}
\iftechreport 
See Appendix~\ref{appendix:proof_of_thm:bdy-layer-fct-entire-case-II} 
for details. 
\else 
See \cite[Appendix~\ref{appendix:proof_of_thm:bdy-layer-fct-entire-case-II}]{mxo-1} for details. 
\fi
\end{proof}
\subsection{Remainder estimates}
We now turn to estimating the remainder obtained by truncating 
the formal expansion (\ref{eq:ansatz-case-II}). We write 
\begin{equation}
\mathbf{U}(x)=\left( 
\begin{array}{c}
u(x) \\ 
v(x)%
\end{array}%
\right) =\mathbf{W}_{M}(x)+\widehat{\mathbf{U}}_{BL}^{M}(\widehat{x})
+\widehat{\mathbf{V}}%
_{BL}^{M}(\widehat{x})+\mathbf{R}_{M}(x),  \label{U_oir}
\end{equation}
where
\begin{equation}
\mathbf{W}_{M}(x)=\sum_{i=0}^{M}\left(\frac{\varepsilon}{\mu}\right) ^{i}\left( 
\begin{array}{c}
u_{i}(x) \\ 
v_{i}(x)%
\end{array}%
\right) ,  \label{wM}
\end{equation}
denotes the outer (smooth) expansion,
\begin{equation}
\widehat{\mathbf{U}}_{BL}^{M}(\widehat{x})=
\sum_{i=0}^{M}\left(\frac{\varepsilon}{\mu}\right)^{i}\left( 
\begin{array}{c}
\widehat{u}_{i}^{L}(\widehat{x}) \\ 
\widehat{v}_{i}^{L}(\widehat{x})%
\end{array}%
\right) \;,\;\widehat{\mathbf{V}}_{BL}^{M}(\widehat{x})=\sum_{i=0}^{M}
\left(\frac{\varepsilon}{\mu}\right)^{i}\left( 
\begin{array}{c}
\widehat{u}_{i}^{R}(\widehat{x}^R) \\ 
\widehat{v}_{i}^{R}(\widehat{x}^R)%
\end{array}%
\right) ,  \label{BL_M}
\end{equation}
denote the left and right inner (boundary layer) expansions, respectively,
and
\begin{equation}
\mathbf{R}_{M}(x):=\left( 
\begin{array}{c}
r_{u}(x) \\ 
r_{v}(x)%
\end{array}%
\right) =\mathbf{U}(x)-\left( \mathbf{W}_{M}(x)
+\widehat{\mathbf{U}}_{BL}^{M}(\widehat{%
x})+\widehat{\mathbf{V}}_{BL}^{M}(\widehat{x}^R)\right)   \label{rM}
\end{equation}
denotes the remainder. Our goal is to show that the remainder 
$\mathbf{R}_{M}(x)$ 
is exponentially (in $\varepsilon$) small. First, we need to
obtain results on the other terms in (\ref{U_oir}). Note that by the
linearity of the operator $L_{\varepsilon,\mu }$ we have
\begin{equation*}
L_{\varepsilon,\mu }\mathbf{R}_{M}=
L_{\varepsilon,\mu }\left( \mathbf{U}-\mathbf{W}%
_{M}\right) -L_{\varepsilon,\mu }\widehat{\mathbf{U}}_{BL}^{M}
-L_{\varepsilon,\mu }\widehat{\mathbf{V}}%
_{BL}^{M}.
\end{equation*}
The terms on the right-hand side are treated separately. 
%
\subsubsection{Remainder resulting from the outer expansion: 
${\mathbf F} - L_{\varepsilon,\mu} {\mathbf W}_M$}
We have the following theorem.

\begin{thm}
\label{thm_case2_smooth_part}Let $\mathbf{U}$ 
be the solution to the problem \emph{(\ref{eq:model-problem})}. 
Then there exist $C$, $\gamma > 0$ independent of $\mu$, $\varepsilon$, 
and $M$ such that for $\mathbf{W}_{M}$,
given by \emph{(\ref{wM})} we have 
\begin{equation*}
\left\Vert L_{\varepsilon,\mu }\left( \mathbf{U}-\mathbf{W}_{M}\right)
\right\Vert_{L^\infty(I)} \leq C \left[
\left(\gamma \frac{\varepsilon}{\mu}\right)^{M+1} + 
\left(\gamma (M +1)\varepsilon\right)^{M+1}
\right].
\end{equation*}
\end{thm}

\begin{proof} 
Using 
(\ref{eq:case_2-separation-of-scales-slow-variable})
(\ref{14}), (\ref{15}) we obtain, after some calculations, 
\begin{eqnarray*}
\left\Vert L_{\varepsilon,\mu }\left( \mathbf{U}-\mathbf{W}_{M}\right)
\right\Vert_{L^\infty(I)} &=&
\sum_{i=M-1}^M \left(\frac{\varepsilon}{\mu}\right)^{i+2} \mu^2 
\left(\begin{array}{c} u_{i}^{\prime\prime} \\ 0 \end{array}\right)
\end{eqnarray*}
From Theorem~\ref{thm2} we therefore get, with the observation 
$
\frac{\varepsilon}{\mu} ( M \mu + 1)  = M\varepsilon + \frac{\varepsilon}{\mu},
$
the desired result. 
\end{proof}
\subsubsection{Remainder resulting from the inner expansion: 
$L_{\varepsilon,\mu} \widehat{\mathbf U}^M_{BL}$}
We now turn our attention to estimating 
$L_{\varepsilon,\mu} \widehat {\mathbf U}^M_{BL}$. Since 
$L_{\varepsilon,\mu} \widehat {\mathbf V}^M_{BL}$ is treated with
similar arguments, we will not work out the details. We have: 
\begin{lem} 
The functions $\widehat{\mathbf U}^M_{BL}$ satisfy 
\begin{equation}
\label{eq:remainder-case-II}
L_{\varepsilon,\mu} \widehat {\mathbf U}^M_{BL} = 
\sum_{i \ge M+1} 
\left(\frac{\varepsilon}{\mu}\right)^i 
\sum_{k = i-M}^{i} 
\mu^k \widehat x^k {\mathbf A}_k \widehat {\mathbf U}_{i-k}^L
+ 
\sum_{i=M-1}^M 
\left(\frac{\varepsilon}{\mu}\right)^i 
\left( 
\begin{array}{c} 
0 \\ (\widehat v^L_{i+2})^{\prime\prime}
\end{array}\right)  . 
\end{equation}
\end{lem}
\begin{proof}

In order 
to simplify the calculation of 
$L_{\varepsilon,\mu} \widehat{\mathbf U}^M_{BL}$, 
we employ temporarily the convention that 
\begin{equation}
\label{eq:convention-case-II}
\widehat u_i^L = \widehat v_i^L = 0 \qquad \forall i \ge M+1. 
\end{equation} 
With this convention, we calculate (cf.
(\ref{eq:case_2-recursions-iii}), 
(\ref{eq:case_2-recursions-iv}))
\begin{eqnarray*}
L_{\varepsilon,\mu} \widehat {\mathbf U}^M_{BL}
&=& \sum_{i=0}^\infty \left(\frac{\varepsilon}{\mu}\right)^i 
\left(\begin{array}{c}
- (\widehat u^L_i)^{\prime\prime} + 
\sum_{k=0}^i \mu^k \widehat x^k 
\left( \frac{a_{11}^{(k)}(0)}{k!} \widehat u^L_{i-k}  + \frac{a_{12}^{(k})(0)}{k!} \widehat v^L_{i-k}\right) 
\\
- (\widehat v^L_{i+2})^{\prime\prime} + 
\sum_{k=0}^i \mu^k \widehat x^k 
\left( \frac{a_{21}^{(k)}(0)}{k!} \widehat u^L_{i-k}  + \frac{a_{22}^{(k)}(0)}{k!} \widehat v^L_{i-k}\right) 
\end{array}
\right)\\
&=& 
\sum_{i \ge M+1} 
\left(\frac{\varepsilon}{\mu}\right)^i 
\sum_{k = 0}^{\infty} 
\mu^k \widehat x^k {\mathbf A}_k \widehat {\mathbf U}_{i-k}^L
- 
\sum_{i=M-1}^M 
\left(\frac{\varepsilon}{\mu}\right)^i 
\sum_{k=0}^i\mu^k \widehat x^k 
\left( 
\begin{array}{c} 
0 \\
\frac{a_{21}^{(k)}(0)}{k!} \widehat u^L_{i-k}  + \frac{a_{22}^{(k)}(0)}{k!} \widehat v^L_{i-k}
\end{array}\right)  ; 
\end{eqnarray*}
here, we employed additionally 
(\ref{eq:case_2-recursions-iii}), 
(\ref{eq:case_2-recursions-iv}) to see that the terms corresponding 
to $i \in \{0,\ldots,M-2\}$ in the first sum are zero. Finally, 
our convention (\ref{eq:convention-case-II}) and the fact that 
$\widehat u^L_j = \widehat v^L_j = 0$ for $j < 0$ implies 
the restrictions 
$$
0 \leq i - k \leq M
$$
so that we obtain 
$$
\sum_{i \ge M+1} 
\left(\frac{\varepsilon}{\mu}\right)^i 
\sum_{k = 0}^{\infty} 
\mu^k \widehat x^k {\mathbf A}_k \widehat {\mathbf U}_{i-k}^L
= \sum_{i \ge M+1} 
\left(\frac{\varepsilon}{\mu}\right)^i 
\sum_{k = i-M}^{i} 
\mu^k \widehat x^k {\mathbf A}_k \widehat {\mathbf U}_{i-k}^L, 
$$
which produces the double sum in (\ref{eq:remainder-case-II}). 
Lifting now the convention (\ref{eq:convention-case-II})
we can use (\ref{eq:case_2-recursions-iii}) 
to replace 
$\sum_{k=0}^i \mu^k \widehat x^k 
(\frac{a_{21}^{(k)}(0)}{k!} \widehat u^L_{i-k}  
+ \frac{a_{22}^{(k)}(0)}{k!} \widehat v^L_{i-k})$ with 
$-\widehat v^L_{i+2}(\widehat x)$. 
\end{proof}
It remains to bound 
$L_{\varepsilon,\mu }\widehat {\mathbf{U}}_{BL}^{M}$. For that 
purpose, we need the following lemma: 

\begin{lem}
\label{lemma:estimate-remainder-case-II} Assume $0\leq \widehat{x}%
\varepsilon \gamma _{A}\leq X<1$ with $\gamma _{A}$, $X$, $C_1$, $\beta > 0$  
known constants. Let $0 <\varepsilon \leq \mu \leq 1$. 
Then
\begin{align*}
&
\sum_{i\geq M+1}
\left(\frac{\varepsilon}{\mu}\right)^{i}
\sum_{k =i-M}^{i}\mu^{k} \widehat{x}^{k }\gamma
_{A}^{k }\left( \mu +\frac{1}{i-k+1 }\right) ^{i-k }\frac{1}{(i-k)!}
(C_{1}(i-k) +\widehat{x})^{2(i-k)}e^{-\beta \widehat{x}}\\
& \qquad \qquad 
\leq \frac{%
C\mu}{1-X}(\gamma \varepsilon (M+1+1/\mu ))^{M+1}e^{-\beta \widehat{x}/2},
\end{align*}
where $C$, $\gamma >0 $ are positive constants independent 
of $\varepsilon$ and $\mu$. 
\end{lem}
\begin{proof}
The key ingredients of the proof are the estimates given 
in Lemma~\ref{lemma:elementary-properties-of-factorial}. 
\iftechreport
For details, see Appendix~\ref{appendix:case_2_estimate-remainder-case-II}.
\else
For details, see 
\cite[Appendix~\ref{appendix:case_2_estimate-remainder-case-II}]{mxo-1}.
\fi
\end{proof}

\begin{thm}
\label{thm:estimate-remainder-case-II} 
Let $\beta > 0$ be given by Theorem~\ref{thm2}, i.e., 
$\beta = \sqrt{a_{11}(0)}$
There exist constants $C$, $\gamma > 0$ such that, 
under the assumptions
$$
\left(\varepsilon (M+1) + \frac{\varepsilon}{\mu} \right)\gamma 
\leq 1 \qquad \widehat x \varepsilon \gamma \leq 1,
$$
$\widehat{\mathbf U}^M$ given by \emph{(\ref{BL_M})}
satisfies
\begin{equation*}
\left|  L_{\varepsilon,\mu }\widehat {\mathbf{U}}_{BL}^{M}(\widehat x)\right| 
\leq C (\varepsilon
(M+1+1/\mu )\gamma )^{M-1}e^{-\beta \widehat x}.
\end{equation*}
For the case $\widehat x \varepsilon \gamma > 1$ (but still assuming
$\left(\varepsilon (M+1) + \frac{\varepsilon}{\mu} \right)\gamma \leq 1$), we 
have 
\begin{equation*}
\left|  L_{\varepsilon,\mu }\widehat {\mathbf{U}}_{BL}^{M}(\widehat x)\right| 
\leq C e^{-\beta/2 \widehat x}. 
\end{equation*}
\end{thm}

\begin{proof}
For the case of $\widehat x\varepsilon$ being sufficiently small, 
the starting point is (\ref{eq:remainder-case-II}), which is
a sum of two contributions, namely, the double sum and the single
sum (consisting of merely two terms). For the double sum, we 
recall that $\widehat x = x/\varepsilon$, so that 
Lemma~\ref{lemma:estimate-remainder-case-II} produces the desired
estimate. For the single sum, we use the estimates of 
Theorem~\ref{thm2} for $\widehat v_{i+2}^L$. From Cauchy's Integral 
Formula for derivatives with contour being taken as $\partial B_1(x)$, 
we get 
$$
|(\widehat v_{i+2}^L)^{\prime\prime}(x)| \leq C \frac{1}{(i+2)!}
(C (i+ 3) + |\widehat x|)^{2(i+2)} e^{-\beta \widehat x}. 
$$
Using 
(\ref{eq:lemma:estimate-remainder-case-II-1}) and 
(\ref{eq:lemma:estimate-remainder-case-II-2}), 
we see that this contribution can be bounded by 
$C (\gamma (M+1)\varepsilon/\mu)^{M-1}$ for suitable $C$, $\gamma > 0$. 

For the case $\widehat x \varepsilon \gamma > 1$, we use the exponential
decay of $\widehat {\mathbf U}^M_{BL}$ expressed in 
Theorem~\ref{thm:bdy-layer-fct-entire-case-II}. 
The factors $\varepsilon^{-2}$, which arise when computing 
$L_{\varepsilon,\mu} \widehat{\mathbf U}^M_{BL}$ can be absorbed by the 
expontially decaying term since $\varepsilon^{-1} \leq \gamma \widehat x$
(see Lemma~\ref{lemma:elementary-properties-of-factorial}). 
\end{proof}

\subsection{Boundary mismatch of the expansion}
\begin{thm}
\label{thm:bdy-mismatch-case-II}
There exist constants $C$, $b$, $\gamma > 0$ such that under the assumptions
\begin{equation*}
\left( \varepsilon (M+1) + \frac{\varepsilon}{\mu} \right) \gamma \leq 1, 
\end{equation*}
one has 
$$
\|{\mathbf W}_M + \widehat{\mathbf U}^M_{BL}
+ \widehat{\mathbf V}^M_{BL}\|_{L^\infty(\partial I)} 
\leq  C e^{-b/\varepsilon}. 
$$
\end{thm}
\begin{proof}
We consider the left endpoint of $I$ (the right endpoint is similar).
By construction, 
\begin{eqnarray}
\left\Vert \mathbf{R}_{M}(0)\right\Vert &=&\left\Vert \mathbf{U}(0)-\left( 
\mathbf{W}_{M}(0)+\widehat{\mathbf{U}}_{BL}^{M}(0)+\widehat{\mathbf{V}}_{BL}^{M}(1/\varepsilon)\right)
\right\Vert 
 = \left\Vert \widehat {\mathbf{V}}^M(1/\varepsilon) \right\Vert.
\end{eqnarray}
The result now follows from Theorem~\ref{thm2}.
\end{proof}
\subsection{Proof of Theorem~\ref{thm:case-II}} 
The result of Theorem~\ref{thm:case-II} now follows from combining Theorems~\ref{thm_case2_smooth_part}, 
\ref{thm:estimate-remainder-case-II}, 
\ref{thm:bdy-mismatch-case-II}. 


\section{The second two scale case\label{case_3}: Case~\ref{case:III}}

Recall that this occurs when $\mu /1$ is small but $\varepsilon /\mu $ is 
\emph{not} small. The main result is: 
\begin{thm}
\label{thm:case-III}
The solution ${\mathbf U}$ of (\ref{eq:model-problem}) can be
written as ${\mathbf U} = {\mathbf W} + \widetilde{\mathbf U}_{BL} 
+ {\mathbf R}$, where ${\mathbf W} \in {\mathcal A}(1, C_W,\gamma_W)$
and 
$\widetilde {\mathbf U}_{BL} \in \BLtwo(\delta \mu, C_{BL},\gamma_{BL})$,
for suitable constants $C_W$, $C_{BL}$, $\gamma_W$, $\gamma_{BL}$, $\delta > 0$
independent of $\mu$ and $\varepsilon$. Furthermore, ${\mathbf R}$ satisfies
\begin{eqnarray*}
\|{\mathbf R}\|_{L^\infty(\partial I)} + 
\|L_{\varepsilon,\mu} {\mathbf R}\|_{L^2(I)} 
& \leq & C \left(\mu/\varepsilon\right)^{-2} e^{-b/\mu}, 
\end{eqnarray*}
for some constants $C$, $b > 0$ independent of $\mu$ and $\varepsilon$. In particular,
$\|{\mathbf R}\|_{E,I} \leq (\mu/\varepsilon)^2 e^{-b/\mu}$. 
\end{thm}

In this case we recall the stretched variables 
$\widetilde{x}=x/\mu $ and $\widetilde{x}^R = (1-x)/\mu$ 
and make the formal ansatz
\begin{equation}
\label{9b}
{\mathbf U} \sim  \sum_{i=0}^\infty \mu^i 
\left[ {\mathbf U}_i(x) + \widetilde {\mathbf U}^L_i(\widetilde x) 
+ \widetilde {\mathbf U}^R_i(\widetilde x^R)\right],
\end{equation}
%
%
where again
$$
{\mathbf U}_i(x) = \left(\begin{array}{c} u_i \\ v_i \end{array}\right), 
\qquad 
{\mathbf U}^L_i(\widetilde x) = 
\left(\begin{array}{c} \widetilde u_i(\widetilde x) \\ 
                       \widetilde v_i(\widetilde x) 
      \end{array}\right), 
$$
and an analogous definition for $\widetilde {\mathbf U}^R$. 
We also recall that the differential operator $L_{\varepsilon,\mu}$ takes 
the form (\ref{eq:L-on-tilde-scale}) on the $\widetilde x$-scale. 
The separation of the slow ($x$) and the fast variables ($\widetilde x$
and $\widetilde x^R$) leads to the following equations: 
\begin{subequations}
\label{eq:case_3-separation-of-scales}
\begin{eqnarray}
\sum_{i=0}^\infty \mu^i 
\left( - {\mathbf E}^{\varepsilon,\mu} {\mathbf U}_i^{\prime\prime} + {\mathbf A}(x) {\mathbf U}_i
\right) &=& {\mathbf F} = \left(\begin{array}{c} f \\ g \end{array}\right), 
\label{eq:case_3-separation-of-scales-slow-variable}
\\
\sum_{i=0}^\infty \mu^i 
\left( - \mu^{-2} {\mathbf E}^{\varepsilon,\mu} 
(\widetilde {\mathbf U}^L_i)^{\prime\prime} + 
\sum_{k=0}^\infty \mu^k \widetilde x^k{\mathbf A}_k \widetilde {\mathbf U}_i^L
\right) &=& 0,
\label{eq:case_3-separation-of-scales-fast-variable}
\end{eqnarray}
\end{subequations}
and an analogous system for $\widetilde {\mathbf U}^R$. 
Next, for $\varepsilon$ appearing in
(\ref{eq:case_3-separation-of-scales}) we write 
$\varepsilon = \mu \frac{\varepsilon}{\mu}$ and equate like powers of 
$\mu$ to get with the matrix 
$$
{\mathbf E}^{\varepsilon/\mu,1} = 
\left(\begin{array}{cc} 
\left(\frac{\varepsilon}{\mu}\right)^2 & 0 \\ 0 & 1 
\end{array}
\right),
$$
the following two recurrence relations: 
\begin{subequations}
\label{eq:recurrence-case-III}
\begin{eqnarray}
\label{eq:recurrence-case-III-outer}
   - {\mathbf E}^{\varepsilon/\mu,1} {\mathbf U}_{i-2}^{\prime\prime} 
+ {\mathbf A}(x) {\mathbf U}_i&=&
{\mathbf F}_i, \\
\label{eq:recurrence-case-III-inner}
   - {\mathbf E}^{\varepsilon/\mu,1} (\widetilde {\mathbf U}_{i}^L)^{\prime\prime} 
+ \sum_{k=0}^i \widetilde x^{i-k} {\mathbf A}_{i-k} \widetilde {\mathbf U}^L_k &=& 0, 
\end{eqnarray}
\end{subequations}
where, as usual functions with negative index are assumed to be 
zero and ${\mathbf F}_i$ are defined by ${\mathbf F}_0 = (f,g)^T$ 
and ${\mathbf F}_i = 0$ for $i > 0$. The terms of the 
outer expansion, ${\mathbf U}_i$, are obtained immediately from 
(\ref{eq:recurrence-case-III-outer}): 
\begin{eqnarray}
\label{16a}
{\mathbf A} {\mathbf U}_i &=& \left(\begin{array}{c} f_i \\ g_i\end{array}\right) + 
\left( \begin{array}{c} (\varepsilon/\mu)^2 u_{i-2}^{\prime\prime} \\ v_{i-2}^{\prime\prime}
       \end{array}
\right). 
\end{eqnarray}
The functions $\widetilde {\mathbf U}_i^L$ of the inner expansion 
are defined as the solutions of the following boundary value problems: 
\begin{subequations}
\label{eq:16b-17b}
\begin{eqnarray}
\label{16b}
- 
{\mathbf E}^{\varepsilon/\mu,1}(\widetilde {\mathbf U}^L_i)^{\prime\prime} + 
{\mathbf A}(0) {\mathbf U}^L_i = 
- \sum_{n=0}^{i-1} \widetilde x^{i-n} \widetilde A^L_{i-n} \widetilde{\mathbf U}^L_n, \\
\label{17b}
\widetilde {\mathbf U}^L_i(0)  = - {\mathbf U}_i(0), 
\qquad 
\widetilde {\mathbf U}^L_i(\widetilde x)  \rightarrow 0 \quad \mbox{ as $\widetilde x \rightarrow \infty$},
\end{eqnarray}
\end{subequations}
and an analogous system for ${\mathbf U}^R_i$. 
\subsection{Properties of some solution operators}
The functions ${\mathbf U}_i^L$ of the inner expansion are solutions
to elliptic {\em systems}. In contrast to the previous arguments, 
which were based on estimates for {\em scalar} problems 
(for which strong tools such as maximum principles are readily available),
we employ more general energy type arguments here to deal with the case 
of systems.  We start by introducing exponentially weighted spaces 
on the half-line $(0,\infty)$, by defining the norm 
\begin{equation}
\|u\|^2_{0,\beta}:= \int_{x=0}^\infty e^{2\beta x} |u(x)|^2\,dx,
\end{equation}
with the obvious interpretation in case $u$ is vector valued. 
The following lemma shows that elliptic systems of the relevant 
type (\ref{eq:16b-17b}) can be solved in a setting of exponentially
weighted spaces: 
\begin{lem}
\label{lemma:system-inf-sup}
Let $\nu \in (0,1]$ and set ${\mathbf E}:= \left(\begin{array}{cc} \nu^2 & 0 \\ 0 & 1\end{array}\right)$.
Let ${\mathbf B} \in \R^2$ be positive definite,
i.e., $x^\top {\mathbf B} x \ge \beta_0^2 |x|^2$ for all $x \in \R^2$.
Then the bilinear form
$$
a({\mathbf U},{\mathbf V}) = \int_{x=0}^\infty {\mathbf U}^\prime \cdot {\mathbf E} {\mathbf V}^\prime + 
                                             {\mathbf U} \cdot {\mathbf B} {\mathbf V}\,dx,
$$
satisfies for a constant $C > 0$ that depends solely on $\beta_0$, 
the following inf-sup condition for all $0 < \beta < \beta_0$:
For any ${\mathbf U} \in H^1_\beta (0, \infty) $, where
\begin{equation}
H^1_{\beta}(0,\infty) = \{ u: || u ||_{1,\beta} < \infty \},
\label{H1_b}
\end{equation}
there exists ${\mathbf V} \in H^1_{-\beta}$
with ${\mathbf V} \ne 0$ such that 
$$
a({\mathbf U},{\mathbf V}) \ge C \frac{1}{\beta_0 - \beta} \|{\mathbf U}\|_{1,\beta} \|{\mathbf V}\|_{1,-\beta}. 
$$

Here, we define for $\alpha \in \R$
\begin{equation}
\|{\mathbf U}\|_{1,\alpha}^2:= \int_{x=0}^\infty e^{2\alpha x} 
\left[ {\mathbf U}^\prime\cdot {\mathbf E} {\mathbf U}^\prime + {\mathbf U}\cdot {\mathbf B} {\mathbf U}\right]\,dx.
\end{equation}
\end{lem}
\begin{proof}
Given ${\mathbf U} \in H^1_\beta (0, \infty)$ we set ${\mathbf V}:= e^{2\beta x} {\mathbf U}$. Then
${\mathbf V}^\prime(x) = 2\beta e^{2\beta x}{\mathbf U}(x) + e^{2\beta x} {\mathbf U}^\prime(x)$ and thus
\begin{eqnarray*}
\|{\mathbf V}\|_{1,-\beta}^2  &\leq & 4 (1 + \frac{|\beta|}{\beta_0})^2 \|{\mathbf U}\|^2_{1,\beta}\\
a({\mathbf U},{\mathbf V}) &=& \|{\mathbf U}\|^2_{1,\beta} + 
2 \beta \int_{x=0}^\infty e^{2\beta x}{\mathbf U}^\prime \cdot {\mathbf E}{\mathbf U}\,dx
\ge \left(1 - \frac{|\beta|}{\beta_0}\right)\|{\mathbf U}\|^2_{1,\beta}.
\end{eqnarray*}
The result then follows.
\end{proof}
\begin{lem}
\label{lemma:system-a-priori} 
Let ${\mathbf f}$ satisfy 
$\|{\mathbf f}\|_{0,\beta} < \infty$ for some 
$\beta \in [0,\beta_0)$, and let ${\mathbf g} \in \R^2$. 
Then the solution ${\mathbf U}$ of 
$$
-{\mathbf E} {\mathbf U}^{\prime\prime} + {\mathbf B} {\mathbf U} = {\mathbf f},
\qquad {\mathbf U}(0) = {\mathbf g},
$$
satisfies for a $C > 0$ independent of $\beta$, the estimate 
$$
\|{\mathbf U}\|_{1,\beta} \leq 
C (\beta_0 - \beta)^{-1}\left[ \|{\mathbf f}\|_{0,\beta}  
+ |{\mathbf g}|\right].
$$
\end{lem}
\begin{proof}
Let $\beta_1 > \beta_0$ and set ${\mathbf U}_0 = {\mathbf g} e^{-\beta_1 x}$. 
Then 
${\mathbf U}_0$ satisfies the desired estimates. The remainder ${\mathbf U} - {\mathbf U}_0$
satisfies an inhomogeneous  differential equation with homogeneous boundary conditions
at $x = 0$. Hence, Lemma~\ref{lemma:system-inf-sup} is applicable and yields the desired result.
\end{proof}

\begin{lem}
\label{lemma:foo}
There exist $\delta_0 > 0$ and $C_0 > 0$ such that for every $\delta \in (0,\delta_0]$ 
and every $m \in \N_0$, there holds 
$$
\sum_{n=0}^{i-1} \delta^{i-1-n} \frac{(n+m)^n}{(i+m)^i} \frac{(i+n+1+m)^{i+n+1+m}}{(2n+1+m)^{2n+1+m}}
\leq C_0.
$$
\end{lem}
\begin{proof}
\iftechreport
See Appendix~\ref{appendix:proof_of_lemma:foo}
\else
See \cite[Appendix~\ref{appendix:proof_of_lemma:foo}]{mxo-1}
\fi
\end{proof}

\subsection{Regularity of the functions 
${\mathbf U}_i$, $\widetilde {\mathbf U}^L_i$, $\widetilde {\mathbf U}_i^R$}
\begin{lem}
\label{lemma:case-III-outer-expansion}
The function ${\mathbf U}_i$ defined by (\ref{16a}), are holomorphic
in a neighborhood $G$ of $\overline{I}$ and satisfy for some $C$, $\widetilde K >0$, 
$$
|{\mathbf U}_i(z)| \leq \widehat C \delta^{-i} \widehat K^i i^i 
\qquad \forall z \in G_\delta:= \{z \in G\,|\, \operatorname*{dist}(z,\partial G) > \delta\}, \qquad \forall i \ge 0. 
$$
Additionally, ${\mathbf U}_{2i+1} = 0$ for all $i \in \N_0$. 
\end{lem}
\begin{proof}
We note that $\varepsilon/\mu \leq 1$. 
The arguments are then analogous to those of \cite[Lemma~2]{m}. The arguments are also
structurally similar to the more complicated case studied in Lemma~\ref{lem10}.
\end{proof}
We now turn to estimates for the inner expansion functions
$\widetilde{\mathbf U}_i^L$. 
\begin{thm}
\label{thm2a}
The functions $\widetilde {\mathbf U}^L_i$ defined by (\ref{16b}), (\ref{17b}) are entire
functions and satisfy for all $\beta \in (0,\beta_0)$ 
(with $\beta_0 = \alpha$ and $\alpha$ given by (\ref{eq:A-positive-definite})), 
\begin{equation}
\label{eq:thm2a-1}
\|\widetilde {\mathbf U}_i^L\|_{1,\beta} \leq \widetilde C K^i (\beta_0 - \beta)^{-(2i+1)} i^i \qquad \forall i \in \N_0.
\end{equation}
\end{thm}
\begin{proof}
We note that Lemma~\ref{lemma:system-a-priori} is to be applied 
with ${\mathbf B} = {\mathbf A}(0)$ and ${\mathbf E} = {\mathbf E}^{\varepsilon/\mu,1}$. 
The case $i = 0$ follows from Lemma~\ref{lemma:system-a-priori} 
and Lemma~\ref{lemma:case-III-outer-expansion}. 
For $i \ge 1$, we proceed by
induction. In order to be able to apply Lemma~\ref{lemma:system-a-priori}, 
we define
$$
\widetilde {\mathbf F}(\widetilde x) = 
\sum_{n=0}^{i-1} \widetilde {\mathbf A}_{i-n} 
\widetilde x^{i-n} \widetilde {\mathbf U}^L_n(\widetilde x). 
$$
Next, we estimate, for an arbitrary $\beta \in (0,\beta_0)$ 
and $\tilde \beta \in (\beta,\beta_0)$, with the aid 
of Lemma~\ref{lemma:elementary-properties-of-factorial},
\begin{eqnarray*}
\int_0^\infty e^{2\beta \widetilde x}\widetilde x^{2(i-n)} 
|\widetilde {\mathbf U}^L_n(\widetilde x)|^2\,d\widetilde x 
& \leq &
\int_0^\infty e^{- 2(\tilde \beta -\beta)\widetilde x}\widetilde x^{2(i-n)} 
e^{2\tilde \beta \widetilde x} |\widetilde {\mathbf U}^L_n(\widetilde x)|^2\,d\widetilde x \\
&\leq& \sup_{x > 0} e^{-2(\tilde \beta -\beta) x} x^{2(n-i)} \|\widetilde {\mathbf U}^L_n\|^2_{0,\tilde \beta}\\
&\leq& e^{-2(i-n)} \left(\frac{i-n}{\tilde \beta-\beta}\right)^{2(i-n)}
\|\widetilde{\mathbf U}^L_n\|^2_{0,\tilde\beta}\\
& \leq & \widetilde C^2 K^{2n} n^{2n} (\beta_0 - \tilde \beta )^{-2 (2 n+1)} 
e^{-2(i-n)} \left(\frac{i-n}{\tilde \beta-\beta}\right)^{2(i-n)}. 
\end{eqnarray*}
Selecting $\tilde \beta = (\beta_0 - \beta)\kappa + \beta $ for some
$\kappa \in (0,1)$ to be chosen shortly, we get
\[
\int_0^\infty e^{2\beta \widetilde x}\widetilde x^{2(i-n)} 
|\widetilde {\mathbf U}^L_n(\widetilde x)|^2\,d\widetilde x \leq 
\]
\[
\leq \widetilde C^2 K^{2n} n^{2n} (\beta_0 - \beta )^{-2 (i+n+1)} 
e^{-2(i-n)} (i-n)^{2(i-n)}\frac{1}{\kappa^{2(i-n)} (1-\kappa)^{2(2n+1)}}.
\]
The choice $\kappa = \frac{i-n}{i+n+1}$ yields
\begin{eqnarray}
\label{eq:thm2a-10}
\int_0^\infty e^{2\beta \widetilde x}\widetilde x^{2(i-n)} 
|\widetilde {\mathbf U}^L_n(x)|^2\,dx 
& \leq &
\widetilde C^2 K^{2n} (\beta_0 - \beta )^{-2 (i+n+1)} 
e^{-2(i-n)} \frac{(n+i+1)^{2(i+n+1)}}{(2n+1)^{2(2n+1)}} n^{2n}.
\end{eqnarray}
Hence,
\begin{eqnarray}
\nonumber 
\|\widetilde {\mathbf F}\|_{0,\beta} &\leq& 
\widetilde C C_A (\beta_0 - \beta)^{-2i} K^{i-1} i^i
\sum_{n=0}^{i-1} \gamma_A^{i-n} K^{n-i+1} e^{-(i-n)} 
\frac{n^n}{i^i} \frac{(i+n+1)^{i+n+1}}{(2n+1)^{2n+1}} (\beta_0 -\beta)^{i-n-1}\\
& \leq &
\label{eq:thm2a-20}
\widetilde C C_A C_0 (\beta_0 - \beta)^{-2i} K^{i-1} i^i , 
\end{eqnarray}
where we appealed to Lemma~\ref{lemma:foo} and used implicitly 
that $K$ is sufficiently large. Using Lemma~\ref{lemma:case-III-outer-expansion}
for a fixed $\delta$, we get from  Lemma~\ref{lemma:system-a-priori} 
\begin{eqnarray*}
\|\widetilde {\mathbf U}_i\|_{1,\beta} &\leq& (\beta_0 - \beta)^{-1} 
\left[\|\widetilde {\mathbf F}\|_{0,\beta} 
+ \widehat C (\widehat K/\delta)^i i^i\right] \\
&\leq& \widetilde C (\beta_0 - \beta)^{-2i-1} K^i i^i 
\left[ K^{-1} C_A C_0 + \frac{\widehat C}{\widetilde C}(\beta_0- \beta)
\left(\frac{(\beta_0 - \beta)^2 \widetilde K}{K}\right)^i\right]. 
\end{eqnarray*}
The expression in square brackets can be bounded by $1$ uniformly in $i$ and 
$\beta \in (0,\beta_0)$ if we assume that $\widetilde C$ and $K$ are sufficiently large. 
\end{proof}
We next refine the argument to include bounds on all derivatives 
of ${\widetilde {\mathbf U}}_i$: 
\begin{thm}
\label{thm:higher-derivatives-in-exponentially-weighted-spaces}
There exist $C_U$, $K_1$, $K_2 > 0$, independent of $\beta$, 
$\nu = \varepsilon/\mu$, $m$ and $i$, such that 
$$
\|(\widetilde {\mathbf U}_i^L)^{(m)}\|_{0,\beta}
\leq C_U (\beta_0 - \beta)^{-(2i+1+m)} (i+m)^{i} K_1^{i} K_2^{m} \nu^{-m}.
$$
\end{thm}
\begin{proof}
The cases $m = 0$ and $m=1$ are covered by the above lemma
(note: $\|{\mathbf E}^{-1}\| \leq \nu^{-2}$). The remaining 
cases are obtained as usual by differentiating the equation
satisfied by $\widetilde {\mathbf U}^L_i$ and then proceed
by induction on $m$. 
\iftechreport
For details, see 
Appendix~\ref{appendix:case_3_thm:higher-derivatives-in-exponentially-weighted-spaces}
\else
\cite[Appendix~\ref{appendix:case_3_thm:higher-derivatives-in-exponentially-weighted-spaces}]{mxo-1}.
\fi
\end{proof}
We conclude this section by showing that the boundary layer 
functions 
$$
\widetilde{\mathbf U}^M_{BL}:= \sum_{i=0}^M \mu^i \widetilde {\mathbf U}^L_i,
$$ 
are entire functions.
\begin{thm}
\label{thm:boundary-layer-fct-entire-case-III}
Fix $\beta \in (0,\beta_0)$. 
There exist constants $C$, $\gamma$, $K > 0$ such that under the assumption
$\mu(M+1) \gamma \leq 1$ there holds 
$$
\left\|\frac{d^m}{d\widetilde x^m} \widetilde{\mathbf U}_{BL}^M
\right\|_{0,\beta} \leq C K^m \nu^{-m} 
\qquad \forall m \in \N_0. 
$$
\end{thm}
\begin{proof}
From Theorem~\ref{thm:higher-derivatives-in-exponentially-weighted-spaces}, 
we see that
for all $m \in \N_0$,
\begin{eqnarray*}
\|(\widetilde{\mathbf U}_{BL}^M)^{(m)}\|_{0,\beta} & \leq &
C (\beta_0 - \beta)^{-1-2m} 
\sum_{i=0}^M \mu ^i (i + m)^i (\beta_0 - \beta)^{-2i} K_1^i K_2^m \nu^{-m}\\
& \leq &
C (\beta_0 - \beta)^{-1-2m} K_2^m \nu^{-m} \sum_{i=0}^M (2 (\beta_0 - \beta)^{-2}K_1 \mu i)^i + (2 K_1 \mu m)^i \\
& \leq &
C (\beta_0 - \beta)^{-1-2m} K_2^m \nu^{-m} \sum_{i=0}^M (2 K_1 \mu M)^i \left( (\beta_0 - \beta)^{-2i} + \left(\frac{m}{M}\right)^i\right) 
\leq C \widetilde K_2^m \nu^{-m} 
\end{eqnarray*}
for an appropriate $\widetilde K_2 $ (depending on $\beta_0 - \beta$!),
if we assume that $\mu M $ is sufficiently
small. The key observation for this fact is to note that for $m > M$ we have 
\begin{eqnarray*}
\sum_{i=0}^M (m/M)^i &\leq & (M+1) (m/M)^M =  m \frac{M+1}{M} (m/M)^{M-1} 
\leq m \frac{M+1}{M} \left(\frac{m}{M-1}\right)^{(M-1)/m) m},
\end{eqnarray*}
and $n^{1/n} \rightarrow 1$ for $n \rightarrow \infty$. 
\end{proof}
\subsection{Remainder estimates}
\subsubsection{Remainder estimates for the outer expansion: ${\mathbf F} - L_{\varepsilon,\mu} {\mathbf W}_M$}
As before, the formal expansion (\ref{9b}) is truncated after $M$
terms to yield the decomposition 
$$
{\mathbf U} = \sum_{i=0}^M \mu^i {\mathbf U}_i(x)
+ \sum_{i=0}^M \mu^i \widetilde{\mathbf U}^L_i(\widetilde x) + 
 \sum_{i=0}^M \mu^i \widetilde{\mathbf U}^R_i(\widetilde x^R) 
+ {\mathbf R}^M .
$$
A calculation shows 
$$
{\mathbf F} - L_{\varepsilon,\mu} \sum_{i=0}^M \mu^i {\mathbf U}_i = 
\mu^{M+2} {\mathbf E}^{\varepsilon/\mu,1} {\mathbf U}_{M}^{\prime\prime}.
$$
We therefore get
\begin{thm}
\label{thm:case-III-remainder-outer-expansion}
There exists $\gamma > 0$ independent of $\varepsilon$, $\mu$ such that 
for $\mu (M+1) \gamma \leq 1$, there holds 
$$
\|{\mathbf F} - L_{\varepsilon,\mu} \sum_{i=0}^M \mu^i {\mathbf U}_i \|_{L^\infty(I)}
\leq C (\mu(M+1)\gamma)^{M+2}.
$$
\end{thm}
\begin{proof}
The proof is analogous to that of \cite[Thm.~6]{m}. 
\end{proof}
\subsubsection{Remainder estimates for the inner expansion: 
$L_{\varepsilon,\mu} \widetilde {\mathbf U}^M_{BL}$}
We consider only the contribution from the left endpoint and 
define $\widetilde {\mathbf U}_{BL}^M:= 
\sum_{i=0}^M \mu^i \widetilde {\mathbf U}_i^{L}$. A calculation shows
\begin{equation}
\label{eq:remainder-bl-case-IV}
L_{\varepsilon,\mu} \widetilde{\mathbf U}_{BL}^M = 
\sum_{i \ge M+1} \mu^i \sum_{k=0}^{M} 
\widetilde x^{i-k} {\mathbf A}_{i-k} \widetilde {\mathbf U}^L_k. 
\end{equation}
The following lemma provides an estimate (in an exponentially weighted
space) for $L_{\varepsilon,\mu} \widetilde{\mathbf U}^M_{BL}$ 
near the left endpoint: 
\begin{lem}
\label{lemma:remainder-exponentially-weighted-spaces}
There exist $C$, $\delta$, $\beta > 0$, $K > 0$ such that 
$$
\int_{\widetilde x=0}^{\delta/\mu} e^{2\beta \widetilde x} 
\left|\sum_{i \ge M+1} \mu^i \sum_{k=0}^{M} 
\widetilde x^{i-k} {\mathbf A}_{i-k} \widetilde {\mathbf U}^L_k(\widetilde x)
\right|^2\,d\widetilde x
\leq C (K \mu (M+1))^{2(M+1)}.
$$
\end{lem}
\begin{proof} 
For fixed $\widetilde x > 0$ we estimate 
\begin{eqnarray}
\label{eq:lemma:remainder-exponentially-weighted-spaces-10}
\left|\sum_{i \ge M+1} \mu^i \sum_{k=0}^{M} 
\widetilde x^{i-k} {\mathbf A}_{i-k} \widetilde {\mathbf U}^L_k(\widetilde x)
\right|
\leq C_A 
\sum_{k=0}^M |\widetilde{\mathbf U}^L_k(\widetilde x)| 
\sum_{i=M+1}^\infty \gamma_A^{i-k} \widetilde x^{i-k} \mu^i.
\end{eqnarray}
For $\mu \widetilde x \gamma_A \leq 1/2$, we estimate further 
\begin{eqnarray*}
\sum_{i=M+1}^\infty \gamma_A^{i-k} \widetilde x^{i-k} \mu^i
\leq 2 (\mu \widetilde x \gamma_A)^{M+1} (\widetilde x \gamma_A)^{-k} 
\leq 2 \mu^{M+1}  (\widetilde x \gamma_A)^{M+1-k}.
\end{eqnarray*}
Inserting this in (\ref{eq:lemma:remainder-exponentially-weighted-spaces-10}), we see
that employing the estimate (\ref{eq:thm2a-10}) we can reason in exactly the same
way as we have to reach (\ref{eq:thm2a-20}), to get 
$$
\int_{\widetilde x=0}^{\delta/\mu} e^{2\beta \widetilde x} 
\left|\sum_{i \ge M+1} \mu^i \sum_{k=0}^{M} 
\widetilde x^{i-k} {\mathbf A}_{i-k} \widetilde {\mathbf U}^L_k(\widetilde x)
\right|^2\,d\widetilde x
\leq C (K \mu (M+1))^{2(M+1)}, 
$$
which is the desired estimate.
\end{proof} 
Lemma~\ref{lemma:remainder-exponentially-weighted-spaces}
provides an estimate for $L_{\varepsilon,\mu}\widetilde{\mathbf U}^M_{BL}$
near the left endpoint; the following result provides an estimate 
on the whole interval $I$:
\begin{thm}
\label{thm:case-III-remainder-inner-expansion}
There exist $C$, $\gamma$, $K$, $b > 0$ such that for $\mu (M+1) \gamma \leq 1$, there holds 
$$
\|L_{\varepsilon,\mu} \widetilde{\mathbf U}_{BL}^M\|_{L^2(I)} 
\leq C \sqrt{\mu} \left[ (\mu(M+1)K)^{M+1} + \left(\frac{\mu}{\varepsilon}\right)^{2} e^{-b/\mu}\right]. 
$$
\end{thm}
\begin{proof}
We merely consider the left boundary layer. 
Lemma~\ref{lemma:remainder-exponentially-weighted-spaces} allows us 
to estimate for $L_{\varepsilon,\mu} \widetilde{\mathbf U}^M_{BL}$ 
on the interval $(0,\delta)$ with the change of variables
 $x = \mu \widetilde x$: 
$$
\int_{x=0}^{\delta} |L_{\varepsilon,\mu} \widetilde {\mathbf U}^{M}_{BL}|^2\,dx = 
\mu \int_{\widetilde x=0}^{\delta/\mu} 
|L_{\varepsilon,\mu} \widetilde {\mathbf U}^{M}_{BL}(\widetilde x)|^2
\,d\widetilde x  
\leq C \mu ((M+1)\mu K)^{2(M+1)}. 
$$
For the interval $(\delta,1)$, we note that 
$$
L_{\varepsilon,\mu} \widetilde{\mathbf U}^M_{BL} = 
- {\mathbf E}^{\nu,1} 
\left( \widetilde{\mathbf U}^M_{BL}\right)^{\prime\prime} (\widetilde x)
+ {\mathbf A}(x) \widetilde{\mathbf U}^M_{BL}(\widetilde x). 
$$
Hence, we can estimate 
$$
\int_{x=\delta}^1 
| L_{\varepsilon,\mu} \widetilde {\mathbf U}^M_{BL}|^2\,dx 
\leq C \mu \int_{\widetilde x=\delta/\mu}^\infty 
|\left(\widetilde{\mathbf U}_{BL}^M\right)^{\prime\prime}(\widetilde x)|^2 + 
|\widetilde{\mathbf U}_{BL}^M(\widetilde x)|^2 \,d\widetilde x
\leq C e^{-2\delta \beta /\mu} 
\left[ 
\| \left(\widetilde{\mathbf U}_{BL}^M\right)^{\prime\prime}\|^2_{0,\beta} 
+ 
\| \widetilde{\mathbf U}_{BL}^M\|^2_{0,\beta} 
\right]. 
$$
With 
Theorem~\ref{thm:boundary-layer-fct-entire-case-III}, we 
therefore arrive at 
\begin{eqnarray*}
\|L_{\varepsilon,\mu}\widetilde{\mathbf U}_{BL}^M\|^2_{L^2(I)} &\leq& 
C \mu \left[ (\mu (M+1)K)^{2(M+1)} + \nu^{-4} e^{-2\delta \beta/\mu}\right].
\end{eqnarray*}
\end{proof}
\begin{remark}
The factor $(\mu/\varepsilon)^{-2}$ in the second term on the right-hand 
side of 
Theorem~\ref{thm:case-III-remainder-inner-expansion} is likely suboptimal. 
\end{remark}
\subsection{Boundary mismatch of the expansion}
\begin{thm}
\label{thm:bdy-mismatch-case-III}
There exist constants $C$, $b$, $\gamma > 0$ such that under the assumptions
\begin{equation*}
 \mu (M+1) \gamma \leq 1,
\end{equation*}
one has 
$$
\|{\mathbf W}_M + \widetilde{\mathbf U}^M_{BL}
+ \widetilde {\mathbf V}^M_{BL}\|_{L^\infty(\partial I)} 
\leq  C \left( \frac{\mu}{\varepsilon}\right)^{1/2}  e^{-b/\mu} . 
$$
\end{thm}
\begin{proof}
Moreover, at the endpoints of the interval $I$, the remainder is small. \ To
see this, consider the left endpoint of $I$ (the right endpoint is similar).
By construction, 
\begin{eqnarray}
\left\Vert \mathbf{R}_{M}(0)\right\Vert &=&\left\Vert \mathbf{U}(0)-\left( 
\mathbf{W}_{M}(0)+\widetilde{\mathbf{U}}_{BL}^{M}(0)+
\widetilde{\mathbf{V}}_{BL}^{M}(1/\mu)\right)
\right\Vert 
 = \left\Vert \widetilde  {\mathbf{V}}^M(1/\mu) \right\Vert.
\end{eqnarray}
Theorem~\ref{thm:boundary-layer-fct-entire-case-III} informs us
that $\widetilde{\mathbf V}_M$ is an entire function. In fact, 
from 
Theorem~\ref{thm:boundary-layer-fct-entire-case-III} 
and the Sobolev embedding theorem in the form 
$\|v\|^2_{L^\infty(\tilde I)} \leq 
C \|v\|_{L^2(\tilde I)} \|v\|_{H^1(\tilde I)}$ 
applied to the interval $\tilde I = \tilde [1/\mu-1,1/\mu]$ 
of unit length, allows us to infer 
\begin{eqnarray*}
\|\widetilde{\mathbf V}^M_{BL}(1/\mu)\|^2_{L^\infty(\tilde I)}
& \leq & C 
\|\widetilde{\mathbf V}^M_{BL}\|_{L^2(\tilde I)} 
\left[ 
\|\widetilde{\mathbf V}^M_{BL}\|_{L^2(\tilde I)}  + 
\|\left(\widetilde{\mathbf V}^M_{BL}\right)^\prime\|_{L^2(\tilde I)}
\right] \\
& \leq & C e^{-2 \beta/\mu} 
\|\widetilde{\mathbf V}^M_{BL}\|_{0,\beta}
\left[ 
\|\widetilde{\mathbf V}^M_{BL}\|_{0,\beta}  + 
\|\left( \widetilde{\mathbf V}^M_{BL}\right)^\prime\|_{0,\beta} 
\right]
\leq C e^{-2 \beta/\mu} \nu^{-1},
\end{eqnarray*}
where the constant $C > 0$ depends only on the choice of $\beta$
made in 
Theorem~\ref{thm:boundary-layer-fct-entire-case-III}.  Recalling
the definition of $\nu$ concludes the argument. 
\end{proof}
\subsection{Proof of Theorem~\ref{thm:case-III}} 
The proof of Theorem~\ref{thm:case-III} now follows by combining 
Theorems~\ref{thm:case-III-remainder-outer-expansion}, 
\ref{thm:case-III-remainder-inner-expansion},
\ref{thm:bdy-mismatch-case-III} with $M = O(1/\mu)$ 
and using the stability result (\ref{eq:a-priori}). 

%

{\bf Acknowledgements:} The first author cordially thanks his colleagues 
W.~ Auzinger (Vienna) and P.~Szmolyan (Vienna) for fruitful discussions 
on the topics the paper. 
\appendix
\section{Some miscellaneous results}
\label{appendix-misc}

\begin{lem}
Let $\gamma >0$ and $\eta \in (0,1)$. Then the function
\begin{equation*}
x\mapsto f(x)=\gamma ^{x}(\eta x)^{x},
\end{equation*}
is convex on $(0,\infty )$ and monotone decreasing on $(0,1/(\eta \gamma e))$.
\end{lem}

\begin{proof} We only check the monotonicity assertion. To that end, we compute 
\begin{equation*}
\frac{d}{dx}\ln f(x)=\ln \gamma +1+\ln (\eta x)
\end{equation*}
and see that $\frac{d}{dx}\ln f(x)<0$ for $x\in (0,1/(\eta \gamma e))$. 
This proves the claim.
\end{proof}

The following lemma provides a proof for (\ref{eq:det-positive}) and (\ref{eq:diagonal-positive}).
\begin{lem}
\label{lemma:elementary-matrix-properties}
Let $\alpha > 0$ and let ${\mathbf B} \in \R^2$ be such that 
$\vec \xi \cdot {\mathbf B} \vec \xi \ge \alpha^2 \|\vec\xi\|^2$
for all $\vec \xi \in \R^2$. Then 
\begin{eqnarray}
\label{eq:lemma:elementary-matrix-properties-ii}
{\mathbf B}_{kk}& \ge &\alpha^2 ,\quad k=1,2,\\
\label{eq:lemma:elementary-matrix-properties-i}
\operatorname*{det} {\mathbf B}  &\ge & \alpha^2 \max\{{\mathbf B}_{11},{\mathbf B}_{22}\}. 
\end{eqnarray}
\end{lem}
\begin{proof}
Property (\ref{eq:lemma:elementary-matrix-properties-ii}) follows immediately from the 
choice $\vec\xi = \vec e_k$, where $\vec e_k$ is the $k$-th unit vector. 

To see property (\ref{eq:lemma:elementary-matrix-properties-i}), we start by noting that 
${\mathbf B}^{-1}$ is also positive definite: 
$$
x^T {\mathbf B}^{-1} x  \overset{x = B y}{=} y^T {\mathbf B}^T {\mathbf B}^{-1}{\mathbf B} y
= y^T {\mathbf B}^T y  
= (y^T {\mathbf B}^T y)^T  
= y^T {\mathbf B} y > 0. 
$$
Property (\ref{eq:lemma:elementary-matrix-properties-i}) follows from 
properties of the representation of ${\mathbf B}^{-1}$ in terms of 
the cofactor matrix as we now show. From 
$$
{\mathbf B}^{-1} = \frac{1}{\operatorname*{det} {\mathbf B}} 
{\mathbf C}, 
\qquad {\mathbf C}:= 
\left(\begin{array}{cc}
      {\mathbf B}_{22} & -{\mathbf B}_{12} \\
      -{\mathbf B}_{21} & {\mathbf B}_{11} \\
      \end{array}
\right)
$$
and (\ref{eq:lemma:elementary-matrix-properties-ii}), we see that $\operatorname*{det}{\mathbf B}$
is positive. The well-known fact that 
$$
\|{\mathbf B}^{-1}\|_2 \leq \alpha^{-2}
$$
allows us to conclude 
$$
\left| \frac{1}{\operatorname*{det}{\mathbf B}}\right|  \|{\mathbf C}\|_2 
= \|{\mathbf B}^{-1}\|_2 \leq \alpha^{-2},
$$
which implies 
$$
\left| \frac{1}{\operatorname*{det}({\mathbf B})}\right| \leq \frac{1}{\alpha^2 \|{\mathbf C}\|_2}.
$$
Estimating $\|{\mathbf C}\|_2 \ge \max\{|{\mathbf C}_{11}|,|{\mathbf C}_{22}|\}$ and 
recalling $\operatorname{det}{\mathbf B} > 0$ concludes the argument.
\end{proof}

\section{Proofs for Section~\ref{case_2}}
\label{appendix:case_2}
\begin{lem}
\label{lem3}If $u$ satisfies
\begin{equation*}
\left\Vert u^{(n)}\right\Vert _{\infty ,I}\leq CK^{n}\max \{n,\mu
^{-1}\}^{n}\quad \forall \;n\in \mathbb{N}_{0},
\end{equation*}
for some positive constants $C,K$ independent of $\mu $, then its complex
extension (denoted by $u(z)$) satisfies
\begin{equation*}
\left\vert u(z)\right\vert \leq Ce^{\beta \operatorname*{dist}(z,I)/\mu },
\end{equation*}
provided $\operatorname*{dist}(z,I)$ is sufficiently small.
\end{lem}

\begin{proof} Fix $x\in I$ and let $B_{r}(x)$ be the ball of radius $r$
centered at $x$. \ Then, by Taylor's theorem, we have for $z\in B_{r}(x)$
\begin{eqnarray*}
\left\vert u(z)\right\vert &\leq &\overset{\infty }{\underset{k=0}{\sum }}%
\left\vert \frac{u^{(k)}(x)}{k!}(z-x)^{k}\right\vert \leq \overset{\lfloor
1/\mu \rfloor }{\underset{k=0}{\sum }}C\frac{\mu ^{-k}}{k!}K^{n}r^{k}+%
\overset{\infty }{\underset{k=\lfloor 1/\mu \rfloor }{\sum }}C\frac{k^{k}}{k!%
}K^{k}r^{k} \\
&\leq &C\overset{\lfloor 1/\mu \rfloor }{\underset{k=0}{\sum }}\left( \frac{%
rK}{\mu }\right) ^{k}\frac{1}{k!}+C\overset{\infty }{\underset{k=\lfloor
1/\mu \rfloor }{\sum }}e^{k}K^{k}r^{k}.
\end{eqnarray*}
If $r<1/(2eK)$ then the second sum above is bounded and we get%
\begin{equation*}
\left\vert u(z)\right\vert \leq Ce^{rK/\mu }+C\leq \widehat{C}e^{rK/\mu }.
\end{equation*}
\end{proof}
\subsection{Proof of Lemma~\ref{lem4}}
\label{appendix:case_2_lem4}
We work out the details here: 
We reduce to the case of homogeneous boundary conditions
by introducing the linear function $u_0$ with $u_0(0) = g_{-}$
and $u_0(1) = g_{+}$. Then the difference $\tilde u:= u  -u_0$ 
solves
$$
-\mu^2 \tilde u^{\prime\prime} + c \tilde u = 
g - c u_0 =:\tilde g
\qquad \tilde u(0) = \tilde u(1) = 0
$$
with $\|\tilde g\|_{L^\infty(I)} \leq C_g + |g_-| + |g_+|$. 
By the comparison principle (see, e.g., \cite[Lemma~{2.1}]{kopteva07}
for the present form), 
we have 
\begin{equation}
\label{eq:lem4-10}
\|\tilde u\|_{L^\infty(I)} \leq \frac{1}{\underline{c}} \|\tilde g\|_{L^\infty(I)}. 
\end{equation}
From the differential equation, we furthermore get 
\begin{equation}
\label{eq:lem4-20}
\|\tilde u^{\prime\prime}\|_{L^\infty(I)} 
\leq \mu^{-2} \|\tilde g\|_{L^\infty(I)} + \frac{\|c\|_{L^\infty(I)} }{\underline{c}}
\|\tilde g\|_{L^\infty(I)}
\leq C \mu^{-2} \|\tilde g\|_{L^\infty(I)},
\end{equation}
where $C$ is a constant that depends solely on $c$. 
By an interpolation inequality in H\"older spaces 
(see \cite[Thm.~{3.2.1}]{krylov96}) we get 
from (\ref{eq:lem4-10}), (\ref{eq:lem4-20}), the estimate 
\begin{equation}
\label{eq:lem4-30}
\|\tilde u^{(1)}\|_{L^\infty(I)} \leq C \mu^{-1} \|\tilde g\|_{L^\infty(I)},
\end{equation}
for a constant $C > 0$ that depends only on $c$. Combining the above 
estimates together with $u = \tilde u + u_0$ and recalling that $u_0$ is linear,
we have 
\begin{equation}
\label{eq:lem4-40}
\|u^{(n)}\|_{L^\infty(I)} \leq C_u\max\{n,\mu^{-1}\}^n, 
\qquad n \in \{0,1,2\},
\end{equation}
where $C_u = C C_{\tilde g}$ for a constant $C$ that depends solely on the 
function $c$. Higher order estimates for $u$ are obtained from the 
differential equation in the standard way. 
Differentiating the differential equation satisfied by $u$ yields 
\begin{equation}
\label{eq:lem4-50}
-\mu^{-2} u^{(n+2)} = g^{(n)} - \sum_{\nu=0}^n 
\binom{n}{\nu} c^{(n-\nu)} u^{(\nu)}.
\end{equation}
The analyticity of $c$ implies the existence of $C_c$, $\gamma_c > 0$ 
such that 
\begin{equation}
\label{eq:lem4-55}
\|c^{(n)}\|_{L^\infty(I)} \leq C_c \gamma_c^n n! \qquad \forall n \in \N_0. 
\end{equation}
We claim 
\begin{equation}
\label{eq:lem4-60}
\|u^{(n)}\|_{L^\infty(I)} \leq C_u \gamma^n \max\{n,\mu^{-1}\}^n 
\qquad \forall n \in \N_0. 
\end{equation}
This estimate is valid for $n \in \{0,1,2\}$ by (\ref{eq:lem4-40}). 
To see that it is valid in general, one proceeds by induction on $n$. 
From (\ref{eq:lem4-50}) we get in view (\ref{eq:lem4-55}) 
\begin{eqnarray*}
\mu^2 \|u^{(n+2)}\|_{L^\infty(I)} 
&\leq& 
C_{\tilde g} \gamma^n \max\{n,\mu^{-1}\}^n + 
\sum_{\nu=0}^n \binom{n}{\nu} C_c \gamma_c^{\nu} (\nu)! C_c \gamma^{n-\nu}
\max\{\nu,\mu^{-1}\}^{n-\nu} \\
&\leq& 
C_{\tilde g} \gamma^n \max\{n,\mu^{-1}\}^n + 
C_c C_u \sum_{\nu=0}^n n^{\nu} \gamma_c^{\nu} \gamma^{n-\nu}
\max\{\nu,\mu^{-1}\}^\nu \\
&\leq& 
C_{\tilde g} \gamma^n \max\{n,\mu^{-1}\}^n + 
C_c C_u \max\{n,\mu^{-1}\}^n \gamma^n 
\sum_{\nu=0}^n \left(\frac{\gamma_c}{\gamma}\right)^{\nu}\\
&\leq& 
C_{\tilde g} \gamma^n \max\{n,\mu^{-1}\}^n + 
C_c C_u \max\{n,\mu^{-1}\}^n \gamma^n \frac{1}{1-\gamma_c/\gamma} \\
&\leq& 
C_u \gamma^{2+n} \max\{n,\mu^{-1}\}^n 
\left[\gamma^{-2} \left\{\frac{C_{\tilde g}}{C_u}  + \frac{C_c}{1-\gamma_c/\gamma}
\right\}\right].
\end{eqnarray*}
Since the expression in square brackets is bounded by $1$ for 
$\gamma \ge \gamma_0$, if we make $\gamma_0$ sufficiently large, we have
the desired estimate (\ref{eq:lem4-60})
for $u$. 
\subsection{Proof of Lemma~\ref{lem5}}
\label{appendix:proof_of_lem5}
Set $\delta:= 2 \widehat \delta \leq 1/e$. 
Define $F:k\mapsto \delta ^{k}\left( \frac{i+2}{i+1-k}\right) ^{i-k}$ and
note that
\begin{equation*}
\ln (F)=k\ln \delta +(i-k)\left( \ln (i+2)-\ln (i+1-k)\right) ,
\end{equation*}
\begin{eqnarray*}
\left( \ln (F)\right) ^{\prime } &=&\ln \delta -\ln (i+2)+\ln (i+1-k)+\frac{%
i-k}{i+1-k}  \\
&\leq&\ln \delta +\ln \left( \frac{i+1-k}{i+2}\right) + 1 
\leq 1+\ln \delta .
\end{eqnarray*}
Since $\delta \leq 1/e$ we have $\left( \ln (F)\right) ^{\prime }\leq0$, 
and $F$ is monotone decreasing which gives
\begin{equation*}
\overset{i}{\underset{k=0}{\sum }}\widehat{\delta }^k \left( \frac{i+2}{%
i+1-k}\right) ^{i-k}\leq 
\overset{i}{\underset{k=0}{\sum }}2^{-k} e 
\leq 2e. 
\end{equation*}
%
\subsection{Proof of Theorem~\ref{thm2}}
\label{appendix:case_2_thm2}
We will only consider the inner expansions 
$\widehat u_i^L$, $\widehat v_i^L$ 
at the left endpoint
of $I$, since for the expansions at the right endpoint the arguments are
almost identical. \ The proof is by induction on $i$. For $i=0$ we have
from (\ref{14})--(\ref{17}),
\begin{equation}
\left\{ 
\begin{array}{c}
-\mu ^{2}v_{0}^{\prime \prime }+\frac{\left(
a_{22}a_{11}-a_{12}a_{21}\right) }{a_{11}}v_{0}=g-\frac{a_{21}}{a_{11}}f \\ 
v_{0}(0)=v_{0}\left( 1\right) =0%
\end{array}%
\right. ,  \label{22}
\end{equation}
\begin{equation}
u_{0}=\frac{1}{a_{11}}\left( f-a_{12}v_{0}\right) ,  \label{23}
\end{equation}
\begin{equation}
\left\{ 
\begin{array}{c}
-\left( \widehat{u}_{0}^{L}\right) ^{\prime \prime }+a_{11}(0)\widehat{u}%
_{0}^{L}=0 \\ 
\widehat{u}_{0}^{L}(0)=-u_{0}(0)%
\end{array}%
\right. ,  \label{24}
\end{equation}
\begin{equation}
\widehat{v}_{2}^{L}=
- \int_z^\infty\int_t^\infty a_{21}(0) \widehat u_0^L(\tau)\,d\tau\,dt.
\label{25}
\end{equation}
Combining Lemmas~\ref{lemma:leibniz-rule} and \ref{lem4}, we can find 
a constant $C_{4}>0$ such that $v_{0}$
satisfies (\ref{21}), which in turn shows that $u_{0}$ satisfies (\ref{20})
by Lemma~\ref{lemma:leibniz-rule}. 
The solution formula for $\widehat{u}_0^L$ then shows
that (\ref{19}) is valid. Finally, Lemma~\ref{lemma:two-anti-derivatives}
establishes the desired bound (\ref{18}) for $\widehat v_2$. 

So, assume (\ref{18})--(\ref{21}) hold for up to $i\ge 0$ 
and establish them for $i+1$. We will choose the constants 
$C_i$ such that the ratios $C_2/C_1$, $C_3/C_2$, and 
$C_4/C_3$ are sufficiently small. Furthermore, the constants $K_i$
and $C_u$, $C_v$ are sufficiently large and are selected such that 
\begin{subequations}
\label{eq:simplifying-assumptions-section-B}
\begin{eqnarray}
K_1 = K_2 && \qquad K_3 = K_4, 
\qquad  C_u = C_v \\
C_u^2 K_1 = K_3 && C_u = C_v > 2/\min\{1/a_{11}(0), 1/a_{11}(1)\}.  
\end{eqnarray}
\end{subequations}
We start by bounding $v_{i+1}$ and $u_{i+1}$. 
We first consider $v_{i+1}$, which satisfies (see eq. (\ref{14})): 
\begin{equation*}
\left\{ 
\begin{array}{c}
v_{i+1}^{\prime \prime }+\frac{(a_{22}a_{11}-a_{12}a_{21})}{a_{11}}%
v_{i+1}=-\mu^2 \frac{a_{21}}{a_{11}}u_{i-1}^{\prime \prime }\smallskip \\ 
v_{i+1}(0)=-\widehat{v}_{i+1}^{L}(0)\;,\;v_{i+1}(1)=-\widehat{v}_{i+1}^{R}(0)%
\end{array}%
\right. .
\end{equation*}
To that end, we assume, as we may, that $\gamma$ is sufficiently large
for Lemma~\ref{lemma:leibniz-rule} to be applicable for the function
$u_{i-1}$. Together with the induction hypothesis, it then yields $\forall n \in \N_0$ (with the constant $C^\prime$ of 
Lemma~\ref{lemma:leibniz-rule} for the function $g = a_{21}/a_{11}$)
\begin{eqnarray*}
\|\mu^2 \left(\frac{a_{21}}{a_{11}} u_{i-1}^{\prime\prime}\right)^{(n)}\|_{L^\infty(I)} 
&\leq& \mu^2 C^\prime 
C_3 K_3^{i-1} ( (i+1+n)\mu + 1)^{i-1} \frac{(i-1)^{i-1}}{(i-1)!}\gamma^{n+2} 
\max\{n+2,\mu^{-1}\}^{n+2}.
\end{eqnarray*}
%
%
From
\begin{eqnarray*}
\mu^2 \max\{n+2,\mu^{-1}\}^{n+2} 
&\leq& \max\{n,\mu^{-1}\}^n \left(\frac{n+2}{n}\right)^{n+2} 
\max\{(n+2)\mu,1\}^2 \\
&\leq & C^{\prime\prime} \max \{n,\mu^{-1}\}^n ((i+n+1)\mu+1)^2,
\end{eqnarray*}
for some $C^{\prime\prime} > 0$ independent of $n$, 
we get $\forall n \in \N_0$
\begin{eqnarray*}
\|\mu^2 \left(\frac{a_{12}}{a_{11}} u_{i-1}^{\prime\prime}\right)^{(n)}\|_{L^\infty(I)} 
&\leq&
C_3 \gamma^2 C^\prime C^{\prime\prime} K_3^{-2} 
K_3^{i+1} ((i+1+n)\mu + 1)^{i+1} \frac{(i-1)^{i-1}}{(i-1)!}\gamma^{n} 
\max\{n,\mu^{-1}\}^{n}.
\end{eqnarray*}
The induction hypothesis for $\widehat v_{i+1}^L$ and $\widehat v_{i+1}^R$
and Lemma~\ref{lem4} therefore produce (with $\widetilde C$ of 
Lemma~\ref{lem4}), for all $n \in \N_0$
\begin{eqnarray*}
\|v_{i+1}^{(n)}\|_{L^\infty(I)} &\leq &
\widetilde C \gamma^n \max\{n,\mu^{-1}\}^n \frac{(i-1)^{i-1}}{(i-1)!}\\
&&\mbox{}\times 
\left[ 
C_3 K_3^{i-1} ( (i+1+n)\mu + 1)^{i+1} \gamma^{2} 
+ 2 C_1 K_1^{i-1} \left(\mu + \frac{1}{i-1}\right)^{i-1} 
(i-1)^{(i-1)}
C_v^{2(i-1)}
\right]\\
&\leq& C_4 K_4^{i+1}
\gamma^n \max\{n,\mu^{-1}\}^n (\mu(i+n)+1)^{i+1} \frac{(i+1)^{(i+1)}}{(i+1)!}\\
&&\mbox{}\times 
\left[ 
C^{(IV)} \widetilde C \frac{C_3}{C_4} \gamma^2 K_3^{-2} (K_3/K_4)^{i+1}
+ 
2 C^{(IV)}\widetilde C  \left(\frac{1}{C_v^2 K_1}\right)^2 
 \frac{C_1}{C_4} (C_v^2 K_1/K_4)^{i+1} 
\right] \\
&\leq& C_4 K_4^{i+1}
\gamma^n \max\{n,\mu^{-1}\}^n (\mu(i+n)+1)^{i+1} \frac{(i+1)^{(i+1)}}{(i+1)!}\\
&&\mbox{}\times 
\left[ 
C^{(IV)} \widetilde C \frac{C_3}{C_4} \gamma^2 K_3^{-2} 
+ 
2 C^{(IV)}\widetilde C  \left(\frac{1}{C_v^2 K_1}\right)^2 
\right],
\end{eqnarray*}
where we made use of 
$\frac{(i-1)^{i-1}}{(i-1)!} \leq C^{(IV)} \frac{(i+1)^{i+1}}{(i+1)!}$
and the simplifying assumptions on the relations between the 
constants $K_i$ in (\ref{eq:simplifying-assumptions-section-B}). 
For the induction argument to work, we have to require that 
the expression in square brackets is bounded by $1$: 
\begin{equation}
\label{eq:condition-1-section-B}
\left[ 
C^{(IV)} \widetilde C \frac{C_3}{C_4} \gamma^2 K_3^{-2} 
+ 
2 C^{(IV)}\widetilde C  \left(\frac{1}{C_v^2 K_1}\right)^2 
\right] \stackrel{!}{\leq} 1.
\end{equation} 
We will collect further conditions on the constants $C_i$ and $K_i$ 
and see at the end of the proof that these conditions can be met. 

Next, we turn to to $u_{i+1}$, which is given by (\ref{15}):  
\begin{equation*}
u_{i+1}(z) = - \frac{1}{a_{11}(z)} 
\left[ \mu^2  u_{i-1}^{\prime \prime }(z)  +
a_{12}(z)  v_{i+1}(z) \right]. 
\end{equation*}
Lemma~\ref{lemma:leibniz-rule}, the induction hypothesis and the just
proved bounds for $v_{i+1}$ then give  for a constant $C$ that depends
solely on the coefficients $a_{ij}$ of the differential operator,
\begin{eqnarray*}
\|u_{i+1}^{(n)}\|_{L^\infty(I)} &\leq & C_3 K_3^{i+1} 
(\mu(i+1+n)+1)^{i+1} \max\{n,\mu^{-1}\}^n \frac{(i+1)^{i+1}}{(i+1)!}
\left[C
\gamma^2 K_3^{-2}  + C \frac{C_4}{C_3} \left(\frac{K_4}{K_3}\right)^{i+1}
\right].
\end{eqnarray*}
In view of $K_3 = K_4$ by (\ref{eq:simplifying-assumptions-section-B}), 
we recognize a second condition for the induction argument, namely, 
\begin{equation}
\label{eq:condition-2-section-B}
\left[C
\gamma^2 K_3^{-2}  + C \frac{C_4}{C_3} \right]
\stackrel{!}{\leq} 1. 
\end{equation}
Next, we consider $\widehat{u}_{i+1}^{L}$ which satisfies (see (\ref{16})):
\begin{equation}
\left\{ 
\begin{array}{c}
-\left( \widehat{u}_{i+1}^{L}\right) ^{\prime \prime }+a_{11}(0)\widehat{u}%
_{i+1}^{L}=-a_{12}(0)\widehat{v}_{i+1}^{L}-\underset{k=1}{\overset{i+1}{\sum 
}}\left( \frac{a_{11}^{(k)}(0)}{k!}
\mu^k \widehat{x}^{k}\widehat{u}_{i+1-k}^{L}+%
\frac{a_{12}^{(k)}(0)}{k!}\mu^k \widehat{x}^{k}\widehat{v}_{i+1-k}^{L}\right)\\ 
\widehat{u}_{i+1}^{L}(0)=-u_{i+1}(0)\;,\;\widehat{u}_{i+1}^{L}\rightarrow 0%
\text{ as }\widehat{x}\rightarrow \infty%
\end{array}%
\right. .  \label{26}
\end{equation}
We will estimate the right-hand side of (\ref{26}). To that end, since $1 \leq k \leq i+1$,
we start with the observation
\begin{eqnarray*}
\mu^k \left (\mu + \frac{1}{i+2-k}\right)^{i+1-k} 
&\leq& 
\mu^k \left (\mu + \frac{1}{i+2}\right)^{i+1-k} 
\left(\frac{i+2}{i+2-k}\right)^{i+1-k} \\
&\leq& 
\left (\mu + \frac{1}{i+2}\right)^{i+1} 
\left(\frac{i+2}{i+2-k}\right)^{i+1-k},
\end{eqnarray*}
so that we  can estimate 
\begin{eqnarray*}
\sum_{k=1}^{i+1} \frac{|a_{11}^{(k)}(0)|}{k!} 
|z|^k\mu^k |\widehat u_{i+1-k}^L| &\leq& 
C_a C_2 K_2^{i+1} 
\left(\mu + \frac{1}{i+2}\right)^{i+1} 
\frac{1}{i!}
(C_v (i+1) + |z|)^{2i+1}
e^{-\beta\operatorname*{Re}(z)} \\
&&\mbox{} \times 
\sum_{k=1}^{i+1} 
\left(\frac{\gamma_a}{K_2}\right)^k \left(\frac{i+2}{i+2-k}\right)^{i+1-k}.
\end{eqnarray*}
If $K_2$ is sufficiently large, then 
the sum is bounded by $2\frac{\gamma_a}{K_2}e $ 
by Lemma~\ref{lem5}.  
Analogously, we get (note that the case $i = 0$ leads to empty sums)
\begin{eqnarray*}
&&
\sum_{k=0}^{i+1} \frac{|a_{12}^{(k)}(0)|}{k!} 
|z|^k\mu^k |\widehat v_{i+1-k}^L| = 
\sum_{k=0}^{i-1} \frac{|a_{12}^{(k)}(0)|}{k!} 
|z|^k\mu^k |\widehat v_{i+1-k}^L| \\
&& \leq 
C_a C_1 K_1^{i-1} 
\left(\mu + \frac{1}{i}\right)^{i-1} 
\frac{1}{i!}
(C_v (i-1) + |z|)^{2i-1}
e^{-\beta\operatorname*{Re}(z)} 
\sum_{k=0}^{i-1} 
\left(\frac{\gamma_a}{K_1}\right)^k \left(\frac{i}{i-k}\right)^{i-1-k}\\
&& \leq 
K_1^{-2} C^{\prime\prime\prime}  C_a C_1 K_1^{i+1} 
\left(\mu + \frac{1}{i}\right)^{i+1} 
\frac{1}{i!}
(C_v (i+1) + |z|)^{2i+1}
e^{-\beta\operatorname*{Re}(z)}, 
\end{eqnarray*}
where we appealed again to Lemma~\ref{lem5} and used that $K_1 = K_2$ is 
sufficiently large to bound the sum by $2e$, and noticed additionally that
$$
\left(\mu + \frac{1}{i}\right)^{i-1} 
\leq 
\left(\mu + \frac{1}{i+2}\right)^{i-1} 
\left(\frac{i+2}{i}\right)^{i-1}
\leq 
\left(\mu + \frac{1}{i+2}\right)^{i+1}  (i+1)^2 
\underbrace{
\sup_{i \ge 1} \left(\frac{i+2}{i+1}\right)^2 
\left(\frac{i+2}{i}\right)^{i-1}
}_{=:C^{\prime\prime\prime}/(2e)},
$$
where the last supremum is finite. 

These estimates allow us to bound $\widehat u_{i+1}^L$ with the 
aid of Lemma~\ref{lemma:scalar-bvp-constant-coefficients} to arrive 
at 
\begin{eqnarray*}
|\widehat u_{i+1}^L(z)| &\leq& 
C_2 K_2^{i+1} \left(\mu + \frac{1}{i+2}\right)^{i+1} e^{-\beta \operatorname*{Re}(z)} 
\frac{1}{(i+1)!} (C_v (i+1) + |z|)^{2(i+1)} \\
&& \mbox{}\times 
\left[
C \frac{\gamma_a}{K_2} + K_1^{-2} \frac{C_1}{C_2} \left(\frac{K_1}{K_2}\right)^{i+1} + 
\frac{C_3}{C_2} \left(\frac{K_3}{C_v^2 K_1}\right)^{i+1} 
\right].
\end{eqnarray*}
Together with the simplifying 
assumptions (\ref{eq:simplifying-assumptions-section-B}) we see 
that for the induction argument to work, we need to require 
\begin{equation}
\label{eq:condition-3-section-B}
\left[
C \frac{\gamma_a}{K_2} + K_1^{-2} \frac{C_1}{C_2}  + 
\frac{C_3}{C_2} \right] \stackrel{!}{\leq} 1. 
\end{equation}

Finally, for $\widehat v_{i+3}^L$ we have from (\ref{17})
\begin{eqnarray*}
\left\vert \widehat{v}_{i+3}^{L}(z)\right\vert  &=&\left\vert 
\underset{z}{\overset{\infty }{\int }}\underset{t}{\overset{\infty }{\int }}%
\underset{k=0}{\overset{i+1}{\sum }}\left( \frac{a_{21}^{(k)}(0)}{k!}
\mu^k \widehat{%
\tau }^{k} \widehat{u}_{i+1-k}^{L}
+\frac{a_{22}^{(k)}(0)}{k!}\mu^k \widehat{\tau }^{k}%
\widehat{v}_{i+1-k}^{L}\right) d\widehat{\tau }dt\right\vert  \\
&\leq &\underset{k=0}{\overset{i}{\sum }}\frac{1}{k!}\max \left\{
\left\vert a_{21}^{(k)}(0)\right\vert ,\left\vert a_{22}^{(k)}(0)\right\vert
\right\} 
\left\{
\left\vert \underset{z}{\overset{\infty }{\int }}\underset{t}{%
\overset{\infty }{\int }} \widehat{\tau }^{k}\widehat{u}_{i-k}^{L}
d\widehat{\tau}dt\right\vert + 
\left\vert \underset{z}{\overset{\infty }{\int }}\underset{t}{%
\overset{\infty }{\int }}
\widehat{\tau }^{k}\widehat{v}_{i-k}^{L} d\widehat{\tau }%
dt\right\vert\right\} .
\end{eqnarray*}
Proceeding analogously as before, we find that this last sum can be bounded as
\begin{eqnarray*}
|rhs(z)| &\leq & C_1 K_1^{i+1} \left(\mu + \frac{1}{i+2}\right)^{i+1} 
\frac{1}{(i+1)!} (C_v (i+1) + |z|)^{2(i+1)} e^{-\beta\operatorname*{Re}(z)}
\left[C 
\frac{C_2}{C_1} \left(\frac{K_2}{K_1}\right)^{i+1}  + C K_1^{-2}
\right], 
\end{eqnarray*}
where the constant $C > 0$ is suitably chosen. 
Lemma~\ref{lemma:two-anti-derivatives} then gives 
\begin{eqnarray*}
\left\vert \widehat{v}_{i+3}^{L}(z)\right\vert  &\leq&
 K_1^{i+1} \left(\mu + \frac{1}{i+2}\right)^{i+1} 
\frac{1}{(i+1)!} (C_v (i+1) + |z|)^{2(i+1)} e^{-\beta\operatorname*{Re}(z)}\\
&& \mbox{}\times 
\left[C
\frac{C_2}{C_1} \left(\frac{K_2}{K_1}\right)^{i+1}  + C K_1^{-2}
\right] 
\frac{1}{\beta^2} \left(\frac{1}{1-(2(i+1))/(\beta C_v(i+1))}\right)^2. 
\end{eqnarray*}
Hence, by our assumption (\ref{eq:simplifying-assumptions-section-B}), 
we see that we can find $C^\prime$ (depending solely on $\beta$) such that 
our induction argument will work if we can satisfy 
\begin{equation}
\label{eq:condition-4-section-B}
C^\prime\left[C
\frac{C_2}{C_1}   + C K_1^{-2}
\right] 
\stackrel{!}{\leq} 1. 
\end{equation}
In total, we have completed the induction argument if we can select
the constants $C_i$ and $K_i$ such that 
(\ref{eq:condition-1-section-B}), 
(\ref{eq:condition-2-section-B}),
(\ref{eq:condition-3-section-B}), and 
(\ref{eq:condition-4-section-B}) are satisfied. 
Inspection shows that this is the case by taking the $K_i$ sufficiently
large and appopriately controlling the ratios $C_2/C_1$, $C_3/C_2$,
and $C_4/C_3$. 
%
\subsection{Proof of Theorem~\ref{thm:bdy-layer-fct-entire-case-II}}
\label{appendix:proof_of_thm:bdy-layer-fct-entire-case-II}

From Lemma~\ref{lemma:bdy-layer-fct-entire-case-II} and the estimates 
(\ref{eq:lemma:estimate-remainder-case-II-1}),  
(\ref{eq:lemma:estimate-remainder-case-II-2}), we get 
\begin{eqnarray*}
\left|\left(\widehat{\mathbf U}_{BL}^{M}\right)^{(n)}(\widehat{x})\right| 
&\leq &C e^{\beta \widehat x} \gamma_2^n 
\sum_{i=0}^{M}\left(\frac{\varepsilon}{\mu}\right)^{i} 
\left( i \mu+1\right)^{i}\gamma_1^i 
\leq C e^{\beta \widehat x} \gamma_2^n 
\sum_{i=0}^{M}
\left( M \varepsilon +\frac{\varepsilon}{\mu}\right)^{i}\gamma_1^i  
\leq C \gamma_2^n,
\end{eqnarray*}
where, in the last step we have used 
the assumption that $\varepsilon(M+1+\varepsilon/\mu)$ is sufficiently small 
so that the sum can be estimated by a convergent geometric series.
\subsection{Proof of Lemma~\ref{lemma:estimate-remainder-case-II}}
\label{appendix:case_2_estimate-remainder-case-II}
We start by noting that the change of summation index $k$ to 
$\ell = i-k$ leads us to having to estimate
$$
S:= \sum_{i\geq M+1}
\left(\frac{\varepsilon}{\mu}\right)^{i}
\sum_{\ell =0}^{M}\mu^{i-\ell} \widehat{x}^{i-\ell }\gamma
_{A}^{i-\ell }\left( \mu +\frac{1}{\ell+1 }\right) ^{\ell }\frac{1}{\ell!}
(C_{1}\ell +\widehat{x})^{2\ell}e^{-\beta \widehat{x}}. 
$$
We start with the elementary observation 
(cf. Lemma~\ref{lemma:elementary-properties-of-factorial})
\begin{equation*}
(C_{1}\ell +\widehat{x})^{2\ell}\leq 
2^{\ell}(C_{1}\ell )^{2\ell}
+2^{\ell}\widehat{x}^{2 \ell }
\leq \gamma ^{\ell }\left( \ell^{2\ell }+\widehat{x}%
^{2\ell }\right), 
\end{equation*}
for suitable $\gamma >0$. Next, we estimate 
\begin{equation*}
\frac{1}{\ell!}(C_{1}\ell+\widehat{x})^{2\ell}e^{-\beta \widehat{x}%
}\leq C\gamma ^{\ell }[\ell ^{\ell }e^{\ell}
+\frac{1}{\ell!}\widehat{x}%
^{2\ell }e^{-\widehat{x}\beta /4}]e^{-3\beta \widetilde{x}/4}\leq C\tilde{%
\gamma}^{\ell }\ell^{\ell}e^{-3\beta \widetilde{x}/4},
\end{equation*}
where we suitably chose $\tilde{\gamma}$ independent of $\ell $ and 
$\widetilde{x}$. We therefore conclude
\begin{eqnarray*}
&&\sum_{i\geq M+1}\left(\frac{\varepsilon}{\mu}\right)^{i}
\sum_{\ell =0}^{M}\mu^{i-\ell} \widehat{x}^{i-\ell
}\gamma _{A}^{i-\ell }\left( \mu+\frac{1}{\ell + 1}\right) ^{\ell }
\frac{1}{\ell!}(C_{1}\ell +\widehat{x})^{2\ell}
e^{-\beta \widehat{x}} \\
&\leq &Ce^{-3\beta \widehat{x}/4}
\sum_{i\geq M+1}(\varepsilon \widehat x \gamma_A)^i
\sum_{\ell=0}^{M}\widehat{x}^{- \ell }\gamma _{A}^{-\ell }
\left( 1+\frac{1}{\mu(\ell+1)}%
\right) ^{\ell }\ell^{\ell} \tilde{\gamma}^{\ell } \\
&\leq &Ce^{-3\beta \widehat{x}/4} 
\frac{1}{1-X} (\varepsilon \widehat x \gamma_A)^{M+1} 
\sum_{\ell=0}^{M}\widehat{x}^{- \ell }\gamma _{A}^{-\ell }
\left( \ell +\frac{1}{\mu}%
\right) ^{\ell } \tilde{\gamma}^{\ell } \\
&\leq &\frac{C}{1-X}e^{-\beta \widehat{x}/2}\varepsilon ^{M+1}\sum_{\ell
=0}^{M}(\gamma _{A}\widehat{x})^{M+1-\ell }e^{-\beta \widehat{x}/4}\tilde{%
\gamma}^{\ell }(\ell +1/\mu )^{\ell }.
\end{eqnarray*}
Again, elementary considerations 
(cf. Lemma~\ref{lemma:elementary-properties-of-factorial}) 
show $\widehat{x}^{M+1-\ell }e^{-\beta 
\widehat{x}/4}\leq \gamma _{1}^{M+1-\ell }(M+1-\ell )^{M+1-\ell }$ 
for suitable $\gamma_1 > 0$, so that we arrive at 
\begin{eqnarray*}
S &\leq &
\frac{C}{1-X}e^{-\beta \widehat{x}/2}\varepsilon ^{M+1}\sum_{\ell
=0}^{M}(\gamma_A \gamma _{1})^{M+1-\ell }(M+1-\ell)^{M+1-\ell }
\tilde{\gamma}^{\ell }(\ell
+1/\mu )^{\ell } \\
&\leq &\frac{C\mu}{1-X}e^{-\beta \widehat{x}/2}\varepsilon ^{M+1}\gamma
_{2}^{M+1}(M+1+1/\mu )^{M+1},
\end{eqnarray*}
where, in the last step we have used the convexity of the function 
$$
\ell \mapsto a^{M+1-\ell} (M+1-\ell)^{M+1-\ell} (\ell +1/\mu)^\ell
$$
on the intervall $[0,M]$ -- see also the related Lemma~\ref{lemma:convexity}. 
%
\section{Proofs for Section~\ref{case_3}}
\label{appendix:case_3}
\subsection{Proof of Lemma~\ref{lemma:foo}}
\label{appendix:proof_of_lemma:foo}

The proof consists of showing that the terms of the sum can be bounded by the 
terms of a convergent geometric series if $\delta$ is sufficiently small. To simplify
some notation, we restrict our attention to the case $i \ge 2$ (the cases $i=0$ and $i=1$
being easily seen). We denote
the terms of the sum by $F(n)$ and 
note that it suffices to study the sum $\sum_{n=1}^{i-1} F(n)$ since the case $n = 0$ 
produces the term 
$$
F(0) =\delta^{i-1} \frac{(i+1+m)^{i+1+m}}{(i+m)^{i}(m+1)^{m+1}} = 
\delta^{i-1} \left(\frac{i+1+m}{i+m}\right)^{i} \left(\frac{i+1+m}{m+1}\right)^{m+1}
\leq \delta^{i-1} e e^i, 
$$
which is bounded uniformly in $i$ if $\delta \leq \delta_0 \leq e^{-1}$. 

We define, for $n \in [1,i-1]$,
\begin{eqnarray*}
f(n) &:=& \ln F(n) \\
&=& (i-1-n)\ln \delta + n \ln (n+m) + (i+n+1+m) \ln (i+n+1+m)- \\
&\mbox{}& - i \ln (i+m)  - (2n+1+m) \ln (2n+1+m)
\end{eqnarray*}
and compute 
\begin{eqnarray*}
f^\prime(n) &=& \ln \frac{(n+m)(n+i+1+m)}{(2n+1+m)^2} - \ln \delta  - 1 + \frac{n}{n+m}. 
\end{eqnarray*}
It is easy to see that $f^\prime(n) \ge -1 - \ln 4 - \ln \delta$ for $n \ge 1$ so that 
we can find a constant $c > 0$ with $f^\prime(n) \ge c > 0$ for all $n\ge 1$ by taking
$\delta$ sufficiently small.  By the mean value theorem, 
we therefore get for each $n$, that 
$
\ln \frac{F(n+1)}{F(n)} = f(n+1) - f(n) \ge c. 
$
Hence, 
$
F(n) \leq e^{-c} F(n+1). 
$
Iterating this estimate, we get 
$$
F(n) \leq e^{-(i-1-n)c} F(i-1). 
$$
Hence, 
$$
\sum_{n=1}^{i-1} F(n) \leq \sum_{n=1}^{i-1} e^{-(i-1-n)c}F(i-1) \leq \frac{1}{1-e^{-c}} F(i-1). 
$$
The argument is concluded by noting that $F(i-1)$ is bounded uniformly in $i$ and $m$
as the following rearrangement shows: 
$$
F(i-1) = \frac{(i-1+m)^{i-1}}{(i+m)^i} \frac{(2i+m)^{2i+m}}{(2i+m-1)^{2i+m-1}}
= 
\frac{2i+m}{i+m} \left(\frac{i-1+m}{i+m}\right)^{i-1} \left(\frac{2i+m}{2i+m-1}\right)^{2i+m-1}. 
$$
\subsection{Proof of Theorem~\ref{thm:higher-derivatives-in-exponentially-weighted-spaces}}
\label{appendix:case_3_thm:higher-derivatives-in-exponentially-weighted-spaces}

Specifically, 
differentiating $m$ times yields 
(we write $\widetilde {\mathbf U}_i$ instead of $\widetilde {\mathbf U}^L_i$)
with ${\mathbf B} = {\mathbf A}(0)$
\begin{eqnarray}
\nonumber 
- {\mathbf E} \widetilde{\mathbf U}_i^{(m+2)}
&=& {\mathbf B} \widetilde {\mathbf U}_i^{(m)}
+ \sum_{n=0}^{i-1} \left(\widetilde x^{i-n} {\mathbf A}_{i-n} 
\widetilde{\mathbf U}_n\right)^{(m)}\\
\label{eq:thm:higher-derivatives-in-exponentially-weighted-spaces-10}
&=& {\mathbf B} \widetilde {\mathbf U}_i^{(m)}
+ \sum_{n=0}^{i-1} {\mathbf A}_{i-n}
\sum_{j=0}^{\min\{m,i-n\}} \binom{m}{j}\binom{i-n}{j} j! \widetilde x^{i-n-j} 
\widetilde{\mathbf U}_n^{(m-j)}.
\end{eqnarray}
We abbreviate 
$$
{\mathbf F}_n:= \sum_{j=0}^{\min\{m,i-n\}} 
\binom{m}{j}\binom{i-n}{j} j! \widetilde x^{i-n-j} 
\widetilde{\mathbf U}_n^{(m-j)},
$$
and we proceed as in the proof of 
Theorem~\ref{thm2a}. 
The induction hypothesis yields, for 
any $\tilde \beta \in (\beta,\beta_0)$,
\begin{eqnarray*}
\|\widetilde x^{i-n-j} \widetilde {\mathbf U}_n^{(m-j)}\|_{0,\beta}
& \leq & 
\sup_{\widetilde x > 0} \widetilde x^{i-n-j} e^{-(\tilde \beta - \beta)\widetilde x } 
\|\widetilde {\mathbf U}_n^{(m-j)}\|_{0,\tilde \beta} \\
& \leq & 
C_U \left(\frac{i-n-j}{\tilde\beta - \beta}\right)^{i-n-j} 
(\beta_0 - \tilde \beta)^{-(2n+1+m)}(n+m-j)^{n} K_1^n K_2^{m-j} \nu^{-(m-j)}.
\end{eqnarray*}
For $j < i-n$, we select
$\tilde \beta = \beta + \kappa (\beta_0 - \beta)$ with 
$\kappa = \frac{i-n-j}{i+n-j+m+1}$ and get 
\begin{eqnarray}
\label{eq:foo-10}
\lefteqn{
\|\widetilde x^{i-n-j} \widetilde {\mathbf U}_n^{(m-j)}\|_{0,\beta}
 \leq  } \\
&& 
\nonumber 
C_U (\beta_0 - \beta)^{-(i+n-j+1+m)} \frac{(i+n-j+1+m)^{i+n-j+1+m}}{(2n+1+m)^{2n+1+m}} 
(n+m-j)^{n} K_1^n K_2^{m-j} \nu^{-(m-j)}; 
\end{eqnarray}
the induction hypothesis shows that the estimate (\ref{eq:foo-10}) is also true
for $j = i-n$. 
Therefore, by estimating $n+m-j\leq n+m$, we get 
\begin{eqnarray*}
\|{\mathbf F}_n\|_{0,\beta} &\leq& 
C_U (\beta_0 - \beta)^{-(i+n+1+m)} K_1^n K_2^{m} \nu^{-m} \frac{(n+m)^n}{(2n+1+m)^{2n+1+m}}
\\
&& \qquad \mbox{}\times 
\sum_{j=0}^{\min\{i-n,m\}} \binom{m}{j}\binom{i-n}{j} j! 
(\beta_0 - \beta)^j K_2^{-j} \nu^j
{(i+n-j+1+m)^{i+n-j+1+m}}.
\end{eqnarray*}
With the estimates $\binom{m}{j} j! \leq m^j$ we bound 
\begin{eqnarray*}
\lefteqn{
\sum_{j=0}^{\min\{i-n,m\}} \binom{m}{j}\binom{i-n}{j} j! 
(\beta_0 - \beta)^j K_2^{-j} \nu^j
{(i+n-j+1+m)^{i+n-j+1+m}} }\\
&& \leq 
\sum_{j=0}^{i-n} \binom{i-n}{j} 
(\beta_0 - \beta)^j K_2^{-j} \nu^j (i+n+1+m)^{i+n+1+m}\\
 && = \left( 1 + (\beta_0 - \beta) \nu K_2^{-1}\right)^{i-n} (i+n+1+m)^{i+n+1+m}, 
\end{eqnarray*}
where we have recognized a binomial sum in the last equality. 
Hence, 
\begin{eqnarray*}
\sum_{n=0}^{i-1} \|{\mathbf A}_{i-n}\|_2 \|{\mathbf F}_n\|_{0,\beta} 
&\leq& C_U C_A \gamma_A (\beta_0 - \beta)^{-(2i+1+m)} K_1^{i-1} K_2^{m} \nu^{-m} \\
&& \mbox{} \times 
\sum_{n=0}^{i-1} \gamma_A^{i-1-n} K_1^{n-i+1}\left(1 + (\beta_0 - \beta)\nu K_2^{-1}\right)^{i-n}
\frac{(n+m)^{n}(i+n+1+m)^{i+n+1+m}}{(2n+1+m)^{2n+1+m}}\\
&& \leq C_U C_0 C_A \gamma_A (\beta_0 - \beta)^{-(2i+1+m)} K_1^{i-1} K_2^m \nu^{-m} (i+m)^i, 
\end{eqnarray*}
where we appealed to Lemma~\ref{lemma:foo} and implicitly used
that $K_1$ and $K_2$ are such that $\gamma_A K_1^{-1} (1 + (\beta_0-\beta)\nu K_2^{-1})$ 
is sufficiently small. 
Using $\|{\mathbf E}^{-1} \|_2\leq \nu^{-2}$, we therefore get from 
(\ref{eq:thm:higher-derivatives-in-exponentially-weighted-spaces-10})
\begin{eqnarray*}
\|{\mathbf U}_i^{(m+2)} \|_{0,\beta} 
& \leq & \nu^{-2}  C_U (\beta_0 - \beta)^{-(2i+1+m)} (i+m)^i K_1^i K_2^m \nu^{-m} 
\left[ \|{\mathbf B}\|_2 + C_0 C_A \gamma_A K_1^{-1} 
\right]\\
& \leq & C_U (\beta_0 - \beta)^{-(2i+1+m)} (i+m)^i K_1^i K_2^{m+2} \nu^{-(m+2)} 
\left[ K_2^{-2} \|{\mathbf B}\|_2 + K_2^{-1} C_0 C_A \gamma_A K_1^{-1} 
\right]; 
\end{eqnarray*}
the expression in square brackets can be bounded by $1$ if $K_1$ and $K_2$ are sufficiently large. 

\section{Proofs for Section~\ref{case_4}}
\label{appendix:case_4}
\subsection{Proof of Lemma~\ref{lem10}}
\label{appendix:case_4_lem10}

%
The proof is by induction on $i$. \ For $i=0$, equation (\ref{39}) gives all the
assertions (trivially). So, assume (\ref{39b})--(\ref{39d}) hold and
establish them for $i+1$.

We first consider (\ref{39b}) and assuming
\begin{equation*}
\overrightarrow{0}=\left[ 
\begin{array}{c}
u_{ij} \\ 
v_{ij}%
\end{array}%
\right] ={\mathbf A}^{-1}\left[ 
\begin{array}{c}
u_{i-2,j-2}^{\prime \prime } \\ 
v_{i-2,j}^{\prime \prime }%
\end{array}%
\right] \;\forall \;j>i,
\end{equation*}
we want to show that 
\begin{equation*}
\overrightarrow{0}=\left[ 
\begin{array}{c}
u_{i+1,j} \\ 
v_{i+1,j}%
\end{array}%
\right] ={\mathbf A}^{-1}\left[ 
\begin{array}{c}
u_{i-1,j-2}^{\prime \prime } \\ 
v_{i-1,j}^{\prime \prime }%
\end{array}%
\right] \;\forall \;j>i+1.
\end{equation*}
By the induction hypothesis and $j > i+1 \ge i-1$, we have
$v_{i-1,j} = 0$. Also, from $j> i+1$ we get $j-2 =j-1-1> i-1$, so that 
the induction hypothesis implies $u_{i-1,j-2} = 0$. Therefore, the right-hand
side of (\ref{39}) vanishes and thus the induction step for (\ref{39b}) 
is accomplished.  


For (\ref{39c}) let ${\mathcal M}({\cal I}) := 
\{(i,j) \in {\mathbb Z} \times {\mathbb Z} \colon i \leq {\cal I}, \mbox{ $i$ odd or $j$ odd}\}$. 
We proceed by induction
on ${\cal I}$ by assuming (\ref{39c}) to be true up to ${\cal I}$. For $i \leq {\cal I}$ 
with $(i,j) \in {\mathcal M}({\cal I})$, we have that $(i+1,j) \in {\mathcal M}({\cal I}+1)$ 
implies $(i-1,j) \in {\mathcal M}({\cal I})$ (either $i+1$ is odd and then $i-1$ is odd
or $j$ is odd) and additionally 
$(i-1,j-2) \in {\mathcal M}({\cal I})$. Hence, for $(i+1,j) \in {\mathcal M}({\cal I}+1)$, 
by the induction hypothesis, the right-hand side of (\ref{39}) vanishes, which proves
the induction step. 

We finally consider (\ref{39d}) and we want to show
\begin{equation*}
\left\vert u_{i+1,j}(z)\right\vert +\left\vert v_{i+1,j}(z)\right\vert \leq
C_S\delta ^{-i-1}K^{i+1}(i+1)^{i+1}\;\forall \;z\in G_{\delta }.
\end{equation*}
We set 
$$
C_A:= \sup_{x \in I} \|{\mathbf A}^{-1}(x)\|_{\ell^1}, 
$$
where $\|(x_1,x_2)\|_{\ell^1}:=|x_1| + |x_2|$ denotes the usual $\ell^1$-norm
and ensure that $K$ satisfies $2 C_A/K^2 \leq 1$. 

Let $\kappa \in (0,1)$. \ By (\ref{39}), the induction hypothesis with 
$G_{(1-\kappa )\delta }\subset G_{\delta }$ and Cauchy's Integral Theorem, we
have
\begin{eqnarray*}
\left\vert u_{i+1,j}(z)\right\vert +\left\vert v_{i+1,j}(z)\right\vert &\leq
&C_A \left( \left\vert u_{i-1,j-2}^{\prime \prime }(z)\right\vert +\left\vert
v_{i-1,j}^{\prime \prime }(z)\right\vert \right) \\
&\leq &C_S C_A\frac{2}{(\kappa \delta )^{2}}\left( (1-\kappa )\delta \right)
^{-i+1}K^{i-1}(i-1)^{(i-1)} \\
&\leq &C_S \delta ^{-i-1}K^{i+1}(i+1)^{(i+1)}\left[ \frac{1}{K^{2}}\frac{1}{%
(i+1)^{2}}\frac{2 C_A}{\kappa ^{2}(1-\kappa )^{i-1}}
\left(\frac{i-1}{i+1}
\right)^{i-1}
\right] .
\end{eqnarray*}
The choice $\kappa =\frac{2}{i+1}$ gives 
\begin{eqnarray*}
\left\vert u_{i+1,j}(z)\right\vert +\left\vert v_{i+1,j}(z)\right\vert &\leq
&C_S \delta ^{-(i+1)}K^{i+1}(i+1)^{(i+1)}\left[ \frac{C_A}{2 K^{2}}
\right] 
\leq 
C_S \delta ^{-(i+1)}K^{i+1}(i+1)^{(i+1)}
\end{eqnarray*}
by the choice of $K$. 
\subsection{Proof 
of Lemma~\ref{lemma:scalar-bvp-constant-coefficients}}
\label{appendix:case_4_lemma:scalar-bvp-constant-coefficients}

We provide some details for the case $\operatorname*{Re} z < 0$. 
For $z \in (0,\infty)$, the use of a Green's function gives the following
representation of the solution $u(z)$:
\begin{eqnarray*}
u(z) &= &\frac{1}{2 a^2}\Bigl(
       e^{-a z }\int_0^{a z} e^y f(y/a)\, dy + 
       e^{a z} \int_{a z}^\infty e^{-y} f(y/a) \, dy \\
&& \quad \mbox{} - 
       e^{-a z} \int_0^\infty e^{-y} f(y/a) \, dy
          \Bigr) 
 + g e^{- a z}.
\end{eqnarray*}
Analytic continuation then removes the restriction to $(0,\infty)$.
In order to get the desired bound, we estimate each of these four terms
separately. We restrict here our attention to the case of $\operatorname*{Re} z  < 0$. 

For the first integral, we use as the path of integration the
straight line connecting $0$ and $a z$ to get
\begin{eqnarray*}
\left|e^{-a z}\int_0^{a z} e^y f(y/a)\, dy \right| &\leq &
e^{-\operatorname*{Re}( a z)}\int_0^{1} C_f (q + t|z|)^j |a z|\,
e^{-\operatorname*{Re}(t \oa z)} e^{\operatorname*{Re}( t a z)} \, dt \\ 
&\leq& 
e^{-\operatorname*{Re}( \oa z)}\int_0^{1} C_f (q + t|z|)^j |a z|\,
e^{\operatorname*{Re}((1-t) z (\oa -a))} \, dt \\ 
&\leq &
 C_f e^{-\operatorname*{Re}( \oa z)}
\frac{a}{j+1}\left\{(q + |z|)^{j+1} - 
                         q^{j+1}\right\} .
\end{eqnarray*}
For the third integral, we calculate with 
\cite[eq.\ {8.353.5}]{gradshtein80} 
and the incomplete Gamma-function $\Gamma(\cdot,\cdot)$,
\begin{eqnarray*}
\left| \int_0^\infty e^{-y} f(y/a)\,dy \right| &\leq& C_f 
\int_0^\infty e^{-y - \ua/a y} (q + y/a)^j\,dy \\
&=& C_f a^{-j} (1+\ua/a)^{-(j+1)} e^{a q (1+\ua/a)} \Gamma(j+1,aq(1+\ua/a))\\
&=& C_f a (a+\ua)^{-(j+1)} e^{q (a+\ua)} \Gamma(j+1,q(a+\ua)).
\end{eqnarray*}
In view of the assumption $ (a+\ua) q \ge 2 j + 1 \ge j$, we may employ
the estimate
$$
|\Gamma(\alpha,\xi)| \leq \frac{\left|e^{-\xi} \xi^\alpha\right|}
{|\xi|-\alpha_0}, 
\qquad \alpha_0 = \max\,\{\alpha-1,0\}, 
\quad \operatorname*{Re}( \xi) \ge 0, \quad |\xi| >\alpha_0,
$$
(see, e.g., \cite[Chap.\ 4, Sec.\ 10]{olver74}) to arrive at
\begin{equation}
\label{eq:lemma:scalar-bvp-constant-coefficients-10}
\left| \int_0^\infty e^{-y} f(y/a)\right| 
\leq C_f a q^{j+1} \frac{1}{q(a+\ua)-j}.
\end{equation}
Hence, the third integral can be estimated by 
$$
\left| e^{-az} \int_0^\infty e^{-y} f(y/a)\right| 
\leq C_f a q^{j+1} \frac{1}{q(a+\ua)-j} e^{-\oa \operatorname*{Re}(z)}.
$$
We now turn to the second integral in the representation formula for $u$. We split
the integral as 
\begin{eqnarray*}
\left| e^{a z} \int_{a z}^\infty e^{-y} f(y/a)\, dy\right|
& \leq &
\left| e^{az} \int_{0}^{az} e^{-y} f(y/a)\,dy\right| + 
\left| e^{az} \int_{0}^{\infty} e^{-y} f(y/a)\,dy\right| .  
\end{eqnarray*}
We recognize that the second integral can be estimated using 
(\ref{eq:lemma:scalar-bvp-constant-coefficients-10}). The first integral 
is treated as follows: 
\begin{eqnarray*}
\left| e^{az} \int_{0}^{az} e^{-y} f(y/a)\,dy\right| 
&\leq& C_f 
e^{\operatorname*{Re}(az)} \int_0^1 e^{-t \operatorname*{Re}(az)} |az| 
e^{-t\operatorname*{Re}(z \oa)} (q + t|z|)^j\,dt\\
&\leq& C_f e^{-\oa \operatorname*{Re}(z)} 
       \int_0^1 (q + t|z|)^{j} a|z| e^{(1-t)(a+\oa)\operatorname*{Re} z}\,dt\\
&\leq& C_f e^{-\oa \operatorname*{Re}(z)} \frac{a}{j+1} \left[ (q+|z|)^{j+1} - q^{j+1}\right].
\end{eqnarray*}
Combining the above estimates and recalling $q(a+\ua) - j \ge 2j+1-j \ge j+1$, we conclude 
\begin{eqnarray*}
&& \left| 
       e^{-a z }\int_0^{a z} e^y f(y/a )\, dy + 
       e^{a z} \int_{a z}^\infty e^{-y} f(y/a) \, dy - 
       e^{-a z} \int_0^\infty e^{-y} f(y/a) \, dy
\right| \leq  \\
&& \qquad 
2 C_f e^{-\operatorname*{Re}( \oa z)} (q +|z|)^{j+1} \frac{a}{j+1} .
\end{eqnarray*}
Combining  this estimate with the obvious one for the fourth term,
we arrive at the desired bound.
\subsection{Proof of Theorem~\ref{thm11}}
\label{appendix:case_4_thm11}
We begin
by setting
$$
a^\prime:= \frac{\operatorname*{det}{\mathbf A}(0)}{a_{11}(0)} 
= \frac{a_{11}(0) a_{22}(0) - a_{12}(0) a_{21}(0)}{a_{11}(0)},
$$
and 
choose the constants $C_{\widetilde{u}},C_{\widetilde{v}},C_{\widehat{u}%
},C_{\widehat{v}},C_{i},K_{i},\overline{K}_{i}$, $i=1,...,4$ to satisfy the
following:
\begin{subequations}
\label{eq:simplifying-assumptions-1}
\begin{align}
K_1 = K_2  = K_3 = K_4 \ge 1,\\ 
\overline{K}_1 = \overline{K}_2  = \overline{K}_3 = \overline{K}_4 \ge 1,\\
C_{\widetilde u} = C_{\widetilde v} = C_{\widehat u} = C_{\widehat v} \ge 1.
\end{align}
\end{subequations}
Furthermore, the following requirements have to be satisfied: 
with the constants $C_S$, $K$ of Lemma~\ref{lem10} (we assume for notational
simplicity that $\delta = 1$ is admissible in Lemma~\ref{lem10}) 
and $\gamma _{a}$ the constant of analyticity of the data (see (\ref{3})):

\begin{eqnarray}
\label{eq:b-1}
C_2 &\ge & C_S, \\
\label{eq:b-2}
C_3 &\ge & C_S, \\
\label{eq:b-3}
C_1 &\ge & \frac{|a_{12}(0)|}{a_{11}(0)} C_2, \\
\label{eq:b-4}
C_1 &\ge & \frac{|a_{21}(0)|}{a_{22}(0)} C_2, \\
\label{eq:b-5}
C_4 &\ge & \frac{|a_{21}(0)|}{\ua^2} C_3, \\
\label{eq:b-6}
{K}_1 = {K}_2 & > & \max\{1,\gamma_a\}, \\
\label{eq:b-9}
1 &\ge & 
\left[C_a 
\left(1 + \frac{|a_{21}(0)|}{a_{11}(0)}\right) \frac{1}{a^\prime} \left(1+\frac{C_1}{C_2}\right)
\frac{1}{1-\gamma_a/{K}_2}\frac{\gamma_a}{{K}_2} + 
\frac{C_S}{C_2} \left(\frac{K}{{K}_2}\right)^{i+1} 
\right],\\
\label{eq:b-10}
1 &\ge & \left[
\frac{C_2}{C_1} \frac{|a_{12}(0)|}{a_{11}(0)}
+ \frac{C_a}{a_{11}(0)} \frac{\gamma_a}{{K}_1} \frac{1}{1 - \gamma_a/{K}_1}
\left(1 + \frac{C_2}{C_1}\right)
\right],\\
\label{eq:b-11}
1 &\ge & 
\left[
\frac{C_S}{C_3} \left(\frac{1}{\overline{K}_3}\right)^{j+1} 
+ \frac{C_1}{C_3}\left(\frac{\overline{K}_1}{\overline{K}_3}\right)^{j+1} 
+ \overline{K}_4^{-2} \frac{C_4}{C_3} |a_{12}(0)|
\right], \\
\label{eq:b-12-13}
C_{\widehat u} \ua  & > & 4, \\
\label{eq:b-14}
1 &\ge & 
\left[
\frac{C_3}{C_4} \frac{|a_{21}(0)|}{\ua^2} \frac{1}{1-2/(C_{\widehat u}\ua)}
\right],
\\
\label{eq:b-15}
1 &\ge & 
\left[\frac{C_S}{C_3} \left(\frac{K}{{K}_3}\right)^{i+1}\right],
\\
\label{eq:b-16}
1 &\ge & 
\frac{1}{\ua^2}
\left(\frac{1}{1-2/(C_{\widehat v} \ua)}\right)^2
C_a \left[\frac{C_3}{C_4} + \overline{K}_4^{-2}\right] 
\frac{1}{1-\gamma_a/(K_4 \overline{K}_4)},  \\
\label{eq:b-17}
1 &\ge & \frac{4}{\ua^2} |a_{21}(0)|\frac{C_3}{C_4} + \frac{4}{\ua^2} \overline{K}_4^{-2} a_{22}(0), \\
\label{eq:b-18}
1 &\ge &  \overline{K}_1^{-2} \frac{C_1}{C_2} \frac{2 e^{\oa} |a_{21}(0)|}{a^\prime a_{22}(0)}
+ \frac{C_S}{C_2} \left(\frac{1}{\overline{K}_1}\right)^{j+1} 
+ \overline{K}_4^{-2} \frac{C_4}{C_2} \left(\frac{\overline{K}_4}{\overline{K}_1}\right)^{j+1},\\
\label{eq:b-19}
1 &\ge & \left[
\frac{|a_{12}(0)|}{a_{11}(0)} \frac{C_2}{C_1} + \overline{K}_1^{-2} \frac{2 e^{\oa}}{a_{11}(0)} 
\right],\\
\label{eq:b-20}
1 &\ge & 
\Bigl[
\frac{|a_{21}(0)|}{a_{11}(0)} 2 e^{\oa} \frac{C_1}{C_2} \overline{K}_2^{-2} + 
C_a \left(1 + \frac{|a_{21}(0)|}{a_{11}(0)}\right)
\frac{\gamma_a}{K_1}\frac{1}{1-\gamma_a/K_1} 
\left( \frac{C_1}{C_2} + 1\right) \\
&& \qquad \mbox{}
+ \frac{C_S}{C_2} \left(\frac{K}{K_2}\right)^i \left(\frac{1}{\overline{K}_2}\right)^{j+1} + 
\overline{K}_4^{-2} \frac{C_4}{C_2} \left(\frac{K_4}{K_2}\right)^{i}
\left(\frac{\overline{K}_4}{\overline{K}_2}\right)^{j+1}
\Bigr],\\
\label{eq:b-22}
1 &\ge & \left[
\frac{|a_{12}(0)|}{a_{11}(0)} \frac{C_2}{C_1} + 
\overline{K}_1^{-2} \frac{2 e^{\oa}}{a_{11}(0)} + 
\frac{C_a}{a_{11}(0)} \frac{1}{1-\gamma_a/K_1} \frac{\gamma_a}{K_1} \left(1 + \frac{C_2}{C_1}\right)
\right],\\
\label{eq:b-23}
1 &\ge & \Bigl[
\frac{1}{a_{11}(0)}
\left\{C_a \frac{C_4}{C_3} \overline{K}_4^{-2} + 
\frac{\gamma_a}{K_3 \overline{K}_3} \frac{C_a}{1-\gamma_a/(K_3 \overline{K}_3)} 
\left( 1+ \frac{C_4}{C_3} \overline{K}_4^{-2}\right)
\right\} \\
\nonumber 
&&\qquad \mbox{}
+ \frac{C_S}{C_3} \left(\frac{K}{K_3}\right)^i \left(\frac{1}{\overline{K}_3}\right)^{j+1} 
+ \frac{C_1}{C_3} \left(\frac{K_1}{K_3}\right)^i \left(\frac{\overline{K}_1}{\overline{K}_3}\right)^{j+1} 
\Bigr]. 
\end{eqnarray}
Before we proceed with the proof, we make sure that
the requirements (\ref{eq:b-1})--(\ref{eq:b-23}) can be satisfied. 
We make the following simplifying assumptions in addition to 
(\ref{eq:simplifying-assumptions-1}): 
\begin{subequations}
\label{eq:simplifying-assumption-2}
\begin{align}
C_{\widetilde u} = 
C_{\widetilde v} = 
C_{\widehat u} = 
C_{\widehat v} > 8/\ua, \\
\frac{\gamma_a}{K_1}  =
\frac{\gamma_a}{K_2}  =
\frac{\gamma_a}{K_3}  =
\frac{\gamma_a}{K_4}  \leq \frac{1}{2}, \\
K \leq K_1 = K_2 = K_3  = K_4, \\
C_2 = Q, \qquad C_1  = Q^2, \qquad C_3 = Q^3, \qquad C_4 = Q^4.
\end{align}
\end{subequations}
Here, $Q > 0$ will be selected sufficiently large below. 
Then, the requirements (\ref{eq:b-1})--(\ref{eq:b-23}) are satisfied if: 
\begin{eqnarray}
\label{eq:f-1}
Q&\ge & C_S, \\
\label{eq:f-2}
Q^3  &\ge & C_S, \\
\label{eq:f-3}
Q^2  &\ge & \frac{|a_{12}(0)|}{a_{11}(0)} Q, \\
\label{eq:f-4}
Q^2  &\ge & \frac{|a_{21}(0)|}{a_{22}(0)} Q, \\
\label{eq:f-5}
Q^4 &\ge & \frac{|a_{21}(0)|}{\ua^2} Q^3, \\
\label{eq:f-9}
1 &\ge & 
\left[C_a 
\left(1 + \frac{|a_{21}(0)|}{a_{11}(0)}\right) \frac{2}{a^\prime} \left(1+\frac{Q^2}{Q}\right)
\frac{\gamma_a}{{K}_2} + 
\frac{C_S}{Q} 
\right],\\
\label{eq:f-10}
1 &\ge & \left[
\frac{Q}{Q^2} \frac{|a_{12}(0)|}{a_{11}(0)}
+ \frac{C_a}{a_{11}(0)} \frac{2 \gamma_a}{{K}_1} 
\left(1 + \frac{Q}{Q^2}\right)
\right],\\
\label{eq:f-11}
1 &\ge & 
\left[
\frac{C_S}{Q^3} 
+ \frac{Q^2}{Q^3}
+ \overline{K}_4^{-2} \frac{Q^4}{Q^3} |a_{12}(0)|
\right], \\
\label{eq:f-14}
1 &\ge & 
\left[
\frac{Q^3}{Q^4} \frac{2 |a_{21}(0)|}{\ua^2} 
\right],
\\
\label{eq:f-15}
1 &\ge & 
\frac{C_S}{Q^3},
\\
\label{eq:f-16}
1 &\ge & 
\frac{4}{\ua^2}
2 C_a \left[\frac{Q^3}{Q^4} + \overline{K}_4^{-2}\right], \\
\label{eq:f-17}
1 &\ge & \frac{4}{\ua^2} |a_{21}(0)|\frac{Q^3}{Q^4} + \frac{4}{\ua^2} \overline{K}_4^{-2} a_{22}(0), \\
\label{eq:f-18}
1 &\ge &  \overline{K}_1^{-2} \frac{Q^2}{Q} \frac{2 e^{\oa} |a_{21}(0)|}{a^\prime a_{22}(0)}
+ \frac{C_S}{Q} 
+ \overline{K}_4^{-2} \frac{Q^4}{Q}, \\
\label{eq:f-19}
1 &\ge & \left[
\frac{|a_{12}(0)|}{a_{11}(0)} \frac{Q}{Q^2} + \overline{K}_1^{-2} \frac{2 e^{\oa}}{a_{11}(0)} 
\right],\\
\label{eq:f-20}
1 &\ge & 
\Bigl[
\frac{|a_{21}(0)|}{a_{11}(0)} 2 e^{\oa} \frac{Q^2}{Q} \overline{K}_2^{-2} + 
C_a \left(1 + \frac{|a_{21}(0)|}{a_{11}(0)}\right)
\frac{2\gamma_a}{K_1}
\left( \frac{Q^2}{Q} + 1\right) 
+ \frac{C_S}{Q}  + 
\overline{K}_4^{-2} \frac{Q^4}{Q} 
\Bigr],\\
\label{eq:f-22}
1 &\ge & \left[
\frac{|a_{12}(0)|}{a_{11}(0)} \frac{Q}{Q^2} + 
\overline{K}_1^{-2} \frac{2 e^{\oa}}{a_{11}(0)} + 
\frac{2 C_a}{a_{11}(0)} \frac{\gamma_a}{K_1} \left(1 + \frac{Q}{Q^2}\right)
\right],\\
\label{eq:f-23}
1 &\ge & \Bigl[
\frac{1}{a_{11}(0)}
\left\{C_a \frac{Q^4}{Q^3} \overline{K}_4^{-2} + 
\frac{2 \gamma_a}{K_3 \overline{K}_3} 
\left( 1+ \frac{Q^4}{Q^3} \overline{K}_4^{-2}\right)
\right\} 
+ \frac{C_S}{Q^3} 
+ \frac{C_1}{Q^3} 
\Bigr] . 
\end{eqnarray}
It is now easy to see that by first selecting $Q$ sufficiently large and then
choosing the parameters $K_i$ and $\overline{K}_i$ sufficiently large, ensures 
the above requirements.

We now turn to the induction argument.

\bigskip \textbf{(Base) Case }$i=j=0$\textbf{:}

The functions $\widetilde{u}_{0,0}^{L},\widetilde{v}_{0,0}^{L},\widehat{u}%
_{0,0}^{L},\widehat{v}_{0,2}^{L}$ ($\widehat{v}_{0,0}^{L}=0$) satisfy the
following:
\begin{equation}
\left\{ 
\begin{array}{c}
-\left( \widetilde{v}_{0,0}^{L}\right) ^{\prime \prime }+a^\prime%
\widetilde{v}_{0,0}^{L}=0 \\ 
\widetilde{v}_{0,0}^{L}(0)=-v_{0,0}(0)\;,\;\widetilde{v}_{0,0}^{L}(\widetilde x)%
\rightarrow 0\text{ as }\widetilde x\rightarrow \infty%
\end{array}%
\right. ,  \label{61}
\end{equation}
\begin{equation}
\widetilde{u}_{0,0}^{L}=-\frac{a_{12}(0)}{a_{11}(0)}\widetilde{v}_{0,0}^{L},
\label{62}
\end{equation}
\begin{equation}
\left\{ 
\begin{array}{c}
-\left( \widehat{u}_{0,0}^{L}\right) ^{\prime \prime }+a_{11}(0)\widehat{u}%
_{0,0}^{L}=0 \\ 
\widehat{u}_{0,0}^{L}(0)=-u_{0,0}(0)\;,\;\widehat{u}_{0,0}^{L}(\widehat x)\rightarrow 0%
\text{ as }\widehat x\rightarrow \infty%
\end{array}%
\right. ,  \label{63}
\end{equation}
\begin{equation}
\widehat{v}_{0,2}^L(z)=\int_z^\infty \int_t^\infty 
a_{21}(0)\widehat{u}_{0,0}^{L}(\tau )d\tau dt.
\label{64}
\end{equation}
Solution formulas for $\widetilde v^L_{0,0}$ and $\widehat u^L_{0,0}$ 
and Lemma~\ref{lem10} (recall that we assume that $\delta = 1$ is 
admissible in Lemma~\ref{lem10}) give 
us the desired result for $\widetilde{v}_{0,0}^{L}$ and $\widehat{u}_{0,0}^{L}$ 
in view of the requirements (\ref{eq:b-1}), (\ref{eq:b-2}),
while (\ref{62}) gives it for $\widetilde{u}_{0,0}^{L}$, in view of requirement
(\ref{eq:b-4}). 
For 
$\widehat{v}_{0,2}^{L}(z)$ we have from (\ref{64}), the just proven result for 
$\widehat u^L_{0,0}$, and Lemma~\ref{lemma:two-anti-derivatives} (with $j = 0$)
\begin{equation*}
\left\vert \widehat{v}^L_{0,2}(z)\right\vert \leq |a_{21}(0)| C_3 \frac{1}{\ua^2} \decay{z}
\leq C_4 \decay{z},
\end{equation*}
by the choice of $C_4$ in (\ref{eq:b-5}).

\textbf{(Base) Case }$j=0$, $i>0$\textbf{:}

The functions 
$\widetilde{u}_{i,0}^{L},\widetilde{v}_{i,0}^{L},\widehat{u}_{i,0}^{L},\widehat{v}_{i,2}^{L}$ 
satisfy equations (\ref{40})--(\ref{44}), respectively, for up to $i$. We proceed by induction on $i$. 
First, $\widetilde{v}_{i+1,0}^{L}$ satisfies (\ref{eq:41-42}) with $i$ replaced by $i+1$. The right
hand side of that boundary value problem satisfies, in view of the induction hypothesis and 
the choices $K_1 = K_2$, $C_{\widetilde u} = C_{\widetilde v} \ge 1$,
\begin{eqnarray*}
\lefteqn{
\left\vert RHS_{(\ref{41})}\right\vert  
\leq C_a \left(1 + \frac{|a_{21}(0)|}{a_{11}(0)}\right) 
\underset{k=1}{\overset{i+1}{%
\sum }}\gamma _{a}^{k}\left\vert z\right\vert ^{k}\left\{ \left\vert 
\widetilde{u}_{i+1-k,0}^{L}(z)\right\vert +\left\vert \widetilde{v}%
_{i+1-k,0}^{L}(z)\right\vert \right\}  
}\\
&&\leq C_a \left(1 + \frac{|a_{21}(0)|}{a_{11}(0)}\right) 
\underset{k=1}{\overset{i+1}{%
\sum }}\gamma _{a}^{k} |z|^k \frac{1}{(i+1-k)!}(C_{\widetilde u}(i+1-k) + |z|)^{2(i+1-k)} \times \\
&\mbox{}&\times \left( C_1 {K}_1^{i+1-k}  + C_2 {K}_2^{i+1-k}\right)\decay{z} \\
&&\leq C_a \left(1 + \frac{|a_{21}(0)|}{a_{11}(0)}\right) (C_1 + C_2) 
\frac{1}{1-\gamma_a/{K}_2}\frac{\gamma_a}{{K}_2}
{K}_2^{i+1} 
\frac{1}{i!} (C_{\widetilde u} i +|z|)^{2i+1} 
\decay{z},
\end{eqnarray*}
where we used the fact that ${K}_1 =  {K}_2 > \gamma_a$ by (\ref{eq:b-6}). 
Lemma~\ref{lemma:scalar-bvp-constant-coefficients} 
and Lemma~\ref{lem10} yield, for the solution $\widetilde v_{i,0}^L$
of the boundary value problem (\ref{eq:41-42}),
\begin{eqnarray}
\nonumber 
\left\vert \widetilde{v}_{i+1,0}^{L}(z)\right\vert  
&\leq & \Bigl(C_a \left(1 + \frac{|a_{21}(0)|}{a_{11}(0)}\right) \frac{1}{a^\prime} (C_1 + C_2) 
\frac{1}{1-\gamma_a/{K}_2}\frac{\gamma_a}{{K}_2}
{K}_2^{i+1} 
\frac{1}{(i+1)!} (C_{\widetilde u} i +|z|)^{2i+2} \\
\nonumber 
&& \mbox{}
+ C_S K^{i+1} (i+1)^{i+1}\Bigr)
\decay{z}\\
\nonumber 
&\leq & C_2 {K}_2^{i+1} \frac{1}{(i+1)!}(C_{\widetilde u}(i+1) + |z|)^{2(i+1)} \decay{z} \\
&& \mbox{}\times 
\left[C_a 
\left(1 + \frac{|a_{21}(0)|}{a_{11}(0)}\right) \frac{1}{a^\prime} \left(1+\frac{C_1}{C_2}\right)
\frac{1}{1-\gamma_a/{K}_2}\frac{\gamma_a}{{K}_2} + 
\frac{C_S}{C_2} \left(\frac{K}{{K}_1}\right)^{i+1} 
\right]. 
\label{70}
\end{eqnarray}
The expression in square brackets is bounded by $1$ by our requirement (\ref{eq:b-9}).
Next, we consider $\widetilde{u}_{i+1,0}^{L}$ which satisfies (\ref{40})
with $i$ replaced by $i+1$. \ We have by the induction hypothesis and (\ref{70}),
\begin{eqnarray*}
\left\vert \widetilde{u}_{i+1,0}^{L}(z)\right\vert &\leq & 
\Bigl\{ \frac{|a_{12}(0)|}{a_{11}(0)} C_{2}%
{K}_{2}^{i+1}\frac{1}{(i+1)!}\left( C_{%
\widetilde{v}}(i+1)+\left\vert z\right\vert \right) ^{2(i+1)}+ \\
&&+\underset{k=1}{\overset{i+1}{\sum }}\frac{C_{a}}{a_{11}(0)}\gamma _{a}^{k}\left\vert
z\right\vert ^{k}\frac{1}{(i+1-k)!}\left[ C_{1}%
{K}_{1}^{i+1-k}\left( C_{\widetilde{u}}(i+1)+\left\vert
z\right\vert \right) ^{2(i+1-k)}
\right. \\
&&+\left. C_{2}{K}_{2}^{i+1-k}\left( C_{\widetilde{v}%
}(i+1)+\left\vert z\right\vert \right) ^{2(i+1-k)}\right] 
\Bigr\}\decay{z}\\
\\
&\leq &C_{1}{K}_{1}^{i+1}\frac{1}{(i+1)!}\left( C_{\widetilde{u}}(i+1)+\left\vert
z\right\vert \right) ^{2(i+1)} \times \\
&&\times \left[
\frac{C_2}{C_1} \frac{|a_{12}(0)|}{a_{11}(0)}
+ \frac{C_a}{a_{11}(0)} \frac{\gamma_a}{{K}_1} \frac{1}{1 - \gamma_a/{K}_1}
\left(1 + \frac{C_2}{C_1}\right)
\right]
\decay{z},
\end{eqnarray*}
since $K_{1} = K_{2} >\gamma_{a}$ by (\ref{eq:b-6}). 
Again, the expression in square brackets is bounded by $1$ by 
our requirement (\ref{eq:b-10}). 

For $\widehat{u}_{i+1,0}^{L},$ we have by 
Lemma~\ref{lemma:scalar-bvp-constant-coefficients}, (\ref{43}), and
Lemma \ref{lem10},
\begin{eqnarray}
\left\vert \widehat{u}_{i+1,0}^{L}(z)\right\vert &\leq &
e^{-a_{11}(0)\operatorname*{Re}(z)}\left\vert u_{i+1,0}(0)\right\vert 
\leq C_S K^{i+1}(i+1)^{(i+1)} \decay{z}
\nonumber 
\\
& \leq & 
C_3 {K}_3^{i+1} (i+1)^{(i+1)} 
\left[\frac{C_S}{C_3} \left(\frac{K}{{K}_3}\right)^{i+1}\right]
\decay{z}.
\label{71}
\end{eqnarray}
In view of our requirement (\ref{eq:b-15}), the expression in
square brackets is bounded by $1$. 

Finally, for $\widehat{v}_{i+1,2}^{L}$ we have from (\ref{44}), (\ref{71})
and Lemma~\ref{lemma:two-anti-derivatives}, in view of 
${K}_3 = {K}_4$ and $C_{\widehat u} =C_{\widehat v}$,
\begin{eqnarray*}
\left\vert \widehat{v}_{i+1,2}^{L}(z)\right\vert  &\leq &\left\vert \overset{%
\infty }{\underset{z}{\int }}\overset{\infty }{\underset{t}{\int }}a_{21}(0)%
\widehat{u}_{i+1,0}^{L}(\tau )d\tau dt\right\vert  \\
&\leq& \frac{|a_{21}(0)|}{\ua^2} \left(\frac{1}{1-2/(C_{\widetilde u} \ua)}\right)^2
C_3 {K}_3^{i+1} (C_{\widehat u}(i+1) + |z|)^{2(i+1)}\decay{z}\\
&\leq& C_4 {K}_4^{i+1} 
(C_{\widehat v}(i+1) + |z|)^{2(i+1)}\decay{z}
\left[
\frac{C_3}{C_4} \frac{|a_{21}(0)|}{\ua^2} \frac{1}{1-2/(C_{\widehat u}\ua)}
\right].
\end{eqnarray*}
The expression in square brackets is bounded by $1$ by our 
requirements (\ref{eq:b-14}), (\ref{eq:b-12-13}). 

\textbf{(Base) Case }$i=0$, $j>0$\textbf{:}

The functions $\widetilde{u}_{0,j+1}^{L},\widetilde{v}_{0,j+1}^{L},\widehat{u}%
_{0,j+1}^{L},\widehat{v}_{0,j+3}^{L}$ satisfy the following:
\begin{equation}
\left\{ 
\begin{array}{c}
-\left( \widetilde{v}_{0,j+1}^{L}\right) ^{\prime \prime }+a^\prime%
\widetilde{v}_{0,j+1}^{L}=-\frac{a_{21}(0)}{a_{11}(0)}\left( \widetilde{u}%
_{0,j-1}^{L}\right) ^{\prime \prime } \\ 
\widetilde{v}_{0,j+1}^{L}(0)=-\left( v_{0,j+1}(0)+\widehat{v}_{0,j+1}^{L}(0)%
\right) \;,\;\widetilde{v}_{0,j+1}^{L}(\widetilde x)\rightarrow 0\text{ as }\widetilde x\rightarrow
\infty 
\end{array}%
\right. ,  \label{65}
\end{equation}
\begin{equation}
\widetilde{u}_{0,j+1}^{L}=-\frac{a_{12}(0)}{a_{11}(0)}\widetilde{v}_{0,j+1}^{L}+%
\frac{\left( \widetilde{u}_{0,j-1}^{L}\right) ^{\prime \prime }}{a_{11}(0)},
\label{66}
\end{equation}
\begin{equation}
\left\{ 
\begin{array}{c}
-\left( \widehat{u}_{0,j+1}^{L}\right) ^{\prime \prime }+a_{11}(0)\widehat{u}%
_{0,j+1}^{L}=-a_{12}(0) \widehat v_{0,j+1}^{L} \\ 
\widehat{u}_{0,j+1}^{L}(0)=-\left( u_{0,j+1}(0)+\widetilde{u}_{0,j+1}^{L}(0)%
\right) \;,\;\widehat{u}_{0,j}^{L}\rightarrow 0\text{ as }x\rightarrow
\infty 
\end{array}%
\right. ,  \label{67}
\end{equation}
\begin{equation}
\widehat{v}_{0,j+3}^L(z)=\overset{\infty }{\underset{z}{\int }}\overset{\infty 
}{\underset{t}{\int }}\left[ a_{21}(0)\widehat{u}_{0,j+1}^{L}(\tau )+a_{22}(0)%
\widehat{v}_{0,j+1}^{L}(\tau )\right] d\tau dt.  \label{68}
\end{equation}
To establish the desired claims, we proceed by induction on $j$ noting
that the case $j = 0$ and $i = 0$ has been proved already. Assuming that the 
bounds are valid for $i=0$ and up to $j$, we show them for $j+1$. 

We start with $\widehat{v}^L_{0,j+3}$. 
By the induction hypothesis, Lemma~\ref{lemma:two-anti-derivatives}, 
and the assumption $C_{\widehat v} \ua \ge 4$ (see requirement (\ref{eq:b-12-13}))
we have 
\begin{eqnarray*}
\left\vert \widehat{v}^{L}_{0,j+3}(z)\right\vert  &\leq & 
\frac{4}{\ua^2}
\Bigl(
|a_{21}(0)| C_3 \overline{K}_3^{j+1} (C_{\widehat u}(j+1) + |z|)^{2(j+1)}\frac{1}{(j+1)!}
 \\
&& \mbox{} + 
a_{22}(0) C_4 \overline{K}_4^{j-1} (C_{\widehat v}(j-1) + |z|)^{2(j-1)}\frac{1}{(j-1)!}
\Bigr) \decay{z}.
\end{eqnarray*}
In view of $C_{\widehat u} = C_{\widehat v}$ and $\overline{K}_3 = \overline{K}_4$ we get 
\begin{eqnarray*}
\left\vert \widehat{v}^{L}_{0,j+3}(z)\right\vert  &\leq & 
C_4 \overline{K}_4^{j+1} \frac{1}{(j+1)!} (C_{\widehat v}(j+1) + |z|)^{2(j+1)} \decay{z} 
\left[
\frac{4}{\ua^2} |a_{21}(0)| \frac{C_3}{C_4}  
+ 
\frac{4}{\ua^2} |a_{22}(0)| \overline{K}_4^{-2}
\right]; 
\end{eqnarray*}
by requirement (\ref{eq:b-17}), the expression in square brackets is bounded by $1$ as required. 

We next turn our attention to $\widetilde{v}_{0,j+1}^{L}$ which satisfies 
(\ref{65}) and right-hand side (RHS) satisfying
\begin{equation}
\label{eq:foo-200}
\left\vert RHS_{(\ref{65})}\right\vert \leq \frac{\left\vert
a_{21}(0)\right\vert }{\left\vert a_{11}(0)\right\vert }\left\vert \left( 
\widetilde{u}_{0,j-1}^{L}\right) ^{\prime \prime }\right\vert .
\end{equation}
We first study the case $j = 0$. Then $\widetilde u_{0,j-1}^L = 0$, and 
Lemma~\ref{lemma:scalar-bvp-constant-coefficients} yields, together with 
Lemma~\ref{lem10},
$$
|\widetilde v_{0,j+1}^L(z)| e^{-a^\prime \operatorname*{Re}(z)} C_S 
\leq C_2 \overline{K}_2^{j+1} \decay{z}  
\left[ \frac{C_S}{C_2} \overline{K}_2^{-(j+1)} \right].
$$
Since we assume $\overline{K}_2 \ge 1$, our requirement (\ref{eq:b-1}) implies
the desired bound. Returning to (\ref{eq:foo-200}) for the case $j \ge 1$, 
the induction hypothesis and Lemma~\ref{lemma:derivatives-of-entire-fcts}
produce
\begin{equation*}
\left\vert RHS_{(\ref{65})}(z)\right\vert \leq \frac{|a_{21}(0)|}{a_{11}(0)}
{2} e^{\oa}
C_{1}\frac{\overline{K}_{1}^{j-1}}{(j-1)!}\left( C_{\widetilde{u}%
}(j-1)+\left\vert z\right\vert + 1 \right) ^{2(j-1)}\decay{z}. 
\end{equation*}
Therefore, Lemma~\ref{lemma:scalar-bvp-constant-coefficients} 
together with Lemma~\ref{lem10} yields 
\[
\left\vert \widetilde{v}_{0,j+1}^{L}(z)\right\vert  \leq \:\:
\]
\begin{eqnarray}
&\leq &\!\!\!
\left[ \frac{1}{a^\prime} \frac{1}{2j-1} 2 e^{\oa} \frac{|a_{21}(0)|}{a_{22}(0)}
C_{1}\frac{\overline{K}_{1}^{j-1}}{(j-1)!}\left( C_{\widetilde{u}%
}(j-1)+\left\vert z\right\vert + 1 \right) ^{2(j-1)+1} + 
\left\vert v_{0,j+1}(0)\right\vert +\left\vert \widehat{v}^L%
_{0,j+1}(0)\right\vert \right] \decay{z}  
\nonumber  \\
&\leq& 
C_2 \overline{K}_1^{j+1} \frac{1}{(j+1)!}\left(C_{\widetilde v}(j+1) + |z|\right)^{2(j+1)}
\nonumber 
\\
&&\mbox{} \times 
\left[ \overline{K}_1^{-2} \frac{C_1}{C_2} \frac{2 e^{\oa} |a_{21}(0)|}{a^\prime a_{22}(0)}
+ \frac{C_S}{C_2} \left(\frac{1}{\overline{K}_1}\right)^{j+1} 
+ \overline{K}_4^{-2} \frac{C_4}{C_2} \left(\frac{\overline{K}_4}{\overline{K}_1}\right)^{j+1}
\right]
\decay{z},
\label{69}
\end{eqnarray}
where, again, the expression square brackets is bounded by $1$ due to 
our requirement (\ref{eq:b-18}). 

Now, consider $\widetilde{u}_{0,j+1}^{L}$, which satisfies (\ref{66}).  
By (\ref{69}) and Lemma~\ref{lemma:derivatives-of-entire-fcts}, we have 
\begin{eqnarray*}
\left\vert \widetilde{u}_{0,j+1}^{L}(z)\right\vert  &\leq &
\Bigl(
\frac{|a_{12}(0)|}{a_{11}(0)} C_2 \overline{K}_2^{j+1} \frac{1}{(j+1)!}(C_{\widetilde v}(j+1) + |z|)^{2(j+1)} 
\\ 
&& \mbox{} + 
\frac{2 e^{\oa}}{a_{11}(0)} C_1 \overline{K}_1^{j-1}\frac{1}{(j-1)!}  (C_{\widetilde u}(j-1) + 1+ |z|)^{2(j-1)}
\Bigr)\decay{z} \\
&\leq& C_1 \overline{K}_1^{j+1} \frac{1}{(j+1)!} (C_{\widetilde u}(j+1) + |z|)^{2(j+1)}
\left[
\frac{|a_{12}(0)|}{a_{11}(0)} \frac{C_2}{C_1} + \overline{K}_1^{-2} \frac{2 e^{\oa}}{a_{11}(0)} 
\right]\decay{z}; 
\end{eqnarray*}
in view of requirement (\ref{eq:b-19}), the expression in 
square brackets is bounded by $1$ as required .

Finally, for $\widehat{u}_{0,j+1}^{L}$ which satisfies (\ref{67}) 
we have by Lemma~\ref{lemma:scalar-bvp-constant-coefficients}
(the case $j = 0$ needs special treatment in that the third term in the following estimate
is not present) 
\begin{eqnarray*}
\lefteqn{
\left\vert \widehat{u}_{0,j+1}^{L}(z)\right\vert  \leq } \\
&&
\left( \left\vert u_{0,j+1}(0)\right\vert
+\left\vert \widetilde{u}_{0,j+1}^{L}(0)\right\vert + |a_{12}(0)| C_4 \overline{K}_4^{j-1} 
\frac{1}{(j-1)!(2(j-1)+1)}(C_{\widehat v}(j+1) + |z|)^{2(j-1)+1}
\right)  \decay{z}\\
&&\leq C_3 \overline{K}_3^{j+1} \frac{1}{(j+1)!} (C_{\widehat u} (j+1)+|z|)^{2(j+1)}
\left[
\frac{C_S}{C_3} \left(\frac{1}{\overline{K}_3}\right)^{j+1} 
+ \frac{C_1}{C_3}\left(\frac{\overline{K}_1}{\overline{K}_3}\right)^{j+1} 
+ \overline{K}_4^{-2} \frac{C_4}{C_3} |a_{12}(0)|
\right]
\decay{z}. 
\end{eqnarray*}
Again, in view of our requirement (\ref{eq:b-11}), the expression in
square brackets is bounded by $1$.

\textbf{Induction Step:} We proceed by induction on $j$, the induction
hypothesis being that the estimates (\ref{57})--(\ref{60})  have been shown for all 
$i \in \N_0$ up to $j$ and will establish them for all $i \in \N_0$ and $j+1$. 

We start with estimating $\widehat{v}_{i,j+3}^L$. In view of the definition
(\ref{56}), we estimate with the induction hypothesis 
(recall $K_3 = K_4$ and $\overline{K}_3 = \overline{K}_4$; we also point
out that for the case $j=0$, the terms stemming from $\widehat v_{i-k,j+1-k}$ are 
in fact not present)
\begin{eqnarray*}
\lefteqn{
\left| \sum_{k=0}^{\min\{i,j+1\}} \frac{z^k}{k!} \left( a_{21}^{(k)}(0)  \widehat{u}_{i-k,j+1-k}^L(z) + 
a_{22}^{(k)}(0)  \widehat{v}_{i-k,j+1-k}^L(z)\right)\right|
} \\
&& \leq C_a \sum_{k=0}^{\min\{i,j+1\}} \gamma_a^k 
\Bigl[ C_3 K_3^{i-k} \overline{K}_3^{j+1-k} 
\frac{|z|^k}{(i+j+1-2k)!} (C_{\widehat u}(i+j+1-2k)+|z|)^{2(i+j+1-2k)}\\
&& \qquad +  C_4 K_4^{i-k} \overline{K}_4^{j-1-k} \frac{|z|^k}{(i+j-1-2k)!}(C_{\widehat u}(i+j-1-2k)+|z|)^{2(i+j-2k-1)}
\Bigr]\decay{z} \\
&& \leq 
C_4 K_4^{i} \overline{K}_4^{j+1} \frac{(C_{\widehat v}(i+j+1)+|z|)^{2(i+j+1)}}{(i+j+1)!}
\left[\frac{C_3}{C_4} + \overline{K}_4^{-2}\right]
\frac{C_a}{1-\gamma_a/(K_4 \overline{K}_4)}  
\decay{z};
\end{eqnarray*}
here, we employed observations of the form (note that $i+j-2k-1 \ge 0$)
\begin{align*}
& \frac{1}{(i+j-1-2k)!} (C_{\widehat u}(i+j-1-2k)+|z|)^{2(i+j-2k-1)}
\leq
\frac{1}{(i+j-1-2k)!} (C_{\widehat u}(i+j+1)+|z|)^{2(i+j-2k-1)}\\
& \leq \frac{(i+j+1)^{2k+2}}{(i+j+1)!} 
 (C_{\widehat u}(i+j+1)+|z|)^{2(i+j-2k-1)}\\
& =  \frac{(i+j+1)^{2k+2}}{(i+j+1)!} 
 (C_{\widehat u}(i+j+1)+|z|)^{-2k-2}
 (C_{\widehat u}(i+j+1)+|z|)^{2(i+j-k)}\\
&\leq 
\frac{1}{(i+j+1)!} 
 (C_{\widehat u}(i+j+1)+|z|)^{2(i+j)}
\leq 
\frac{1}{(i+j+1)!} 
 (C_{\widehat u}(i+j+1)+|z|)^{2(i+j+1)}. 
\end{align*}
From (\ref{56}) and Lemma~\ref{lemma:two-anti-derivatives}, we therefore get 
\begin{eqnarray*}
\left\vert \widehat{v}_{i,j+3}^{L}(z)\right\vert  &\leq &
C_4 K_4^i \overline{K}_4^{j+1} \frac{(C_{\widehat v}(i+j+1)+|z|)^{2(i+j+1)} }{(i+j+1)!}
\decay{z} \\
&& \mbox{}\times 
\frac{1}{\ua^2}
\left(\frac{1}{1-2/(C_{\widehat v} \ua)}\right)^2
C_a \left[\frac{C_3}{C_4} + \overline{K}_4^{-2}\right] 
\frac{1}{1-\gamma_a/(K_4 \overline{K}_4)} . 
\end{eqnarray*}
By requirement (\ref{eq:b-16}), this is the desired estimate.

We now consider $\widetilde{v}_{i,j+1}^{L}$ which satisfies (\ref{52})--(\ref{53}). 
The required estimate is proved by induction in $i$, the case $i = 0$ having
been studied previously. The right-hand side of the boundary value problem satisfies
\begin{equation*}
\left\vert RHS_{(\ref{52})-(\ref{53})}\right\vert \leq \frac{|a_{21}(0)|}{a_{11}(0)}\left\vert \left( 
\widetilde{u}_{i,j-1}^{L}\right) ^{\prime \prime }(z)\right\vert + 
C_a \left(1 + \frac{|a_{21}(0)|}{a_{11}(0)}\right)
\underset{%
k=1}{\overset{i}{\sum }}\gamma _{a}^{k}|z|^{k} \left( 
|\widetilde{u}_{i-k,j+1}^{L}(z)|+|\widetilde{v}_{i-k,j+1}^{L}(z)|\right).  
\end{equation*}
The induction hypothesis and Lemma~\ref{lemma:derivatives-of-entire-fcts} 
produce 
\begin{eqnarray}
\left\vert RHS_{(\ref{52})-(\ref{53})}\right\vert & \leq &
C_2 \frac{1}{(i+j)!} (C_{\widetilde u}(i+j) + |z|)^{2(i+j)+1} K_2^{i} \overline{K}_2^{j+1} 
\decay{z} \\
&& \mbox{} \times 
\left[
\frac{|a_{21}(0)|}{a_{11}(0)} 2 e^{\oa} \overline{K}_2^{-2}
\frac{C_1}{C_2} 
+ 
C_a \left(1 + \frac{|a_{21}(0)|}{a_{11}(0)}\right)
\frac{\gamma_a}{K_1}\frac{1}{1-\gamma_a/K_1} 
\left( \frac{C_1}{C_2} + 1\right)
\right]. \nonumber
\label{eq:100}
\end{eqnarray}
Lemma~\ref{lemma:scalar-bvp-constant-coefficients} 
and the already proven estimates for $v_{i,j+1}$ and $\widehat v_{i,j+1}^L$ 
lead to 
\begin{eqnarray*}
|\widetilde v_{i,j+1}^L(z)| &\leq& 
C_2 \frac{1}{(i+j+1)!} (C_{\widetilde u}(i+j) + |z|)^{2(i+j+1)} K_2^{i} \overline{K}_2^{j+1} 
\decay{z} \\
&& \mbox{} \times 
\Bigl[
\frac{|a_{21}(0)|}{a_{11}(0)} 2 e^{\oa} \frac{C_1}{C_2} \overline{K}_2^{-2} + 
C_a \left(1 + \frac{|a_{21}(0)|}{a_{11}(0)}\right)
\frac{\gamma_a}{K_1}\frac{1}{1-\gamma_a/K_1} 
\left( \frac{C_1}{C_2} + 1\right) \\
&& \qquad \mbox{}
+ \frac{C_S}{C_2} \left(\frac{K}{K_2}\right)^i \left(\frac{1}{\overline{K}_2}\right)^{j+1} + 
\overline{K}_4^{-2} \frac{C_4}{C_2} \left(\frac{K_4}{K_2}\right)^{i}
\left(\frac{\overline{K}_4}{\overline{K}_2}\right)^{j+1}
\Bigr],
\end{eqnarray*}
where the expression in square brackets is bounded by $1$ due to our requirement 
(\ref{eq:b-20}). 

Next we look at
\begin{equation*}
\widetilde{u}_{i,j+1}^{L}=-\frac{a_{12}(0)}{a_{11}(0)}\widetilde{v}_{i,j+1}^{L}+%
\frac{\left( \widetilde{u}_{i,j-1}^{L}\right) ^{\prime \prime }}{a_{11}(0)}-%
\frac{1}{a_{11}(0)}\underset{k=1}{\overset{i}{\sum }}\frac{\widetilde{x}^{k}%
}{k!}\left[ a_{11}^{(k)}(0)\widetilde{u}_{i-k,j+1}^{L}+a_{12}^{(k)}(0)%
\widetilde{v}_{i-k,j+1}^{L}\right]. 
\end{equation*}
Again, we proceed by induction on $i$, the case $i = 0$ having been handled
already. 
From Lemma~\ref{lemma:derivatives-of-entire-fcts} and the induction hypotheses we get 
\begin{eqnarray*}
\left\vert \widetilde{u}_{i,j+1}^{L}(z)\right\vert  &\leq &
C_1 K_1^{i} \frac{1}{(i+j+1)!}\overline{K}_1^{j+1} (C_{\widetilde u}(i+j+1)+|z|)^{2(i+j+1)} \decay{z} \\
&& \mbox{} \times 
\left[
\frac{|a_{12}(0)|}{a_{11}(0)} \frac{C_2}{C_1} + 
\overline{K}_1^{-2} \frac{2 e^{\oa}}{a_{11}(0)} + 
\frac{C_a}{a_{11}(0)} \frac{1}{1-\gamma_a/K_1} \frac{\gamma_a}{K_1} \left(1 + \frac{C_2}{C_1}\right)
\right].
\end{eqnarray*}
Again, the expression in square brackets is bounded by $1$ in view of our 
requirement (\ref{eq:b-22}). 

Finally, we consider $\widehat{u}_{i,j+1}^{\text{ }L}$ which satisfies (\ref{eq:54-55}). 
The right-hand side of the boundary value problem satisfies
\begin{eqnarray*}
\left\vert RHS_{(\ref{54})}\right\vert  &\leq &C_a |\widehat v_{i,j+1}^L(z)| + C_{a}\underset{k=1}%
{\overset{\min \{i,j+1\}}{\sum }}\gamma _{a}^{k}\left\vert z\right\vert
^{k}\left\{ \left\vert \widehat{u}_{i-k,j+1-k}^{L}\right\vert +\left\vert 
\widehat{v}_{i-k,j+1-k}^{L}\right\vert \right\}  \\
&\leq & 
C_3 K_3^{i}\overline{K}_3^{j+1} \frac{1}{(i+j)!}(C_{\widehat u}(i+j)+|z|)^{2(i+j)+1} \decay{z} \\
&&\mbox{} \times 
\left[C_a \frac{C_4}{C_3} \overline{K}_4^{-2} + 
\frac{\gamma_a}{K_3 \overline{K}_3} \frac{C_a}{1-\gamma_a/(K_3 \overline{K}_3)} 
\left( 1+ \frac{C_4}{C_3} \overline{K}_4^{-2}\right)
\right].
\end{eqnarray*}
Lemma~\ref{lemma:scalar-bvp-constant-coefficients} together with the induction
hypotheses, therefore gives us for the solution $\widehat{u}_{i,j+1}^L$
of (\ref{eq:54-55}),
\begin{eqnarray*}
\left\vert \widehat{u}_{i,j+1}^{\text{ }L}(z)\right\vert  &\leq &C_{3}%
{K}_{3}^{i}\overline{K}_{3}^{j+1} \frac{1}{(i+j+1)!} (C_{\widehat u}(i+j+1)+|z|)^{2(i+j+1)} \decay{z} \\
&& \mbox{} \times 
\Bigl[
\frac{1}{a_{11}(0)}
\left\{C_a \frac{C_4}{C_3} \overline{K}_4^{-2} + 
\frac{\gamma_a}{K_3 \overline{K}_3} \frac{C_a}{1-\gamma_a/(K_3 \overline{K}_3)} 
\left( 1+ \frac{C_4}{C_3} \overline{K}_4^{-2}\right)
\right\} \\
&&\qquad \mbox{}
+ \frac{C_S}{C_3} \left(\frac{K}{K_3}\right)^i \left(\frac{1}{\overline{K}_3}\right)^{j+1} 
+ \frac{C_1}{C_3} \left(\frac{K_1}{K_3}\right)^i \left(\frac{\overline{K}_1}{\overline{K}_3}\right)^{j+1} 
\Bigr]. 
\end{eqnarray*}
By our requirement (\ref{eq:b-23}), the expression in square brackets is bounded by $1$. 
%
\subsection{Proof of Theorem~\ref{thm_BL_3scales}}
\label{appendix:thm_BL_3scales}
We first prove two auxiliary lemmas.
\begin{lem}
\label{lemma:convexity} For every $\gamma >0$ the functions 
\begin{eqnarray*}
f_{1}(k)&:=&\gamma ^{k}(i-k)^{i-k}(j-k)^{j-k}, \\
f_{2}(k)&:=&k^{k}\gamma ^{k}(i-k)^{i-k}(j-k)^{j-k},
\end{eqnarray*}
are convex on $(0,\min \{i,j\})$.
\end{lem}

\begin{proof} It is easy to check that $\frac{d^{2}}{dk^{2}}\ln f_{1}(k)$ and $%
\frac{d^{2}}{dk^{2}}\ln f_{2}(k)>0$. 
\end{proof}

\begin{lem}
\label{lemma:elementary-properties-of-convex-functions} Let $a\leq b\leq
c\leq d$. Let $f$ be non-negative and convex on $[a,d]$. Then
\begin{equation*}
\Vert f\Vert _{L^{\infty }(b,c)}=\max \{\min \{f(a),f(b)\},\min
\{f(c),f(d)\}\}.
\end{equation*}
\end{lem}

\begin{proof} We restrict our attention to the case $a<b<c<d$ -- the general
case can be proved using similar arguments. By convexity, we have $\Vert
f\Vert _{L^{\infty }(b,c)}=\max \{f(b),f(c)\}$. We claim that 
\begin{equation}
\max \{\min \{f(a),f(b)\},\min \{f(c),f(d)\}\}=\max \{f(b),f(c)\}.
\label{eq:lemma:elementary-properties-of-convex-functions-10}
\end{equation}
Suppose $\max \{\min \{f(a),f(b)\},\min \{f(c),f(d)\}\}<\max \{f(b),f(c)\}$.
Then 
\begin{equation}
\min \{f(a),f(b)\}<\max \{f(b),f(c)\}\qquad \mbox{ and }\qquad \min
\{f(c),f(d)\}<\max \{f(b),f(c)\}.
\label{eq:lemma:elementary-properties-of-convex-functions-1}
\end{equation}
If $f(a)<f(b)$, then we write $b=\lambda a+(1-\lambda )c$ for some $\lambda \in
(0,1)$, and use convexity to get 
\begin{equation*}
f(b)=f(\lambda a+(1-\lambda )c)\leq \lambda f(a)+(1-\lambda )f(c)\leq
\lambda f(b)+(1-\lambda )f(c),
\end{equation*}
from which we conclude $f(b)<f(c).$ Thus 
\begin{equation*}
f(a)<f(b)<f(c).
\end{equation*}
The second condition in (\ref{eq:lemma:elementary-properties-of-convex-functions-1}) then produces 
$\min\{f(c),f(d)\}<f(c)$, from which we get $f(d)<f(c)$. On the other hand,
convexity implies upon writing $c=\lambda a+(1-\lambda )d$ for some $\lambda
\in (0,1)$, 
\begin{equation*}
f(c)=f(\lambda a+(1-\lambda d))\leq \lambda f(a)+(1-\lambda )f(d)<\lambda
f(c)+(1-\lambda )f(c)=f(c),
\end{equation*}
which is a contradiction. We conclude that $f(a)\geq f(b)$.

Next, we investigate the possibility $f(d)<f(c)$. We proceed analogously.
Convexity (for the points $b$, $c$, $d$) implies 
\begin{equation*}
f(d)<f(c)<f(b).
\end{equation*}
The first condition in (\ref{eq:lemma:elementary-properties-of-convex-functions-1}) then produces 
$\min\{f(a),f(b)\}<f(b)$ from which we get $f(a)<f(b)$. Using again convexity for
the points $a$, $b$, $c$, yields $f(b)<f(c)$, which is the desired
contradiction. We conclude that $f(c)\leq f(d)$.

Combining the above, have $f(a)\geq f(b)$ and $f(c)\leq f(d)$ and therefore (%
\ref{eq:lemma:elementary-properties-of-convex-functions-10}).
\end{proof}

\begin{numberedproof}{Theorem~\ref{thm_BL_3scales}}
We have to estimate the terms on the right-hand side of (\ref{eq:LhatU^M}). 
A direct check shows that the estimate is valid for the 
special case $M_2 = 0$. Hence, we will assume $M_2 \ge 1$. 

Throughout the proof, we will use that $\widehat x$ is real and strictly positive.

First, we estimate the double sum 
\begin{eqnarray*}
\sum_{i=0}^{M_1} \sum_{j=M_2-1}^{M_2} \mu^i \left(\frac{\varepsilon}{\mu}\right)^j 
|\widehat u_{i,j}^{\prime\prime}|.
\end{eqnarray*}
Using the bounds of Theorem~\ref{thm11} and the Cauchy integral theorem for
derivatives (with contour $\partial B_1(x)$), we get with the aid of 
Lemma~\ref{lemma:elementary-properties-of-factorial},
$$
|\widehat u_{i,j}^{\prime\prime}(x)| \leq C \widetilde \gamma^{i+j} i^i (j+1)^{j+1},
$$
for a suitable constant $\widetilde \gamma$. Therefore, if $\mu \widetilde \gamma (M_1+1) \leq 1/2$, 
then 
\begin{eqnarray*}
\sum_{i=0}^{M_1} \sum_{j=M_2-1}^{M_2} \mu^i \left(\frac{\varepsilon}{\mu}\right)^j 
|\widehat u_{i,j}^{\prime\prime}|
\leq C \left(\frac{\varepsilon}{\mu}\right)^{M_2-1}.
\end{eqnarray*}
Hence, the double sum  in (\ref{eq:LhatU^M}) can be estimated in the desired fashion
by requiring $\gamma \ge \widetilde \gamma/2$.

We now turn to the triple sum in (\ref{eq:LhatU^M}). 
We start by writing the conditions 
(\ref{eq:LhatU^M-conditions-on-i-j-k}) 
on the indices $i$, $j$, and $k$ appearing in the triple sum 
(\ref{eq:LhatU^M}) in a more compact form, by setting 
\begin{equation}
\label{eq:LhatU^M-conditions-on-i-j-k-compact}
{\mathcal I}:= \{(i,j,k)\colon (i \ge M_1+1 \vee j \ge M_2-1) 
\wedge \max\{i-M_1,j-M_2\} \leq k \leq \min\{i,j\} \}.
\end{equation}
Thus, the triple sum can be written as 
\begin{equation}
\label{eq:LhatU^M-compact}
S = \sum_{(i,j,k) \in {\mathcal I}} 
\mu ^{i}(\varepsilon /\mu)^{j}\widehat{x}%
^{k}{\mathbf A}_{k}\left( 
\begin{array}{c}
\widehat{u}_{i-k,j-k} \\ 
\widehat{v}_{i-k,j-k}%
\end{array}%
\right) .
\end{equation}
Using (\ref{59}), (\ref{60}) and $\Vert {\mathbf A}_{k}\Vert_{\ell^1} \leq C_{A}\gamma _{a}^{k}$,
where $C_A:= 2 C_a$, 
we obtain with $K\geq \max \{K_{3},K_{4}\}$, $\overline{K}\geq \max \{%
\overline{K}_{3},\overline{K}_{4}\}$, $\widehat{C}\geq 
\max \{C_{\widehat{u}},C_{%
\widehat{v}}\}$ for $\widehat x \ge 0$, with the aid of 
Lemma~\ref{lemma:elementary-properties-of-factorial},
\begin{eqnarray*}
\left\vert L_{\varepsilon,\mu }\widehat{{\mathbf{U}}}_{BL}^{M}\right\vert  &\leq
&CC_{A}\sum_{(i,j,k) \in {\mathcal I}} 
\mu^{i}(\varepsilon /\mu )^{j}
\widehat{x}^{k}\gamma _{a}^{k}{K}^{i-k}\overline{K}^{j-k}\frac{\left( 
\widehat{C}(i-k+j-k)+|\widehat{x}|\right) ^{2(i-k+j-k)}}{(i-k+j-k)!}e^{-\ua%
\widehat{x}} \\
&\leq &CC_{A}\sum_{(i,j,k) \in {\mathcal I}}
\mu ^{i}(\varepsilon /\mu )^{j}{K}^{i}\overline{K}^{j}
\frac{\widehat{x}^{k}\gamma_{a}^{k}}{\overline{K}^{k}K^{k}}\gamma
^{i-k+j-k}(i-k)^{i-k}(j-k)^{j-k}e^{-3\ua \widehat{x}/4} \\
&\leq &CC_{A}\sum_{ (i,j,k) \in {\mathcal I}}
\mu ^{i-k}(\varepsilon /\mu )^{j-k}{K}^{i}\overline{K}^{j}\frac{(\widehat{x}%
\varepsilon )^{k}\gamma _{a}^{k}}{\overline{K}^{k}K^{k}}\gamma
^{i-k+j-k}(i-k)^{i-k}(j-k)^{j-k}e^{-3\ua \widehat{x}/4}; 
\end{eqnarray*}
here, we selected $\gamma > 0$ suitable in dependence on $\widehat C$ and  $\ua$. 
It is convenient to abbreviate 
$$
f(i,j,k) = 
\mu^{i-k}(\varepsilon/\mu)^{j-k} \widetilde K_1^{i-k} \widetilde K_2^{j-k} 
(\widehat x \varepsilon \gamma_a)^k (i-k)^{i-k} (j-k)^{j-k},  
$$
with $\widetilde K_1:= K \gamma$ and $\widetilde K_2:= \overline{K} \gamma$. 
Hence, we wish to estimate 
\begin{equation}
\label{eq:wish-to-estimate}
e^{-3 \ua \widehat x/4} \sum_{(i,j,k) \in {\mathcal I}} f(i,j,k).
\end{equation}
Next, in order unify the presentation, we consider the cases 
$j = M_2 -1$ and $j = M_2$ separately. That is, we write 
\begin{eqnarray*}
{\mathcal I} &\subset &  \widetilde{\mathcal I} \cup {\mathcal I}_1, \\
{\mathcal I}_1 &:=& 
\{(i,j,k)\colon  j \in \{M_2-1,M_2\} 
\wedge \max\{ i -M_1 ,0\} \leq k \leq \min\{i,j\}\}, \\
\widetilde{\mathcal I} &:=& 
\{(i,j,k)\colon (i \ge M_1 + 1 \vee j \ge M_2 + 1) \wedge
\max\{i-M_1,j-M_2\} \leq k \leq \min\{i,j\}\},
\end{eqnarray*}
and estimate the sums over ${\mathcal I}_1$ and $\widetilde {\mathcal I}$
separately. 

The structure of the remainder of the proof is as follows: 
\begin{itemize}
\item 
In {\em Step 1}, we estimate  
$e^{- \ua \widehat x/4} \sum_{(i,j,k) \in {\mathcal I}_1} f(i,j,k)$; 
\item 
In {\em Step 2}, 
we estimate  
$e^{- \ua \widehat x/4} \sum_{(i,j,k) \in \widetilde {\mathcal I}} f(i,j,k)$; 
\item 
Finally, in {\em Step 3}, we consider the case 
$\widehat x \varepsilon \ge c$ and show that then 
$|L_{\varepsilon,\mu} \widehat{\mathbf U}^M_{BL}(\widehat x)| 
\leq C e^{-\beta \widehat x}$.  
\end{itemize} 

{\em Step 1:}
We estimate 
$$
\sum_{(i,j,k) \in {\mathcal I}_1} f(i,j,k)
= 
\sum_{i = 0}^{M_1} \sum_{j = M_2-1}^{M_2} \sum_{k=0}^{\min\{i,j\}} f(i,j,k)
+ 
\sum_{i = M_1+1}^{\infty} \sum_{j = M_2-1}^{M_2} \sum_{k=i-M_1}^{\min\{i,j\}} f(i,j,k)
=:S_1 + S_2,
$$
using the convexity properties of the function $f$. Specifically, in order
to estimate $S_1$, we first consider the case $M_1 \leq M_2 - 1$. 

{\em Step 1a:}
Assume $M_1 \leq M_2 - 1$. 
Then by convexity of the function $k \mapsto f(i,j,k)$ 
(cf.~Lemma~\ref{lemma:convexity}),
\begin{eqnarray*}
S_1 &=&  \sum_{i=0}^{M_1} \sum_{j=M_2-1}^{M_2} \sum_{k=0}^i f(i,j,k) 
\leq \sum_{i=0}^{M_1} \sum_{j=M_2-1}^{M_2} (i+1) 
\max\{f(i,j,0), f(i,j,i)\}\\
&\leq& (M_1+1)\sum_{i=0}^{M_1} \sum_{j=M_2-1}^{M_2} 
\left[ \mu^{i} (\varepsilon/\mu)^j \widetilde K_1^i\widetilde K_2^j i^i j^j + 
       (\varepsilon/\mu)^{j-i} \widetilde K_2^{j-i} (\widehat x \varepsilon \gamma_a)^i (j-i)^{j-i}
\right]\\
&\leq& C (M_1+1) (\varepsilon/\mu \widetilde K_2 (M_2-1))^{M_2-1} + C (M_1+1)^2 
\max\{(\widetilde K_2 \varepsilon/\mu(M_2-1))^{M_2-1}, (\widehat x \varepsilon\gamma_a)^{M_1}\},
\end{eqnarray*}
where we employed the assumption that $\mu M_1$  and 
$\varepsilon/\mu (M_2-1)$ are sufficiently small 
and the convexity of the second term (as a function of $i$). From 
Lemma~\ref{lemma:elementary-properties-of-convex-functions}, we can bound
$$
e^{-\ua \widehat x/4} \widehat x^{M_1} \leq C (\widehat \gamma M_1)^{M_1},
$$
for suitable $\widehat \gamma$, so that we obtain together with 
$M_1 \leq M_2 - 1$ and the trivial bound $\varepsilon \leq \mu$,
\begin{equation}
\label{eq:foo-100}
e^{-\ua \widehat x/4} S_1 \leq 
C \left[ \left(\frac{\varepsilon}{\mu}(M_2 - 1) \gamma\right)^{M_2-1} + (\mu M_1 \gamma)^{M_1}\right],
\end{equation}
for suitable constants $C$, $\gamma > 0$. 

{\em Step 1b:}
Analogous reasoning covers the case $M_1 = M_2$. More precisely, 
we write in this case 
\begin{eqnarray*}
\sum_{i=0}^{M_1}\sum_{j=M_2-1}^{M_2}  \sum_{k=0}^{\min\{i,j\}} f(i,j,k)
= \sum_{i=0}^{M_1-1} \sum_{j=M_2-1}^{M_2} \sum_{k=0}^i f(i,j,k) 
+ \sum_{j=M_2-1}^{M_2} \sum_{k=0}^j f(M_1,j,k) =:S_{1,1} + S_{1,2}.
\end{eqnarray*}
For $S_{1,1}$, the reasoning of the above case (``\emph{Step 1a}'') is applicable, 
since $M_1 - 1 \leq M_2-1$, and yields 
$$
e^{-\ua \widehat x/4} S_{1,1} \leq 
C \left[ \left(\frac{\varepsilon}{\mu}(M_2 - 1) \gamma\right)^{M_2-1} + 
(\varepsilon (M_1-1) \gamma)^{M_1-1}\right]
\leq 
C \left(\frac{\varepsilon}{\mu}(M_2 - 1) \gamma\right)^{M_2-1}, 
$$
where in the last step, we used the trivial bound
$\varepsilon \leq \varepsilon/\mu$ in view of $\mu \leq 1$ and the 
fact that $M_1 = M_2$. For $S_{1,2}$, we use convexity of 
$f$ in the third argument to arrive at 
\begin{eqnarray*} 
S_{1,2} &\leq& (M_2+1) \sum_{j=M_2-1}^{M_2} \max\{ f(M_1,j,0), f(M_1,j,j)\}. 
\end{eqnarray*} 
Estimating 
$\max\{f(M_1,j,0),f(M_1,j,j)\} \leq f(M_1,j,0) + f(M_1,j,j)$, we 
get 
$$
f(M_1,j,0) + 
f(M_1,j,j) \leq 
\mu^{M_1}(\varepsilon/\mu)^{j} \widetilde K_1^{M_1} \widetilde K_2^{j} 
M_1^{M_1} j^j
+ 
\mu^{M_1-j}\widetilde K_1^{M_1-j} 
(\widehat x \varepsilon \gamma_a)^j (M_1-j)^{M_1-j}.
$$
Using the fact that $\mu M_1$ and $\varepsilon/\mu M_2$ are sufficiently small
and that $M_1 = M_2$, we get 
\begin{eqnarray*}
\lefteqn{
\sum_{j=M_2-1}^{M_2} \max\{f(M_1,j,0),f(M_1,j,j)\} 
\leq 
} \\
&& 
(\mu \widetilde K_1 M_1)^{M_1}
\left(\frac{\varepsilon}{\mu} \widetilde K_2 (M_2-1)\right)^{M_2-1}  
+ (\widehat x \varepsilon \gamma_a)^{M_1} + 
\mu \widetilde K_1 (\widehat x \varepsilon \gamma_a)^{M_2-1}.
\end{eqnarray*}
Upon writing 
$$
\mu (\widehat x \varepsilon \gamma_a)^{M_2-1} = 
\left(\widehat x \frac{\varepsilon}{\mu} \gamma_a\right)^{M_2-1} \mu^{M_2}
\leq 
\left(\widehat x \frac{\varepsilon}{\mu} \gamma_a\right)^{M_2-1} ,
$$
we can use (\ref{eq:foo-100}) to argue as in \emph{Step 1a}.

{\em Step 1c:}
We now consider $S_1$ for the case $M_1 \ge  M_2 + 1$. We write 
$$
S_1 = \sum_{i=0}^{M_2-1} \sum_{j=M_2-1}^{M_2} \sum_{k=0}^{i} f(i,j,k) + 
      \sum_{i=M_2}^{M_1} \sum_{j=M_2-1}^{M_2} \sum_{k=0}^{j} f(i,j,k). 
$$
Checking the arguments for the case $M_1 \leq M_2-1$ 
of \emph{Step 1a}, we see that we can bound 
$$
e^{-\ua \widehat x/4} 
\sum_{i=0}^{M_2-1} \sum_{j=M_2-1}^{M_2} \sum_{k=0}^{i} f(i,j,k) 
\leq C \left[ \left(\frac{\varepsilon}{\mu}(M_2 - 1) \gamma\right)^{M_2-1} + 
(\varepsilon (M_2-1) \gamma)^{M_2-1}\right],
$$
which has the desired form in view of $\mu \leq 1$. 

It therefore remains to estimate 
$\sum_{i=M_2}^{M_1} \sum_{j=M_2-1}^{M_2} \sum_{k=0}^{j} f(i,j,k)$, which we 
do again by exploiting convexity of the function $k \mapsto f(i,j,k)$. We get 
\begin{eqnarray*}
\sum_{i=M_2}^{M_1} \sum_{j=M_2-1}^{M_2} \sum_{k=0}^{j} f(i,j,k)
&\leq& \sum_{i=M_2}^{M_1} \sum_{j=M_2-1}^{M_2} (j+1) 
\max\{f(i,j,0), f(i,j,j)\} \\
&\leq& (M_2 +1)\sum_{i=M_2}^{M_1} \sum_{j=M_2-1}^{M_2} 
\left[ \mu^i \widetilde K_1^i (\varepsilon/\mu)^j \widetilde K_2^j i^i j^j + 
       \mu^{i-j} \widetilde K_1^{i-j} (\widehat x \varepsilon \gamma_a)^j (i-j)^{i-j}
\right]\\
&\leq& C (M_2+1) \left[
(\varepsilon/\mu (M_2-1)\gamma)^{M_2-1} + 
\mu (M_1+1)(\widehat x \varepsilon \gamma_a)^{M_2-1}
+ 
(\widehat x \varepsilon \gamma_a)^{M_2}
\right]\\
&\leq& C (M_2+1) \left[
(\varepsilon/\mu (M_2-1)\gamma)^{M_2-1} + 
(\widehat x \varepsilon \gamma_a)^{M_2-1}
\right],
\end{eqnarray*}
where we exploited again the assumption 
that $\mu (M_1+1)$ and $\varepsilon/\mu (M_2+1)$ are sufficiently
small so that sums can be estimated by convergent geometric series. 
The contribution 
$M_2 (\widehat x \varepsilon \gamma_a)^{M_2-1}$ can now 
be estimated as before using (\ref{eq:foo-100}). 

{\em Step 1d:}
We now turn to $S_2$ and start with assuming $M_1+1 \ge M_2$. 
Then $S_2$ takes the form
$$
S_2 =\sum_{i=M_1+1}^\infty \sum_{j=M_2-1}^{M_2} \sum_{k=i-M_1}^j f(i,j,k)
 =\sum_{i=M_1+1}^{M_1+M_2} \sum_{j=M_2-1}^{M_2} \sum_{k=i-M_1}^j f(i,j,k). 
$$
Convexity of $k \mapsto f(i,j,k)$ allows us to infer 
\begin{eqnarray*}
\max_{k=i-M_1,\ldots,j} f(i,j,k) & \leq & 
\mu^{M_1} \widetilde K_1^{M_1} (\widehat x \varepsilon \gamma_a)^{i-M_1} 
(\varepsilon/\mu)^{j-i+M_1}\widetilde K_2^{j-i+M_1}
M_1^{M_1} (j-i+M_1)^{j-i+M_1} \\
&& \mbox{}+ \mu^{i-j} \widetilde K_1^{i-j} (\widehat x \varepsilon \gamma_a)^j (i-j)^{i-j}.
\end{eqnarray*}
Since $0 \leq j-i+M_1 \leq M_2$, we can exploit that 
$\varepsilon/\mu (M_2+1)$ is sufficiently small and since
$0 \leq i - j \leq M_1$, we can use that $\mu M_1$ is sufficiently small 
to conclude 
$$
S_2 \leq C (M_2 +1)
\left[ 
\sum_{i=M_1+1}^{M_1+M_2} 
(\mu \widetilde K_1 M_1)^{M_1} (\widehat x \varepsilon \gamma_a)^{i-M_1}
+ \sum_{j=M_2-1}^{M_2} (\widehat x \varepsilon \gamma_a)^{j}
\right]. 
$$
With Lemma~\ref{lemma:elementary-properties-of-factorial}, we therefore get 
for suitable $\gamma^\prime$, since $i - M_1 \leq M_2$,
\begin{eqnarray*}
e^{- \widehat x/4} S_2 &\leq& 
C (M_2+1) \left[ 
(\mu M_1 \widetilde K_1)^{M_1} 
\sum_{i=M_1+1}^{M_1+M_2} (\varepsilon M_2 \gamma^\prime)^{i-M_1} 
+ \sum_{j=M_2-1}^{M_2} (\varepsilon M_2 \gamma^\prime)^{j}
\right] \\
&\leq& 
C (M_2+1) \left[ 
(\mu M_1 \widetilde K_1)^{M_1} \varepsilon M_2 + 
(\varepsilon M_2 \gamma^\prime)^{M_2-1}\right], 
\end{eqnarray*}
where, in the last step, we exploited again that $\mu M_1$ and $\varepsilon M_2$
are sufficiently small. Since $M_2 \leq M_1+1$, this last estimate
has the desired form. 

{\em Step 1e:}
Next, we consider the case $M_1 + 1 < M_2$. We write 
\begin{eqnarray*}
S_2 & =&  \sum_{i=M_1+1}^\infty \sum_{j=M_2-1}^{M_2} \sum_{k=i-M_1}^{\min\{i,j\}} f(i,j,k) \\
&=& 
\sum_{i=M_1}^{M_2-1} \sum_{j=M_2-1}^{M_2} \sum_{k=i-M_1}^{i} f(i,j,k) 
+ 
\sum_{i=M_2}^{\infty} \sum_{j=M_2-1}^{M_2} \sum_{k=i-M_1}^{j} f(i,j,k) 
=:S_{2,1} + S_{2,2}. 
\end{eqnarray*}
We note that 
$$
S_{2,2} = \sum_{i=M_2}^{M_1+M_2} \sum_{j=M_2-1}^{M_2} \sum_{k=i-M_1}^j f(i,j,k).
$$
We may bound $S_{2,2}$ using arguments similar to those of \emph{Step 1d}. 
The convexity of $k \mapsto f(i,j,k)$ yields 
\begin{eqnarray*}
\max_{k=i-M_1,\ldots,i} f(i,j,k) & \leq & 
\mu^{M_1} \widetilde K_1^{M_1} (\widehat x \varepsilon \gamma_a)^{i-M_1} 
(\varepsilon/\mu)^{j-i+M_1}\widetilde K_2^{j-i+M_1}
M_1^{M_1} (j-i+M_1)^{j-i+M_1} +\\
&& \mbox{}+ (\widehat x \varepsilon \gamma_a)^i.  
\end{eqnarray*}
Since $0 \leq j-i+M_1 \leq M_1 \leq M_2$, we can estimate by convexity
$$
\left(\frac{\varepsilon}{\mu} \widetilde K_2 (j-i-M_1)\right)^{j-i-M_1} 
\leq \left( 1 + (\varepsilon/\mu \widetilde K_2 M_2)^{M_2}\right)
\leq C, 
$$
since we assume that $\varepsilon/\mu M_2$ is sufficiently small. Hence, 
we get 
$$
S_{2,2} \leq C (M_2+1) 
\sum_{i=M_2}^{M_1+M_2} 
(\mu \widetilde K_1 M_1)^{M_1} 
(\widehat x \varepsilon \gamma_a)^{i-M_1} 
+ (\widehat x \varepsilon \gamma_a)^i.
$$
With Lemma~\ref{lemma:elementary-properties-of-factorial} we therefore get 
$$
e^{- \ua\widehat x/4} S_{2,2} 
\leq C(M_2+1) 
\sum_{i=M_2}^{M_1+M_2} 
(\mu \widetilde K_1 M_1)^{M_1} (\varepsilon \gamma^\prime (i-M_1))^{i-M_1} 
+ (\varepsilon \gamma^\prime i)^i,
$$
for suitable $\gamma^\prime > 0$. Recalling that $M_1 \leq M_2$, we 
obtain by assuming that $\varepsilon M_2$ is sufficiently small, with 
the aid of geometric series arguments, 
$$
e^{- \ua \widehat x/4} S_{2,2} 
\leq C(M_2+1) 
\left[ 
(\mu \widetilde K_1 M_1)^{M_1} 
+ (\varepsilon \gamma^\prime (M_2+1))^{M_2}
\right].
$$

For $S_{2,1}$ we note 
\begin{eqnarray*}
\lefteqn{
\max_{k  = i-M_1,\ldots,i} f(i,j,k) \leq }\\
&& 
(\varepsilon/\mu)^{j-i} \widetilde K_2^{j-i} (j-i)^{j-i} (\widehat x \varepsilon \gamma_a)^i + 
\mu^{M_1} \widetilde K_1^{M_1} M_1^{M_1} (\varepsilon/\mu)^{j-i+M_1} 
(\widehat x \varepsilon \gamma_a)^{i-M_1} 
\widetilde K_2^{j-i+M_1} (j-i + M_1 )^{j-i+M_1}. 
\end{eqnarray*}
We note $i$ and $j$ in the definition of the sum $S_{2,1}$ are 
such that 
$0 \leq j - i\leq M_2$ and $M_1 \leq j-i+M_1 \leq M_2$. Hence, 
this setting simplifies under the assumption
that $\mu M_1$ and $\varepsilon/\mu M_2$ are sufficiently small:
\begin{eqnarray*}
\max_{k  = i-M_1,\ldots,i} f(i,j,k) & \leq &
(\widehat x \varepsilon \gamma_a)^i + 
(\mu \widetilde K_1 M_1)^{M_1} 
(\widehat x \varepsilon \gamma_a)^{i-M_1} 
(\varepsilon/\mu M_2 \widetilde K_2)^{M_1}  \\
& \leq &
(\widehat x \varepsilon \gamma_a)^i + 
(\mu \widetilde K_1 M_1)^{M_1} 
(\widehat x \varepsilon \gamma_a)^{i-M_1}. 
\end{eqnarray*}
With the aid of Lemma~\ref{lemma:elementary-properties-of-factorial}, 
we estimate with suitable $\gamma^\prime > 0$ ,
\begin{eqnarray*}
e^{-\ua  \widehat x/4} S_{2,1} &\leq& 
(M_1+1) \sum_{i=M_1}^{M_2-1} 
(\varepsilon \gamma^\prime i)^i + 
(\mu \widetilde K_1 M_1)^{M_1} 
(\varepsilon \gamma^\prime (i-M_1))^{i-M_1}  \\
&\leq& C (M_1+1) \left[ (\varepsilon M_2 \gamma^\prime)^{M_1} + 
(\mu \widetilde K_1 M_1)^{M_1} \right],
\end{eqnarray*}
where we employed again suitable geometric series arguments. 
Using $\varepsilon = (\varepsilon/\mu) \mu$, we get 
\begin{eqnarray*}
e^{-\ua  \widehat x/4} S_{2,1} &\leq& 
C (M_1+1) \left[ 
\mu^{M_1} \left(\frac{\varepsilon}{\mu} M_2 \gamma^\prime\right)^{M_1} + 
(\mu \widetilde K_1 M_1)^{M_1} \right],
\end{eqnarray*}
which has the desired form since $\varepsilon/\mu M_2$ is assumed to be 
sufficiently small.

{\em Step 2:}
Before proceding, we point out that we make the assumption
\begin{equation}
\label{eq:widehat_x_varepsilon_sufficiently_small}
\widehat x \varepsilon \gamma_a \leq \frac{1}{2},
\end{equation}
as the converse case is covered in \emph{Step 3} below. 
We estimate the contribution arising from the sum over $\widetilde{\mathcal I}$ 
and show
\begin{equation}
\label{eq:goal-of-step-2}
e^{-\ua  \widehat x/4} 
\sum_{(i,j,k) \in \widetilde {\mathcal I}}
f(i,j,k) \leq C \left[ \frac{\varepsilon}{\mu} (\varepsilon M_2 \gamma)^{M_2} + 
\mu (\mu M_1 \gamma)^{M_1}\right]; 
\end{equation}
in fact, the reasoning below shows a slightly sharper estimate. 

Noting the appearance of $\max\{i-M_1,j-M_2\}$, we have 
to study the cases $i - M_1 \ge j-M_2$ and the reverse case $i - M_1 < j - M_2$
separately. We restrict our attention here to the case 
$i-M_{1}\leq j-M_{2}$, since the reverse case is easily obtained using
the same arguments (effectively, $M_{1}$ and $M_{2}$ and $\mu $ and 
$\varepsilon /\mu $ reverse their roles). In this situation, we consider the
following two subcases separately: 
\begin{eqnarray}
(i &\geq &M_{1}+1\ \vee \ j\geq M_{2}+1)\ \wedge \ (i-M_{1}\leq j-M_{2})\
\wedge \ (j-M_{2}\leq k\leq i)\ \wedge \ (i\leq j),  
\qquad 
\qquad 
\label{eq:case1-3scale-remainder} \\
(i &\geq &M_{1}+1\ \vee \ j\geq M_{2}+1)\ \wedge \ (i-M_{1}\leq j-M_{2})\
\wedge \ (j-M_{2}\leq k\leq j)\ \wedge \ (j\leq i).  
\qquad 
\qquad 
\label{eq:case2-3scale-remainder}
\end{eqnarray}
A key ingredient of our proof is the convexity assertion given in Lemma~\ref{lemma:convexity}. 
We recall that a non-negative convex function attains its maximum at the boundary (i.e., in the 
univariate case, at the endpoints of an interval).

{\em Step 2a:} We consider the case (\ref{eq:case1-3scale-remainder}), which can be
 further subdivided into the cases $i\geq M_{1}+1$ and $j\geq M_{2}+1$. 

{\em Step 2a1:} We consider the case (\ref{eq:case1-3scale-remainder}) 
with the further assumption $i \geq M_1+ 1$. We get 
\begin{equation*}
(i\geq M_{1}+1)\ \wedge \ (i\leq j)\ \wedge \ (i\leq j-M_{2}+M_{1})\ \wedge
\ (j-M_{2}\leq k\leq i).
\end{equation*}
That is, we have to estimate the triple sum 
$$
T:= 
\sum_{i\geq M_{1}+1}\sum_{\substack{ j\geq i  \\ j\geq i+M_{2}-M_{1}}}%
^{j\leq i+M_{2}}\sum_{k=j-M_{2}}^{i} f(i,j,k),
$$
where the summation index is $i$ for the outermost sum, $j$ for the middle
sum, and $k$ for the innermost sum. 
This triple sum is estimated using 
convexity of the argument. The innermost sum has at most $i-j+M_2+1$
terms and $k \mapsto f(i,j,k)$ is convex by Lemma~\ref{lemma:convexity}. Hence, we obtain 
\begin{eqnarray*}
T & \leq & 
\sum_{i\geq M_{1}+1}\sum_{\substack{ j\geq i  \\ j\geq i+M_{2}-M_{1}}}%
^{j\leq i+M_{2}} (i-j+M_2+1) \max\{f(i,j,i), f(i,j,j-M_2)\} \\
&\leq& 
(M_1 +1) 
\sum_{i\geq M_{1}+1}\sum_{\substack{ j\geq i  \\ j\geq i+M_{2}-M_{1}}}%
^{j\leq i+M_{2}} \max\{f(i,j,i), f(i,j,j-M_2)\}, 
\end{eqnarray*}
where, in the second step we have used the restrictions on the sum on $j$
to bound $i-j+M_2+1 \leq M_1 + 1$. 
Writing out $f(i,j,i)$ and $f(i,j,j-M_2)$, we have 
\begin{eqnarray*}
f(i,j,i) &=& \widetilde K_2^{j-i} (\varepsilon/\mu)^{j-i} (\widehat x\varepsilon\gamma_a )^i
(j-i)^{j-i}, \\
f(i,j,j-M_2) &=& 
\mu^{i-j+M_2} (\varepsilon/\mu)^{M_2} \widetilde K_1^{i-j+M_2}
(\widehat x \varepsilon \gamma_a)^{j-M_2}
\widetilde K_2^{M_2} 
(i-j+M_2)^{i-j+M_2} M_2^{M_2},  
\end{eqnarray*}
which are again convex functions of $j$ by Lemma~\ref{lemma:convexity}. 
Turning now to the sum on $j$, we see that it has at most $M_2+1$ terms. 
In view of 
Lemma~\ref{lemma:elementary-properties-of-convex-functions}, we bound 
for the relevant $j$: 
\begin{subequations}
\label{eq:fooo}
\begin{eqnarray}
f(i,j,i) \leq A:= \max\{f(i,j,i)|_{j=i+M_2}, f(i,j,i)|_{j=i}\} = \max\{f(i,i+M_2,i), f(i,i,i)\},\\
\nonumber 
f(i,j,j-M_2) \leq B:= \max\{f(i,j,j-M_2)|_{j=i+M_2}, f(i,j,j-M_2)|_{j=i+M_2-M_1}\}  \\
=\max\{f(i,i+M_2,i), f(i,i+M_2-M_1,i-M_1)\}. 
\end{eqnarray}
\end{subequations}
More explicitly, these are 
\begin{eqnarray*}
A &\leq& (\widehat x \varepsilon\gamma_a)^i 
\max\{1, \widetilde K_2^{M_2} (\varepsilon/\mu)^{M_2} M_2^{M_2}\}, \\
B &\leq & (\varepsilon/\mu)^{M_2} M_2^{M_2} \widetilde K_2^{M_2} 
\max\{ (\widehat x \varepsilon \gamma_a)^i, (\mu \widetilde K_1 M_1)^{M_1}
(\widehat x \varepsilon \gamma_a)^{i-M_1} \} .
\end{eqnarray*}

Writing out the sum and using convexity of the argument, we can estimate 
\begin{eqnarray*}
S &:=&
(M_1+1) \sum_{i\geq M_{1}+1}\sum_{\substack{ j\geq i  \\ j\geq i+M_{2}-M_{1}}}%
^{j\leq i+M_{2}}\sum_{k=j-M_{2}}^{i} \max\{f(i,j,i),f(i,j,j-M_2)\}\\
& \leq & (M_1+1)(\min\{M_1,M_2\}+1) \sum_{i=M_1+1}^\infty A + B .
\end{eqnarray*}
Using the facts that $\mu M_1$ and $\varepsilon/\mu M_2$ are sufficiently
small, we can estimate with geometric series arguments, in view 
of our assumption (\ref{eq:widehat_x_varepsilon_sufficiently_small}): 
\begin{eqnarray*}
S \leq C (M_1+1)(\min\{M_1,M_2\}+1) 
\left[
(\widehat x \varepsilon \gamma_a)^{M_1+1} + (\varepsilon/\mu M_2 \widetilde K_2)^{M_2}
\widehat x \varepsilon \gamma_a
\right].
\end{eqnarray*}
Hence, using Lemma~\ref{lemma:elementary-properties-of-factorial} 
to control $\widehat x^{M_1+1}$ in the first term and $\widehat x$
in the second term, gives with
suitable $\gamma^\prime$, 
\begin{eqnarray*}
e^{-\ua  \widehat x/4} S &\leq& 
C (M_1+1)(\min\{M_1,M_2\}+1) 
\left[ (\varepsilon (M_1 +1)\gamma^\prime)^{M_1+1} + \varepsilon (\varepsilon/\mu \widetilde K_2 M_2)^{M_2}
\right]\\
&\leq& 
C 
\left[ (\varepsilon (M_1 +1)\gamma^{\prime\prime})^{M_1+1} + \frac{\varepsilon}{\mu} (\varepsilon/\mu \gamma^{\prime\prime} M_2)^{M_2}
\right], 
\end{eqnarray*}
where, in the second step, we selected $\gamma^{\prime\prime}$ suitable and used the fact
that $\mu M_1$ can be controlled. Since $\varepsilon \leq \mu$, we obtain the 
desired estimate  (\ref{eq:goal-of-step-2}).

{\em Step 2a2:} The other case is $j\geq M_{2}+1$. This leads to 
\begin{equation*}
(j\geq M_{2}+1)\ \wedge \ (i\leq j)\ \wedge \ (i\leq j-M_{2}+M_{1})\ \wedge
\ (j-M_{2}\leq k\leq i).
\end{equation*}
Writing out the sum, we have 
\begin{eqnarray*}
T^\prime &:= &\sum_{j\geq M_{2}+1}\sum_{i\geq j-M_{2}}^{\substack{ i\leq j  \\ i\leq
j-M_{2}+M_{1}}}\sum_{k=j-M_{2}}^{i} f(i,j,k).
\end{eqnarray*}
The innermost sum has at most $M_2+1$ terms which, by convexity, can be estimated
by 
$$
\max\{f(i,j,i), f(i,j,j-M_2)\}.
$$
Writing out $f(i,j,i)$ and $f(i,j,j-M_2)$, we have 
\begin{eqnarray*}
f(i,j,i) &=& (\varepsilon/\mu)^{j-i} \widetilde K_2^{j-i} (\widehat x \varepsilon\gamma_a)^i 
(j-i)^{j-i}, \\
f(i,j,j-M_2) &=& \mu^{i-j+M_2} \widetilde K_1^{i-j+M_2} (i-j+M_2)^{i-j+M_2} 
(\varepsilon/\mu)^{M_2} \widetilde K_2^{M_2} (\widehat x \varepsilon\gamma_a)^{j - M_2}
M_2^{M_2},
\end{eqnarray*}
which are again convex as functions of $i$. Hence, by 
Lemma~\ref{lemma:elementary-properties-of-convex-functions}, we bound 
for the relevant $j$: 
\begin{eqnarray*}
f(i,j,i) &\leq& A^\prime:=\max\{f(i,j,i)|_{i=j-M_2}, f(i,j,i)|_{i=j}\} \\
&& = \max\{(\varepsilon/\mu M_2 \widetilde K_2)^{M_2} (\widehat x \varepsilon \gamma_a)^{j-M_2}, 
(\widehat x \varepsilon \gamma_a)^{j}\}, \\
f(i,j,j-M_2) &\leq& B^\prime:= \max\{ f(i,j,j-M_2)|_{i=j-M_2}, f(i,j,j-M_2)|_{i=j-M_2+M_1} \}\\
&& \leq (\varepsilon/\mu M_2 \widetilde K_2)^{M_2} (\widehat x \varepsilon \gamma_a)^{j-M_2}
\max\{1, (\mu M_1 \widetilde K_1)^{M_1} \}
\}.
\end{eqnarray*}
The middle sum in $T^\prime$ has at most $\min\{M_1,M_2\}+1$ terms. 
Hence, we arrive at 
\begin{eqnarray*}
T^\prime &\leq & (M_2+1)(\min\{M_1,M_2\}+1) \sum_{j \ge M_2+1} 
(\varepsilon/\mu M_2 \widetilde K_2)^{M_2} (\widehat x \varepsilon\gamma_a)^{j-M_2} + (\widehat x \varepsilon \gamma_a)^{j}\\
&\leq& C (M_2+1)(\min\{M_1,M_2\}+1) 
\left[(\varepsilon/\mu M_2 \widetilde K_2)^{M_2} (\widehat x \varepsilon \gamma_a)+ 
(\widehat x\varepsilon \gamma_a)^{M_2+1}\right], 
\end{eqnarray*}
where, in the second step we have employed geometric sum arguments, 
which are applicable in view of (\ref{eq:widehat_x_varepsilon_sufficiently_small}). 
Reasoning as at the end of \emph{Step 2a1}, we get for suitable $\gamma^\prime$, 
\begin{eqnarray*}
e^{-\ua  \widehat x/4} T^\prime &\leq& 
C (M_2+1)(\min\{M_1,M_2\}+1) 
\left[\varepsilon (\varepsilon/\mu M_2 \widetilde K_2)^{M_2}  + 
(\varepsilon (M_2+1) \gamma^\prime)^{M_2+1}\right] \\
&\leq& C 
\left[\varepsilon (\varepsilon/\mu M_2 \gamma^{\prime\prime})^{M_2}  + 
(\varepsilon (M_2+1) \gamma^{\prime\prime})^{M_2+1}\right], 
\end{eqnarray*}
where, we employed the assumption 
that $\mu M_1$ and $\varepsilon/\mu M_2$ are sufficiently small and, in the second step, 
$\gamma^{\prime\prime}$ is selected appopriately. 
This leads to the desired estimate  (\ref{eq:goal-of-step-2}). 

{\em Step 2b}: The case (\ref{eq:case2-3scale-remainder}) 
is further subdivided into the cases $i\geq M_{1}+1$ and $j\geq M_{2}+1$. 

{\em Step 2b1}: In the fist case, $i \ge M_1+1$, we get 
\begin{equation*}
(i\geq M_{1}+1)\ \wedge \ (j\leq i)\ \wedge \ (i\leq j-M_{2}+M_{1})\ \wedge
\ (j-M_{2}\leq k\leq j).
\end{equation*}
We immediately see that we need 
\begin{equation*}
M_{2}\leq M_{1},
\end{equation*}
for this set of indices to be non-empty.

Writing out the sum and using convexity of the argument, we get 
\begin{eqnarray*}
T^{\prime\prime}
&:=&\sum_{i\geq M_{1}+1}\sum_{j\geq i+M_{2}-M_{1}}^{j\leq
i}\sum_{k=j-M_{2}}^{j}  f(i,j,k).
\end{eqnarray*}
As before, the innermost sum has $M_2+1$ terms and convexity yields, for the terms
of the innermost sum, the upper bound 
$\max\{f(i,j,j), f(i,j,j-M_2)\}$. More explicitly, 
\begin{eqnarray*}
f(i,j,j) &=& \mu^{i-j} \widetilde K_1^{i-j} (\widehat x \varepsilon \gamma_a)^{j} (i-j)^{i-j}, \\
f(i,j,j-M_2) &=& \mu^{i-j+M_2} \widetilde K_1^{i-j+M_2} (\widehat x \varepsilon \gamma_a)^{j-M_2} 
(\varepsilon/\mu)^{M_2} \widetilde K_2^{M_2} M_2^{M_2} 
(i-j+M_2)^{i-j+M_2} .
\end{eqnarray*}
Again, we recognize the functions to be convex in $j$, so that 
for the relevant indices $j$ we have 
\begin{eqnarray*}
f(i,j,j) &\leq& A^{\prime\prime}:= \max\{f(i,j,j)|_{j=i}, f(i,j,j)|_{j=i+M_2-M_1}\} \\
&& \leq \max\{ (\widehat x \varepsilon \gamma_a)^i, 
(\widehat x \varepsilon \gamma_a)^{i+M_2-M_1}(\mu (M_1-M_2)\widetilde K_1)^{M_1-M_2} \},
\\
f(i,j,j-M_2) &\leq& B^{\prime\prime}:= \max\{ f(i,j,j-M_2)|_{j=i}, f(i,j,j-M_2)|_{j=i+M_2-M_1} \}\\
&& \leq \max\{ (\mu M_2 \widetilde K_1)^{M_2} 
    (\widehat x \varepsilon \gamma_a)^{i-M_2}(\varepsilon/\mu M_2 \widetilde K_2)^{M_2},
(\mu M_1 \widetilde K_1)^{M_1} (\widehat x \varepsilon \gamma_a)^{i-M_1} (\varepsilon/\mu M_2 \widetilde K_2)^{M_2}
\}. 
\end{eqnarray*}
The middle sum of $T^{\prime\prime}$ has at most $M_1 - M_2+1$ terms. 
Therefore, we get with our assumption that $\mu M_1$ and $\varepsilon/\mu M_2$ are sufficiently
small and the assumption (\ref{eq:widehat_x_varepsilon_sufficiently_small}),
\begin{eqnarray*}
T^{\prime\prime} &\leq& (M_1 - M_2+1) (M_2+1) 
\Bigl[
(\widehat x \varepsilon \gamma_a)^{M_1+1} 
+ 
(\widehat x \varepsilon \gamma_a)^{M_2+1} (\mu (M_1 - M_2)\widetilde K_1)^{M_1-M_2} \\
&& \qquad \mbox{} 
+ 
(\mu M_2 \widetilde K_1)^{M_2} (\widehat x \varepsilon \gamma_a)^{M_1+1-M_2} 
(\varepsilon/\mu M_2 \widetilde K_2)^{M_2} 
+ 
(\mu M_1 \widetilde K_1)^{M_1} (\widehat x \varepsilon \gamma_a) (\varepsilon/\mu M_2 \widetilde K_2)^{M_2}
\Bigr].
\end{eqnarray*}
Since $M_1 \ge M_2$, we can estimate further with the aid of 
Lemma~\ref{lemma:elementary-properties-of-factorial},
\begin{eqnarray*}
e^{-\ua  \widehat x/4} T^{\prime\prime} 
&\leq &C (M_1-M_2+1) (M_1+1) \Bigl[ (\varepsilon (M_1+1)\gamma^\prime)^{M_1+1} \\
&& \mbox{} + 
                          (\varepsilon (M_2+1)\gamma^\prime)^{M_2+1} (\mu(M_1-M_2)\widetilde K_1)^{M_1-M_2} \\
&& \mbox{} + 
                          (\varepsilon M_2 \widetilde K_1 \widetilde K_2)^{M_2} (\varepsilon (M_1-M_2+1)\gamma^\prime)^{M_1-M_2+1}
+ 
                          \varepsilon (\mu M_1 \widetilde K_1)^{M_1} (\varepsilon/\mu M_2 \widetilde K_2)^{M_2} 
                   \Bigr]\\
&\leq& C \left[(\varepsilon (M_1+1)\gamma^\prime)^{M_1+1} + 
               \frac{\varepsilon}{\mu} (\varepsilon (M_2+1)\gamma^\prime)^{M_2}
\right], 
\end{eqnarray*}
where we used again that $\mu M_1$ and $\varepsilon/\mu M_2$ are sufficiently small.
This leads to the desired estimate  (\ref{eq:goal-of-step-2}). 

{\em Step 2b2:}
The last case is $j\geq M_{2}+1$. This reads: 
\begin{equation*}
(j\geq M_{2}+1)\ \wedge \ (j\leq i)\ \wedge \ (i\leq j-M_{2}+M_{1})\ \wedge
\ (j-M_{2}\leq k\leq j).
\end{equation*}
We see again that 
\begin{equation*}
M_{2}\leq M_{1},
\end{equation*}
is a necessary condition for the set of indices to be non-trivial.

Writing out the sum we have 
\begin{eqnarray*}
T^{\prime\prime\prime}
&:= &\sum_{j\geq M_{2}+1}\sum_{i\geq j}^{j-M_{2}+M_{1}}\sum_{k=j-M_{2}}^{j} f(i,j,k). 
\end{eqnarray*}
Again, we use convexity of $k \mapsto f(i,j,k)$ and observe the estimates  
\begin{eqnarray*}
f(i,j,j) &=& \mu^{i-j} \widetilde K_1^{i-j} (\widehat x \varepsilon \gamma_a)^{j} (i-j)^{i-j}, \\
f(i,j,j-M_2) &=& \mu^{i-j+M_2} \widetilde K_1^{i-j+M_2} (\widehat x \varepsilon \gamma_a)^{j-M_2} 
(\varepsilon/\mu)^{M_2} \widetilde K_2^{M_2} M_2^{M_2} 
(i-j+M_2)^{i-j+M_2} .
\end{eqnarray*}
Since these functions are convex functions of $i$, we can estimate for the relevant 
indices $j$: 
\begin{eqnarray*}
f(i,j,j) &\leq& A^{\prime\prime\prime}:= \max\{f(i,j,j)|_{i=j}, f(i,j,j)|_{i=j-M_2+M_1}\} \\
&& \leq (\widehat x \varepsilon \gamma_a)^j \max\{1,(\mu (M_1 -M_2)\widetilde K_1)^{M_1-M_2}\},\\
f(i,j,j-M_2) &\leq& B^{\prime\prime\prime}:= \max\{f(i,j,j-M_2)|_{i=j}, f(i,j,j-M_2)|_{i=j-M_2+M_1}\} \\
&& \leq (\widehat x \varepsilon \gamma_a)^{j-M_2} (\varepsilon/\mu M_2 \widetilde K_2)^{M_2} 
\max\{(\mu M_2 \widetilde K_1)^{M_2} , (\mu M_1 \widetilde K_1 )^{M_1}\} .
\end{eqnarray*}
Furthermore, in the triple sum $T^{\prime\prime\prime}$ the number of terms in the 
innermost sum is $M_2+1$ whereas the number 
of terms of the middle sum is bounded by $M_1 - M_2+1$. Hence, we can bound 
with geometric series arguments, in view of the assumption 
(\ref{eq:widehat_x_varepsilon_sufficiently_small}),
\begin{eqnarray*}
T^{\prime\prime\prime} &\leq& C (M_2+1)(M_1-M_2+1)
\Bigl[
(\widehat x \varepsilon \gamma_a)^{M_2+1} + (\widehat x \varepsilon \gamma_a)
(\varepsilon/\mu M_2 \widetilde K_2)^{M_2}
\max\{(\mu M_2 \widetilde K_1)^{M_2} , (\mu M_1 \widetilde K_1 )^{M_1}\}
\Bigr]. 
\end{eqnarray*}
Since $M_2 \leq M_1$, we can estimate $\mu M_2 \widetilde K_1 \leq \mu M_1 \widetilde K_1 \leq 1$, 
where we appealed in the last step to our standing assumption that 
$\mu M_1$ is sufficiently small. Hence, 
$\max\{(\mu M_2 \widetilde K_1)^{M_2} , (\mu M_1 \widetilde K_1 )^{M_1}\} \leq 1$. 
Thus, we can simplify 
\begin{eqnarray*}
T^{\prime\prime\prime} &\leq& C (M_2+1)(M_1-M_2+1)
\Bigl[
(\widehat x \varepsilon \gamma_a)^{M_2+1} + (\widehat x \varepsilon \gamma_a)
(\varepsilon/\mu M_2 \widetilde K_2)^{M_2}
\Bigr]. 
\end{eqnarray*}
Hence, with the aid of Lemma~\ref{lemma:elementary-properties-of-factorial}
\begin{eqnarray*}
e^{-\ua \widehat x/4} T^{\prime\prime\prime} 
&\leq& 
 C (M_2+1)(M_1-M_2+1)
\Bigl[
(\varepsilon (M_2+1) \gamma^\prime)^{M_2+1}  + \varepsilon (\varepsilon/\mu M_2 \widetilde K_2)^{M_2}
\Bigr]\\
&\leq& C \frac{\varepsilon}{\mu} 
\Bigl[
(\varepsilon (M_2+1) \gamma^{\prime\prime})^{M_2}  + (\varepsilon/\mu M_2 \gamma^{\prime\prime})^{M_2}
\Bigr]. 
\end{eqnarray*}
This leads to the desired estimate  (\ref{eq:goal-of-step-2}). 

{\em Step 3:} 
We now cover the case when $\widehat x \varepsilon$ is bounded away from zero. 
Specifically, let $c > 0$ be fixed and consider the case 
$\widehat x \varepsilon \ge c$. Then we have the pointwise estimate 
\begin{eqnarray*}
\left\vert \widehat{{\mathbf{U}}}_{BL}^{M}(\widehat{x})\right\vert &\leq
&\sum_{i=0}^{M_{1}}\sum_{j=0}^{M_{2}}\mu ^{i}(\varepsilon /\mu )^{j}\left(
\left\vert \widehat{u}^L_{ij}(\widehat{x})\right\vert +\left\vert \widehat{v}^L%
_{ij}(\widehat{x})\right\vert \right) \\
&\leq &\sum_{i=0}^{M_{1}}\sum_{j=0}^{M_{2}}\mu ^{i}(\varepsilon /\mu )^{j}%
{K}^{i}\overline{K}^{j}\frac{\left( \widehat{C}(i+j)+|\widehat{x}|\right)
^{2(i+j)}}{(i+j)!}e^{-\ua \widehat{x}} \\
&\leq &\sum_{i=0}^{M_{1}}\sum_{j=0}^{M_{2}}\mu ^{i}(\varepsilon /\mu )^{j}%
{K}^{i}\overline{K}^{j}\gamma ^{i+j}i^{i}j^{j}e^{-3\ua\widehat{x} /4} \\
&\leq &e^{-3\ua\widehat{x}/4}\sum_{i=0}^{M_{1}}\mu ^{i}
{K}^{i}\gamma ^{i}i^{i}\sum_{j=0}^{M_{2}}(\varepsilon /\mu )^{j}\overline{K}^{j}\gamma
^{j}j^{j} \\
&\leq &e^{-3\ua\widehat{x}/4}\sum_{i=0}^{M_{1}}\left( \mu 
\overline{K}\gamma M_{1}\right) ^{i}\sum_{j=0}^{M_{2}}\left( (\varepsilon
/\mu )K\gamma M_{2}\right) ^{j},
\end{eqnarray*}
from which the desired result follows, provided $\mu {K}\gamma M_{1}<1$ 
and $(\varepsilon /\mu )\overline{K}\gamma M_{2}<1$. Completely analogously, we 
get bounds for the derivatives of $\widehat{\mathbf U}^M_{BL}$.  

In view of the form of the differential operator $L_{\varepsilon,\mu}$ 
applied to functions of $\widehat x$ (see (\ref{eq:L-on-hat-scale})), 
we get in view of $\mu \ge \varepsilon$, 
$$
|L_{\varepsilon,\mu} \widehat{\mathbf U}^M_{BL}(\widehat x)|
\leq \left[ \varepsilon^2 + \mu^2\right] 
\varepsilon^{-2} C e^{-3 \ua \widehat x/4} 
\leq C \mu^2 \varepsilon^{-2}  e^{-3 \ua \widehat x/4} 
\leq C \varepsilon^{-2} e^{-e \ua \widehat x/4}.
$$
From the assumption $\widehat x \varepsilon \ge c$, we see
that $\varepsilon^{-2} \leq c^{-2} \widehat x^2$, and the 
factor $\widehat x$ can again be absorbed by the 
exponentially decaying $e^{-3 \ua  \widehat x/4}$. 
\end{numberedproof}
\subsection{Proof of Theorem~\ref{thm_BL_3scales(b)}}
\label{appendix:case_4_thm_BL_3scales(b)}
We first study the case
\begin{equation}
0 < \widetilde x \mu \leq 1/2.
\end{equation}

The starting point is 
the expression for $L_{\varepsilon,\mu} \widetilde{\mathbf U}^M_{BL}$ 
in (\ref{eq:LwidetildeU}). It 
consists of a 
double sum and a triple sum. We first consider the double sum. 
From the bounds on $\widetilde u_{i,j}$ of Theorem~\ref{thm11} 
and Cauchy's integral theorem
for derivatives as well as Lemma~\ref{lemma:elementary-properties-of-factorial}, we get for suitable $\gamma^\prime > 0$,
\begin{eqnarray*}
\sum_{i = 0}^{M_1} \sum_{j=M_2+1}^{M_2+2} \mu^i (\varepsilon/\mu)^j |\widetilde u_{i,j-2}^{\prime\prime}(\widetilde x)|
& \leq & C e^{-\ua \widetilde x/2} 
\sum_{i=0}^{M_1} \mu^i (i+M_2+1)\gamma^\prime)^{i+M_2} 
(\varepsilon/\mu)^{M_2+1}\\
&\leq& C e^{-\ua \widetilde x} 
(\widetilde K \varepsilon/\mu(M_2+1))^{M_2+1}, 
\end{eqnarray*}
where, in the second step we used Lemma~\ref{lemma:elementary-properties-of-factorial} again.

We next turn to the triple sum in (\ref{eq:LwidetildeU}). 
From (\ref{57}), (\ref{58}),
Lemma~\ref{lemma:elementary-properties-of-factorial}
 and $\Vert {\mathbf A}_{k}\Vert \leq C_{A}\gamma _{A}^{k},$
we obtain with $K\geq \max \{K_{1},K_{2}\}$, 
$\overline{K}\geq \max \{\overline{K}_{1},\overline{K}_{2}\}$, 
$\widetilde{C}\geq \max \{C_{\widetilde{u}},C_{\widetilde{v}}\}$:
\begin{eqnarray*}
\left\vert L_{\varepsilon,\mu }\widetilde{\mathbf{U}}_{BL}^{M}\right\vert  
&\leq &CC_{A}\sum_{i=M_1+1}^\infty\sum_{j=0}^{M_2} \sum_{k=i-M_1}^i
\mu^{i}(\varepsilon /\mu )^{j}\widetilde{x}^{k}\gamma
_{A}^{k}\overline{K}^{i-k}K^{j}\frac{\left( \widetilde{C}(i-k+j)+|\widetilde{%
x}|\right) ^{2(i-k+j)}}{(i-k+j)!}e^{-\ua\operatorname*{Re}(\widetilde{x})} \\
&\leq &CC_{A}\sum_{i=M_1+1}^\infty\sum_{j=0}^{M_2} \sum_{k=i-M_1}^i
\mu^{i}(\varepsilon /\mu )^{j}\overline{K}^{i-k}K^{j}\widetilde x^k \gamma^{i+j-k} (i+j-k)^{i+j-k}
e^{-\ua 3 \widetilde{x}/4}.  
\end{eqnarray*}
The argument is a convex function of $k$. Hence, we can bound for suitable $\gamma > 0$, 
\begin{eqnarray*}
\left\vert L_{\varepsilon,\mu }\widetilde{\mathbf{U}}_{BL}^{M}\right\vert  
&\leq &C e^{-\ua 3 \widetilde x/4} (M_1+1) \sum_{i=M_1+1}^\infty\sum_{j=0}^{M_2} 
\mu^{i}(\varepsilon /\mu )^{j} 
(K\gamma)^j 
\left[j^j \widetilde x^i + \widetilde x^{i-M_1} \overline{K}^{M_1} (j+M_1)^{j+M_1} \right].
\end{eqnarray*} 
In view of our assumption $\widetilde x \mu \leq 1/2$, 
the outer summation on $i$ leads to a convergent 
geometric series. For the inner summation, we use $(j+M_1)^{j+M_1} \leq j^j M_1^{M_1} e^{j+M_1}$,
the assumptions that 
$\mu M_1$ and $\varepsilon/\mu M_2$ are sufficiently small and convergent geometric series arguments
to get, with appropriate $\widetilde \gamma > 0$, 
\begin{eqnarray*}
\left\vert L_{\varepsilon,\mu }\widetilde{\mathbf{U}}_{BL}^{M}\right\vert  
&\leq &C e^{- \ua 3 \widetilde x/4} (M_1+1) 
\left[
(\mu \widetilde x)^{M_1+1} + (\mu M_1 \widetilde \gamma)^{M_1} \mu \widetilde x
\right], 
\end{eqnarray*}
which can be estimated in the desired fashion with the aid of 
Lemma~\ref{lemma:elementary-properties-of-factorial}: 
\begin{eqnarray*}
\left\vert L_{\varepsilon,\mu }\widetilde{\mathbf{U}}_{BL}^{M}\right\vert  
&\leq &
C e^{- \ua \widetilde x/2} (M_1+1) e^{-\ua \widetilde x/4}
\left[
(\mu \widetilde x)^{M_1+1} + (\mu M_1 \widetilde \gamma)^{M_1} \mu \widetilde x
\right] \\
& \leq &
C e^{- \ua \widetilde x/2} 
\left[
(\mu \gamma^\prime (M_1+1) )^{M_1+1} + (\mu (M_1 +1)\widetilde \gamma)^{M_1+1}
\right],  
\end{eqnarray*}
which has the appropriate form. 

We now consider the converse case $\widetilde x \mu \ge 1/2$. 
Here, we need to exploit the fact that the functions $\widetilde{\mathbf U}^M_{BL}$
are exponentially decaying. We observe that the same reasoning as in
Step~3 of the proof of Theorem~\ref{thm_BL_3scales} yields 
\begin{eqnarray*}
\left| \widetilde{\mathbf U}^M_{BL}(\widetilde x)\right| &\leq& 
C e^{-3 \ua \widetilde x/4},
\end{eqnarray*}
and, by Cauchy's integral theorem for derivatives, estimates for the derivatives. 
The arguments of Step~3 of the proof of Theorem~\ref{thm_BL_3scales} therefore, yield 
\begin{eqnarray*}
\left|L_{\varepsilon,\mu} \widetilde{\mathbf U}^M_{BL}(\widetilde x)\right|
\leq C \left[ \varepsilon^2 + \mu^2\right] \mu^{-2} e^{-3\ua \widetilde x/4}. 
\leq C  e^{-3\ua \widetilde x/4}. 
\end{eqnarray*}

\end{document}